\documentclass[final, reqno]{amsart} 
\usepackage{mathtools, amsthm, amssymb}
\usepackage[margin=1in]{geometry}
\usepackage{mathrsfs}
\usepackage{tikz-cd}

\usepackage{hyperref}

\usepackage{thmtools}
\makeatletter
\IfFormatAtLeastTF{2024-11-01}{
  \renewcommand*\@addtoreset[2]{%
    \bgroup
      \edef\aliasctr@@truelist{\aliasctr@follow{#2}}%
      \let\@elt\relax
      \expandafter\@cons\aliasctr@@truelist{{#1}}%
    \egroup
    \expandafter\xdef\csname theH#1\endcsname{%
      \expandafter\noexpand\csname theH#2\endcsname.%
      \noexpand\the\noexpand\value{#1}}%
  }
}{}
\makeatother

\usepackage[capitalise]{cleveref}

\usepackage{microtype}

\newcommand{\resp}{{\sfcode`\.1000 resp.}}
\newcommand{\ie}{{\sfcode`\.1000 i.e.}}
\newcommand{\eg}{{\sfcode`\.1000 e.g.}}
\newcommand{\cf}{{\sfcode`\.1000 cf.}}

\DeclareFontFamily{U}{min}{}
\DeclareFontShape{U}{min}{m}{n}{<-> udmj30}{}
\newcommand\yo{\!\text{\usefont{U}{min}{m}{n}\symbol{'210}}\!}

\DeclareMathOperator\Ind{Ind}

\DeclareMathOperator\Res{Res}

\newcommand\pt{\mathrm{pt}}

\DeclareMathOperator\SH{\mathbf{SH}}

\DeclareMathOperator\Fun{Fun}

\DeclareMathOperator\Op{Op}

\DeclareMathOperator\cofib{cofib}

\DeclareMathOperator\Spec{Spec}

\DeclareMathOperator\Psh{Psh}
\DeclareMathOperator\Shv{Shv}

\DeclareMathOperator\aff{\mathbb{A}}
\DeclareMathOperator\proj{\mathbb{P}}

\newtheorem*{thm*}{Theorem}
\newtheorem{thm}{Theorem}[section]
\Crefname{thm}{Theorem}{Theorems}
\newtheorem{thmX}{Theorem}
\Crefname{thmX}{Theorem}{Theorems}
\newtheorem{prp}[thm]{Proposition}
\Crefname{prp}{Proposition}{Propositions}
\newtheorem{lem}[thm]{Lemma}
\Crefname{lem}{Lemma}{Lemmas}
\newtheorem{cor}[thm]{Corollary}
\Crefname{cor}{Corollary}{Corrolaries}

\theoremstyle{remark}
\newtheorem{rmk}[thm]{Remark}
\Crefname{rmk}{Remark}{Remarks}

\theoremstyle{definition}
\newtheorem{defn}[thm]{Definition}
\Crefname{defn}{Definition}{Definitions}
\newtheorem{exa}[thm]{Example}
\Crefname{exa}{Example}{Examples}

\Crefname{warn}{Warning}{Warnings}
\newtheorem{nota}[thm]{Notation}
\Crefname{nota}{Notation}{Notations}

\Crefname{cnstr}{Construction}{Constructions}

\Crefname{setng}{Setting}{Settings}
\newtheorem{ass}[thm]{Assumption}
\Crefname{ass}{Assumption}{Assumptions}

\newcommand\id{\mathrm{id}}

\newcommand\comps{\mathbb{C}}
\newcommand\reals{\mathbb{R}}

\newcommand\fin{\mathrm{fin}} 
\newcommand\prpr{\mathrm{prpr}} 

\newcommand\lred{\mathrm{lred}}
\newcommand\nice{\mathrm{nice}}

\newcommand\open{\mathrm{open}}

\newcommand\Nis{\mathrm{Nis}}

\newcommand\Sm{\mathrm{Sm}}

\DeclareMathOperator\B{\mathbf{B}\!}

\newcommand\an{\mathrm{an}}
\newcommand\hol{\mathrm{hol}}
\newcommand\alg{\mathrm{alg}}
\newcommand\diff{\mathrm{diff}}

\newcommand\htpy{\mathrm{htpy}}

\newcommand\op{\mathrm{op}}
\newcommand\univ{\mathrm{univ}}
\newcommand\red{\mathrm{red}}

\DeclareMathOperator\Coh{Coh}

\DeclareMathOperator\PF{PF}

\newcommand\Sch{\mathrm{Sch}}
\newcommand\Mfld{\mathrm{Mfld}}
\newcommand\DiffStk{\mathrm{DiffStk}}
\newcommand\HolStk{\mathrm{HolStk}}
\newcommand\AlgStk{\mathrm{AlgStk}}
\newcommand\An{\mathrm{An}}

\newcommand\spaces{\mathcal{S}}

\newcommand\Cat{\mathbf{Cat}}
\newcommand\PrL{{\mathbf{Pr}^{\mathbf{L}}}}
\newcommand\PrR{{\mathbf{Pr}^{\mathbf{R}}}}

\newcommand\CAlg{\mathrm{CAlg}}

\DeclareMathOperator\Pic{Pic}

\DeclareMathOperator\Sym{Sym}

\DeclarePairedDelimiter\cls\lbrack\rbrack

\newcommand\stackslink[1]{\href{https://stacks.math.columbia.edu/tag/#1}{#1}}
\newcommand\stackscite[1]{\cite[Tag \stackslink{#1}]{stacks-project}}

\title[The Localization Theorem for Motivic Homotopy Theories]{The Localization Theorem for the Motivic Homotopy Theory of Complex Analytic Stacks and other Geometric Settings}
\author{Roy Magen}
\address{Institute of Mathematics and Informatics, Bulgarian Academy of Sciences \\ Bulgaria, Sofia 1113, Acad. G. Bonchev Str., Bl. 8}
\thanks{Supported by the Bulgarian Ministry of Education and Science, Scientific Programme "Enhancing the Research Capacity in Mathematical Sciences (PIKOM)", No. DO1-67/05.05.2022}

\begin{document}

\begin{abstract}
	We prove the analog of the Morel-Voevodsky localization theorem over complex analytic stacks, which is used in \cite{6FF} to establish a 6-functor formalism of complex analytic motivic homotopy theory and produce an analytification map that is compatible with the six operations. Along the way, we establish general techniques for proving this theorem over other geometric settings, which also apply, for example, to the settings of algebraic stacks and differentiable stacks.
\end{abstract}
\maketitle

\setcounter{tocdepth}{3}
\tableofcontents

\section{Introduction}

The Morel-Voevodsky localization theorem is a crucial result in motivic homotopy theory that underlies many important properties such as cdh descent and Grothendieck's six operations. This result was first shown for the motivic homotopy theory of schemes in \cite[Theorem 3.2.21]{A1htpysch}, which can be formulated as follows:
\begin{thm*}[Morel-Voevodsky localization theorem]
	Let $i : Z \to S$ be a closed immersion of Noetherian schemes with open complement $j : U \to S$. Then
	\[
		j_! j^* F \to F \to i_* i^* F
	\]
	is an exact triangle in $\SH(S)$ for any $F \in \SH(S)$, where $\SH(S)$ is the stable motivic homotopy category over $S$.

	In fact, we have that
	\[
		\SH(Z) \xrightarrow{i_*} \SH(S) \xrightarrow{j^*} \SH(U)
	\]
	is a fibre sequence of triangulated categories.\footnote{It follows that $\SH(Z), \SH(U)$ can be seen as a recollement of $\SH(S)$ in the sense of \cite[Definition A.8.1]{ha}.}
\end{thm*}

This localization property of $\SH$ is incredibly important, and the problem of establishing analogs of it in various settings remains a topic of ongoing study. Before discussing the various examples, let us mention our main application, which follows from \Cref{thm:hol gluing}:
\begin{thmX}[Complex analytic localization theorem] \label{thmX:main}
	Let $i : Z \to S$ be an embedding of suitable\footnotemark{} complex analytic stacks with open complement $j : U \to S$. Then
	\[
		j_! j^* F \to F \to i_* i^* F
	\]
	is an exact triangle in $\SH^\hol(S)$ for any $F \in \SH^\hol(S)$, where $\SH^\hol(S)$ is the stable motivic homotopy category over $S$.

	In fact, we have that
	\[
		\SH^\hol(Z) \xrightarrow{i_*} \SH^\hol(S) \xrightarrow{j^*} \SH^\hol(U)
	\]
	is a fibre sequence of triangulated categories.
	\footnotetext{For example, it suffices to assume that $S$ has an open cover by global quotients of the form $[X/G]$, where $G$ is a finite group acting on a complex space $X$.}
\end{thmX}

This result is used in \cite{6FF} to show that $\SH^\hol$ admits the structure of a 6-functor formalism and establish cdh descent for it, as well as to show that the natural realisation maps $\SH(X) \to \SH^\hol(X^\an)$ are compatible with the six operations.

Let us now mention some examples of the localization property in algebraic geometry:
\begin{enumerate}

	\item The localization property is one of the axioms listed in \cite{voe-crit} to be used in the construction of a 6-functor formalism, and indeed, this was used crucially by Ayoub in his thesis \cite{Ayoub6I,Ayoub6II} to establish the formalism of Grothendieck's six operations for the stable motivic homotopy theory of quasi-projective schemes.

	\item The localization property was studied extensively in \cite{tri-cat-mixed-motives,Deglise_2007}, where the authors were also particularly interested in the setting of sheaves with transfers. The situation in this case is more delicate, and the problem of showing the result in general for Voevodsky motives remains open. Nevertheless, the results proven in that paper allowed the authors to extend Ayoub's construction of the six operations to schemes that are not necessarily quasi-projective. The localization property is also used to establish cdh descent in this setting.

	\item In \cite{sixopsequiv}, Hoyois established the six operations for an \emph{equivariant} version of $\SH$ using an equivariant version of the localization theorem also established in that paper. The localization property also allowed Hoyois to establish cdh descent for this 6-functor formalism, and deduce that equivariant homotopy $K$-theory satisfies cdh descent in \cite{cdh-equiv-k}. The localization property in the equivariant setting was also used in \cite{SixAlgSt} to deduce the localization property for stacks and thence obtain a 6-functor formalism of stable motivic homotopy theory for derived algebraic stacks that satisfies cdh descent.

	\item In \cite{locspalg}, Khan showed the localization theorem for spectral algebraic spaces, allowing him to prove a version of derived invariance of small \'{e}tale sites in the setting of $\aff^1$-invariant sheaves on lisse-Nisnevich sites. This was also used in \cite{SixAlgSp} to obtain a 6-functor formalism of stable motivic homotopy theory for derived algebraic spaces, and establish cdh descent for it.

	\item In \cite{FrLoc}, Hoyois established the localization property for schemes and \emph{framed correspondences}. This can be seen as a variation of the case of finite correspondences, where the result is still open in general. In this setting, the localization property then allowed Hoyois to show that the theory of \emph{framed} stable motivic homotopy theory coincides with the ordinary one.
	
	\item The localization property is also relevant to the study of mixed Hodge modules -- see \cite{MotivicHodge,mhm,ruimy2026norimotivesandmixed,Hodge_realis}.

\end{enumerate}

The localization property is also a topic of study outside of algebraic geometry. Apart from \Cref{thmX:main} and its applications in \cite{6FF}, we have the following examples:
\begin{enumerate}

	\item Cnossen establishes the localization property in the setting of \emph{differentiable} stacks in \cite{TwAmb}, allowing him to establish a 6-functor formalism of stable motivic homotopy theory on differentiable stacks.

	\item Volpe establishes the localization property for sheaves on topological spaces in \cite{top6FF}, allowing him to establish a 6-functor formalism for sheaves on topological spaces. This was further studied in \cite{cts6FF,cond6FF}. 

	\item The role of the localization property in establishing cdh descent and constructing 6-functor formalisms for general geometric settings was studied in \cite{6FF}.

\end{enumerate}

Due to the great interest in establishing the localization property in diverse settings, we have decided to prove \Cref{thmX:main} by first producing some general techniques for establishing gluing that do not rely on the specific geometric setting. Indeed, the proof of the localization property usually follows the following series of reductions:
\begin{enumerate}

	\item Before establishing the localization property for $\SH$, one establishes the so-called \emph{gluing property} for the unstable version $H$ of $\SH$. It is then possible to deduce the localization property for $\SH$ using \Cref{lem:gluing for ptd,prp:stabilize gluing}.

	\item In order to show that $H$ has the gluing property for $i : Z \to S$, one first shows that for certain ``smooth'' maps $X \to S$, and maps $t : Z \to X$ over $S$, the presheaf $\hat\theta_i(X,t)$ is ``contractible'' in an appropriate sense. One then deduces that $H$ has the gluing property using \Cref{thm:global pseudotop gluing}, or other results from \Cref{S:red to fib gluing}.

	\item When $X \to S$ is actually some sort of ``vector bundle'', so that it admits a section $0 : S \to X$ that is actually a ``homotopy inverse'', it is easy to show that $\hat\theta_i(X, 0 \circ i)$ is contractible. We must use the properties of our specific geometric setting to show that it is possible to relate a sufficiently general $X \to S$ (and map $t : Z \to X$) to this case via some zigzag of ``\'{e}tale neighbourhoods''. \Cref{thm:gluing is adm invar} then allows us to deduce that $\hat\theta_i(X,t)$ is contractible.

\end{enumerate}
This strategy is discussed in more detail in \Cref{S:examples}. We have endeavored to make our results general enough so that they should be useful for establishing the localization property for many different types of geometric settings. Indeed, apart from proving \Cref{thmX:main} in \Cref{thm:hol gluing}, as a demonstration of the wider applicability of our techniques, we also re-establish the localization property in the settings of algebraic and differentiable stacks in \Cref{thm:alg gluing,thm:diff gluing}.

\subsection{Outline}

First let us briefly discuss some of the dependencies between various parts of this paper:
\begin{itemize}

	\item The language of pullback formalisms from \Cref{S:pbf} is useful for many of the results in this paper, but not strictly necessary for the main results in \Cref{S:main gluing}. The notion of LC pullback formalisms from \Cref{S:LCPBF} is only used in the proofs of the results of \Cref{S:red to fib gluing}. The basic notions of pseudotopologies from \Cref{S:pseudotop} are useful for understanding the general setting for constructing unstable motivic homotopy theories, particularly \Cref{prp:univ invar}.

	\item In \Cref{S:basic gluing}, we introduce the abstract notions needed for studying the localization property. The results we study in this section apply in the setting of general pullback formalisms. The results of \Cref{S:stabilizing}, namely \Cref{lem:gluing for ptd,prp:stabilize gluing} are routinely used in applications to deduce the localization property for ``stable motivic homotopy theories'' from the analogous property for the unstabilized versions. Some of the results of \Cref{S:red}, such as \Cref{cor:gluing locality globally on base}, are also sometimes useful in applications, but the main result, \Cref{prp:intrinsic gens gluing}, is mostly used in the proofs of \Cref{S:red to fib gluing}.

	\item In \Cref{S:main gluing}, we study more specialized techniques for establishing the localization property for ``unstable motivic homotopy theories''. In \Cref{S:red to fib gluing}, we describe the reduction to the study of presheaves $\hat \theta_i(X,t)$, and in \Cref{thm:gluing is adm invar}, we describe the key excision result enjoyed by these presheaves. The proof of \Cref{thm:gluing is adm invar} also depends on \Cref{S:excision}.

	\item Finally, in \Cref{S:examples}, we demonstrate how to apply our general results to the study of motivic homotopy theories for differentiable, algebraic, and complex analytic stacks.

\end{itemize}

Our goal is to study the localization property for general notions of motivic homotopy theory. These are generally constructed as follows. We are given the following data:
\begin{itemize}

	\item an $\infty$-category $\mathcal C$ of geometric objects (\eg{} schemes),

	\item a notion of ``quasi-admissible maps'' (\eg{} smooth maps) in $\mathcal C$,

	\item a notion $\tau$ of ``local homotopy equivalence'' (\eg{} the union of $\aff^1$-homotopy equivalence with Nisnevich local equivalence),

	\item for each $S \in \mathcal C$, some collection $\mathcal A_S$ of pointed objects of $\mathcal C_{/S}$ (\eg{} $(\proj^1_S, \{\infty\} \times S)$).

\end{itemize}
We then construct a presheaf of symmetric monoidal presentable $\infty$-categories on $\mathcal C$ by the following procedure:
\begin{enumerate}

	\item We first consider the presheaf given by $S \mapsto \mathcal C_S$, where $\mathcal C_S \subseteq \mathcal C_{/S}$ is the full subcategory given by the quasi-admissible maps to $S$. The functoriality is given by base change.

	\item This induces a presheaf of presentable $\infty$-categories $S \mapsto \Psh(\mathcal C_S)$. This even takes values in symmetric monoidal presentable $\infty$-categories via the Cartesian monoidal structures. We denote this presheaf by $H^\univ$.

	\item This then induces the presheaf $S \mapsto \Psh^\tau(\mathcal C_S)$, where $\Psh^\tau$ refers to $\tau$-invariant presheaves, \ie{} those presheaves respecting the notion of local equivalence encoded by $\tau$ (\eg{} $\aff^1$-invariant Nisnevich sheaves). This presheaf is denoted by $H^\tau$.

	\item Next, we consider the presheaf $H^\tau_\bullet$, where $H^\tau_\bullet(S)$ is the $\infty$-category of pointed objects in $H^\tau(S)$, equipped with the symmetric monoidal category given by the smash product $\wedge$.

	\item Finally, we formally adjoin $\wedge$-inverses of objects coming from the collections $\mathcal A_S$ for $S \in \mathcal C$. We will denote this presheaf by $\SH^\tau$.

\end{enumerate}

If $f : X \to Y$ is a quasi-admissible map, then taking slice projections induces a functor $\mathcal C_X \to \mathcal C_Y$, and it is easy to see that this induces a functor $f_\sharp : H^\univ(X) \to H^\univ(Y)$ that is left adjoint to the functor $f^* = H^\univ(f)$ induced by base change along $f$. In fact, we also have the following properties:
\begin{description}

	\item[Projection formula] the map $f_\sharp(- \otimes f^*) \to f_\sharp \otimes -$ is an equivalence.

	\item[Base change] for any map $q : Y' \to Y$, the map $f'_\sharp p^* \to q^* f_\sharp$ is an equivalence, where $f' = f \times_Y Y'$ is the base change of $f$ along $q$, and $p = X \times_Y q$ is the base change of $q$ along $f$.

\end{description}

Under some hypotheses on $\tau$ and $\{\mathcal A_S\}_{S \in \mathcal C}$ which are studied in \cite{Fundamentals,UnivFF}, the presheaves $H^\tau$ and $\SH^\tau$ also enjoy these properties. The result concerning $H^\tau$ is recalled in \Cref{prp:univ invar}. In general, presheaves satisfying these conditions are called \textbf{pullback formalisms}, and we go over some of their properties in \Cref{S:pbf}.

The localization property is studied at the generality of pullback formalisms in \Cref{S:basic gluing}, but then in \Cref{S:main gluing}, we consider more specialized tools for pullback formalisms of the form $H^\tau$. The tools for deducing the localization property for $\SH^\tau$ from $H^\tau$ are given in \Cref{S:stabilizing}. Finally, we apply these general results in specific geometric settings in \Cref{S:examples}. A more specialized overview for how this is generally done is given at the beginning of that \lcnamecref{S:examples}.

In fact, in order to express the localization property in a way that is sensible for more general pullback formalisms, including $H^\tau$, we consider the condition that for a particular ``closed'' map $i : Z \to S$ with ``open'' (\ie{} quasi-admissible) complement $j : U \to S$, the \textbf{gluing square}
\[
	\begin{tikzcd}
		j_\sharp j^* \ar[d] \ar[r] & \id \ar[d] \\
		j_\sharp j^* i_* i^* \ar[r] & i_* i^*
	\end{tikzcd}
\]
is coCartesian. This is referred to as \textbf{gluing}.

Often for $T \in \mathcal C$, and $Y \in \mathcal C_T$, we write $\cls{Y;T} \in H^\tau(T) = \Psh^\tau(\mathcal C_T)$ for the $\tau$-invariant presheaf determined by $Y$ (the ``$\tau$-sheafification'' of the presheaf represented by $Y$), and we can even make sense of $\cls{Y}$ for a general pullback formalism $D$ by setting $\cls{Y;T} \coloneqq (Y \to T)_\sharp(1)$, where $1 \in D(Y)$ is the monoidal unit. For $X \in \mathcal C_S$, the gluing square above evaluates to the following square in $H^\tau(S)$:
\begin{equation} \label{eqn:eval gluing square}
	\begin{tikzcd}
		\cls{X \times_S U;S} \ar[d] \ar[r] & \cls{X;S} \ar[d] \\
		\cls{U;S} \ar[r] & i_* \cls{X \times_S Z;Z}
	\end{tikzcd}
.\end{equation}

The consequences of gluing for general pullback formalisms were studied in \cite{6FF}, for algebro-geometric settings in \cite{tri-cat-mixed-motives,MotivicHodge,SixAlgSt,sixopsequiv,Ayoub6I,voe-crit,FrLoc,A1htpysch}, and in a differentiable setting in \cite{TwAmb}. In the present work we focus on the techniques needed to prove this property.

The gluing property is reviewed in \Cref{S:basic gluing}, where we also establish some general techniques that are useful for proving it. \Cref{lem:gluing locality on base} says that we can work locally on $S$, and under somewhat more specialized hypotheses, \Cref{cor:gluing locality globally on base} says that it suffices to show that \eqref{eqn:eval gluing square} is coCartesian, where $X \to S$ can be chosen among some ``generating''  quasi-admissible maps to $S$.

The techniques for deducing the gluing property for the stabilized version $\SH^\tau$ are given in \Cref{S:stabilizing}. First we use \Cref{lem:gluing for ptd} to show that $H^\tau_\bullet$ has gluing, and then we apply \Cref{prp:stabilize gluing} to deduce it for $\SH^\tau$, where the key hypothesis is that $i$ ``has $\mathcal A$-lifts'', which roughly says that the elements of $\mathcal A_Z$ are ``generated'' by the objects of the form $X \times_S Z$ for $X \in \mathcal A_S$.

The results of \Cref{S:basic gluing} actually hold in the context of quite general pullback formalisms, but in \Cref{S:main gluing}, we consider more specialized techniques that hold in the specific case of $H^\tau$. The results of \Cref{S:red to fib gluing} show that the problem of showing gluing for $H^\tau$ reduces to the problem of showing that certain presheaves $\hat \theta_i(X,t)$ are ``$\tau$-locally contractible''. These presheaves $\hat \theta_i(X,t)$ should be thought of as a ``local'' or ``fibrewise'' version of the gluing square: given $X \in \mathcal C_{/S}$, a map $t : Z \to X$ corresponds to a section of $X \times_S Z \to Z$, \ie{} a global section $\pt \to \yo_{\mathcal C_{/Z}}(X \times_S Z)$ in $\Psh(\mathcal C_{/Z})$. This corresponds to a global section $\pt \to i_* \yo_{\mathcal C_{/Z}}(X \times_S Z)$ in $\Psh(\mathcal C_{/S})$. We then define $\hat \theta_i(X,t)$ to be given by the base change along $\pt \to i_* \yo_{\mathcal C_{/Z}}(X \times_S Z)$ of the coCartesian gap of the square
\[
	\begin{tikzcd}
		\yo_{\mathcal C_{/S}}(X \times_S U) \ar[d] \ar[r] & \yo_{\mathcal C_{/S}}(X) \ar[d] \\
		\yo_{\mathcal C_{/S}}(U) \ar[r] & i_* \yo_{\mathcal C_{/Z}}(X \times_S Z)
	\end{tikzcd}
,\]
so $\hat \theta_i(X,t)$ can be described as follows:
\[
	\hat \theta_i(X,t) = \pt \times_{i_*(X \times_S Z)} (U \coprod_{X \times_S U} X) \in \Psh(\mathcal C_{/S})
.\]

Although this may seem like a complicated definition, in certain cases, it is easy to show that this becomes ``$\tau$-locally contractible'' by means of an explicit nullhomotopy. The example to keep in mind is the case of $\hat \theta_i(V, 0 \circ i)$, where $V \to S$ is some vector bundle, and $0 : S \to V$ is the zero section. Also see \cite[Lemma 4.17]{sixopsequiv} and \cite[Lemma II.5.2.11]{TwAmb} for specific examples.

Furthermore, \Cref{thm:gluing is adm invar} establishes a key tool for studying the presheaves $\hat \theta_i(X,t)$: it roughly says that, up to $\tau$-local equivalence, these presheaves are invariant under taking ``\'{e}tale neighbourhoods'' of $t$ in $X$. Thus, introducing the datum of the map $t$ allows us to work locally on $X$ in a way that was not available to us when working the plain gluing square. Combined with the observation of the previous paragraph, this gives us a good supply of $i,X,t$ for which $\hat \theta_i(X,t)$ is $\tau$-locally contractible, which we can then use to show that $H^\tau$ has gluing using the results of \Cref{S:red to fib gluing}. An overview of the techniques needed to carry out this strategy is given in the beginning of \Cref{S:examples}, and of course the full details are given in the various examples considered in that \lcnamecref{S:examples}.

\subsubsection*{Filling in a subtle detail}
It is worth mentioning that the proof of the localization property is often somewhat delicate, and in fact, some subtle errors have been found in the proofs appearing in the literature. Already in the original localization theorem of \cite[Theorem 3.2.21]{A1htpysch}, there is a mistake that has even been repeated several times in the literature when proving other versions of the localization property. This mistake involves using a property that holds for presheaves on categories $\Sch_S$ of schemes over a base scheme $S$, but does not hold if we instead consider presheaves on categories $\Sm_S$ of schemes \emph{smooth} over $S$. See \Cref{rmk:the mistake} and \cite[Remark 3.24]{SixAlgSt} for more details.

This error is fixable, and has been corrected in most cases and avoided in more recent proofs such as that of \cite[Corollary 5]{FrLoc} and \cite[Theorem 3.15]{SixAlgSt}, simply by working with presheaves over these larger categories (analogous to $\Sch_S$ instead of $\Sm_S$). This should lead to a statement that is at least as strong, but in principle one should ask how to deduce the statement over $\Sm_S$ from the statement over $\Sch_S$. Indeed, the arguments proceed by showing that certain maps are $\tau$-local equivalences in these presheaf categories, but in general, it may be possible to have a $\tau$-local equivalence between presheaves on $\Sch_S$ that does not restrict to a $\tau$-local equivalence of presheaves on $\Sm_S$. This could happen if some smooth scheme over $S$ were $\tau$-locally equivalent to a presheaf on $\Sch_S$ that does not come from a presheaf on $\Sm_S$. This cannot happen in the cases of interest because the notion of $\tau$-local equivalence given by $\tau$ is generated by Nisnevich covering sieves and $\aff^1$-homotopy equivalence, both of which are given in some sense by \emph{smooth} maps.

The general context of our results demands that we be more explicit about this deduction, and therefore also gives us the opportunity to fill in this detail for the existing proofs of the localization property.
Indeed, the reductions given in \Cref{S:red to fib gluing} make precise the above subtle point about moving between presheaves on $\Sm_S$ and $\Sch_S$ by using \Cref{lem:res local equiv}. \Cref{exa:counter} shows that this result does not hold in complete generality, \eg{} if we allowed for cdh covers or constructible covers as well as Nisnevich covers, or if instead of presheaves on $\Sm_S$, we were interested in presheaves on the category of schemes \emph{proper} over $S$.

\subsection{Acknowledgements}

The author would like to thank Andrew Blumberg and Johan de Jong for their support as PhD advisors, and Elden Elmanto for his encouragement and advice.

\subsection{Notations and Conventions}

Throughout this article, we will make heavy use of the machinery of $\infty$-categories as developed in \cite{htt} and \cite{ha}. Therefore, all of our language will be implicitly $\infty$-categorical:
\begin{enumerate}

	\item We say ``category'' to mean ``$\infty$-category''. Note that then functors, adjoints, and (co)limits must all be understood in the context of $\infty$-categories.

	\item Following \cite[Remark 3.0.0.5]{htt}, we will write $\Cat$ to denote the category of small categories, and $\widehat{\Cat}$ to denote the category of all categories.

	\item We write $\spaces$ for the category of small spaces/$\infty$-groupoids/anima (see \cite[\S1.2.16]{htt}).

	\item Unless otherwise specified, presheaves and sheaves are always implicitly assumed to take values in $\spaces$. Given a category $\mathcal C$, we write $\Psh(\mathcal C)$ to denote that category of presheaves on $\mathcal C$, and if $\mathcal C$ is equipped with a Grothendieck topology that is understood from context, we write $\Shv(\mathcal C)$ to denote the category of sheaves on $\mathcal C$.

	\item Given a category $\mathcal C$, we write $\mathcal C(-,-)$ for the hom functor $\mathcal C^\op \times \mathcal C \to \widehat{\spaces}$. $\mathcal C$ is locally small if this functor takes values in $\spaces$.

	\item The very large categories $\PrL, \PrR$ of presentable categories are defined in \cite[Definition 5.5.3.1]{htt}. These are the categories of presentable categories and left adjoint functors or right adjoint functors respectively. Note that $\PrL$ is equipped with the structure of a symmetric monoidal category as in \cite[Proposition 4.8.1.15]{ha}.

	\item For any symmetric monoidal category $\mathcal C$, we write $\CAlg(\mathcal C)$ for the category of commutative algebra objects in $\mathcal C$ -- see \cite[Definition 2.1.3.1]{ha}. In particular, $\CAlg(\Cat)$ is the category of symmetric monoidal categories, and $\CAlg(\PrL)$ is the category of symmetric monoidal presentable categories where the monoidal product preserves small colimits in each variable.

	\item A ``zero object'' of a category $\mathcal C$ is an object that is both initial and terminal. We will say that a category is pointed if it has a zero object. See \cite[Definition 1.1.1.1]{ha}.

\end{enumerate}

We will also use the following notations and conventions:
\begin{enumerate}

	\item All quotients by group actions are assumed to be stacky. Often these quotients are denoted by $[X/G]$, but we will instead simply write $X/G$.

	\item The abbreviation ``qcqs'' will be used to mean ``quasi-compact and quasi-separated'' in any context where these adjectives make sense.

	\item Whenever we say ``limit-preserving'' or ``colimit-preserving'', we are referring only to \emph{small} limits and colimits.

	\item If $f : X \to Y$ and $g : X' \to Y$ are maps in a category $\mathcal C$, and the fibred product $X \times_Y X'$ exists, we will sometimes write $f^{-1}(g) : X \times_Y X' \to X$ for the base change of $g$ along $f$ in $\mathcal C$.

	\item Given some implicitly specified ambient category, we will write $\pt$ for a terminal object of that category.

	\item Given some implicitly specified ambient monoidal category, we will write $1$ for a monoidal unit of that category.

	\item Given a locally small category $\mathcal C$, we will write $\yo : \mathcal C \to \Psh(\mathcal C)$ for the Yoneda embedding of $\mathcal C$. We generally do not include $\mathcal C$ in the notation as it is often clear from context which category's Yoneda embedding we are considering.

	\item Following \cite[Notation 1.2.8.4]{htt}, for any simplicial set $K$, we write $K^\triangleright$ for the simplicial set obtained by adjoining a terminal cone point to $K$. We will also write $\infty$ to denote the cone point of $K^\triangleright$.

\end{enumerate}

\section{Preliminaries} \label{S:prelims}

\subsection{Pullback Formalisms} \label{S:pbf}

We recall some basic notions about pullback contexts and pullback formalisms from \cite[\S1]{Fundamentals}. This is also summarized in greater detail in \cite[\S2.2]{6FF}.
\begin{defn}
	A \emph{pullback context} is a category $\mathcal C$ equipped with a collection of maps called \emph{quasi-admissible maps}, which is stable under composition and base change, and which contains all equivalences.

	For any $S \in \mathcal C$, we write $\mathcal C_S$ to denote the full subcategory of $\mathcal C_{/S}$ consisting of quasi-admissible maps to $S$.

	Say $\mathcal C$ is \emph{quasi-small} if $\mathcal C_S$ is equivalent to a small category for all $S \in \mathcal C$.
\end{defn}

\begin{defn} \label{defn:anodyne subctx}
	Given a pullback context $\mathcal C$, we say that a full subcategory $\mathcal C' \subseteq \mathcal C$ is a \emph{(full) anodyne pullback subcontext} if every quasi-admissible map to an object of $\mathcal C'$ has domain given by an object of $\mathcal C'$. The category $\mathcal C'$ then inherits the structure of a pullback context in which a map is quasi-admissible exactly when it is quasi-admissible in $\mathcal C$.
\end{defn}

\begin{rmk} \label{rmk:anodyne subctx}
	As in \cite[Remark 1.2.7]{Fundamentals}, we have that \cite[Definition 1.2.5]{Fundamentals} specializes to \Cref{defn:anodyne subctx} in the case of a full subcategory inclusion, so in the setting of \Cref{defn:anodyne subctx}, the inclusion preserves base changes along quasi-admissible maps.
\end{rmk}

\begin{defn}
	A \emph{pullback formalism} on a pullback context $\mathcal C$ is a presheaf $D : \mathcal C^\op \to \CAlg(\PrL)$ such that for any quasi-admissible map $f : X \to Y$,
	\begin{enumerate}

		\item the functor $f^* \coloneqq D(f)$ admits a left adjoint $f_\sharp$,

		\item $D$ satisfies the ``quasi-admissible projection formula'': for any $M \in D(X)$ and $N \in D(Y)$, the map
			\[
				f_\sharp(M \otimes f^* N) \to f_\sharp M \otimes N
			\]
			is an equivalence, and

		\item $D$ satisfies ``quasi-admissible base change'': if
			\[
				\begin{tikzcd}
					X' \ar[d, "p"'] \ar[r, "f'"] & Y' \ar[d, "q"] \\
					X \ar[r, "f"'] & Y
				\end{tikzcd}
			\]
			is a Cartesian square in $\mathcal C$, then the natural map
			\[
				f'_\sharp p^* \to q^* f_\sharp
			\]
			is an equivalence.

	\end{enumerate}
	If $D'$ is also a pullback formalism on $\mathcal C$, then a transformation $\phi : D \to D'$ is a \emph{morphism of pullback formalisms} if for any quasi-admissible $f$, the natural map
	\[
		f_\sharp \phi \to \phi f_\sharp
	\]
	is an equivalence.

	We write $\PF(\mathcal C)$ to denote the \emph{category of pullback formalisms}, which is a subcategory of $\Fun(\mathcal C^\op, \CAlg(\PrL))$.
\end{defn}

\begin{exa} \label{exa:Huniv}
	Given a quasi-small pullback context $\mathcal C$, the initial pullback formalism of \cite[Theorem 1.4.3]{Fundamentals} or \cite[Corollary 4.9]{UnivFF}, $H^\univ$, sends any $S \in \mathcal C$ to the presheaf category $\Psh(\mathcal C_S)$ equipped with the Cartesian monoidal structure. We will see variations of this in \Cref{exa:completed Huniv is LC}, \Cref{prp:univ invar}, and \Cref{exa:completed invar PF}.
\end{exa}

\begin{nota}
	If $D$ is a pullback formalism on a pullback context $\mathcal C$, then for any quasi-admissible map $X \to S$ in $\mathcal C$, we write
	\[
		\cls{X} = \cls{X;S} = \cls{X;S}_D \coloneqq (X \to S)_\sharp(1)
	.\]
\end{nota}

\begin{defn}
	If $D$ is a pullback formalism on a pullback context $\mathcal C$, we say that $D$ is \emph{geometrically generated} if for any object $S \in \mathcal C$, $D(S)$ is generated under small colimits by objects of the form $[X]$ for $X \in \mathcal C_S$.
\end{defn}

The following notions were are studied in \cite[\S2.3]{Fundamentals}.
\begin{defn}
	If $\mathcal C$ is a category, and $D : \mathcal C^\op \to \widehat{\Cat}$ is a presheaf, then we say that a map $f$ in $\mathcal C$ is \emph{$D$-acyclic} if $f^* = D(f)$ is fully faithful. More generally, a diagram $X : K^\triangleright \to \mathcal C$ is $D$-acyclic if
	\[
		D(X(\infty)) \to \varprojlim_{a \in K} D(X(a))
	\]
	is fully faithful.
\end{defn}

\subsubsection{Pseudocovers}

Fix a pullback context $\mathcal C$, and pullback formalism $D$ on $\mathcal C$.

\begin{defn} \label{defn:covers dense generators}
	\hfill
	\begin{itemize}

		\item A (small) $D$-pseudocover for an object $S \in \mathcal C$ is a (small) family of maps $\{S_k \to S\}_k$ such that the functor
			\[
				D(S) \to \prod_k D(S_k)
			\]
			is conservative.

		\item Given a map $i : Z \to S$ in $\mathcal C$, a (small) $D$-pseudocover of $i$ is a (small) family of maps $\{\sigma_k : S_k \to S\}$ such that $\{\sigma_k\}_k \cup i^\complement$ is a $D$-pseudocover of $S$, where $i^\complement$ is the collection of all maps $X \to S$ such that $X \times_S Z$ is initial.\footnote{This only gives a useful definition of $i^\complement$ when $\mathcal C$ has an initial object $I$, and every map to $I$ is an equivalence. \cf{} \cite[Definition 4.1]{6FF}.}

	\end{itemize}
\end{defn}

We recall the following result from \cite[Theorem 2.4.3]{Fundamentals}:
\begin{thm} \label{thm:D-topology}
	\hfill
	\begin{enumerate}

		\item Quasi-admissible $D$-pseudocovers are stable under base change.

		\item A pullback formalism $D'$ has descent along quasi-admissible $D$-pseudocovers if it admits a morphism from $D$.

	\end{enumerate}
\end{thm}

\begin{prp} \label{prp:cover is pseudocover}
	If $\kappa$ is an infinite cardinal, then a $\kappa$-small family of quasi-admissible maps $\{X_i \to S\}_i$ is a $D$-pseudocover if and only if $1 \in D(S)$ is a $\kappa$-small colimit of objects of the form $\cls{X_{i_1} \times_S \dotsb \times_S X_{i_n}}$ for $n \geq 1$ and indices $i_1, \dotsc, i_n$.
	\begin{proof}
		For the ``only if'' direction, since $D$ has descent along $\{X_i \to S\}_i$ by \Cref{thm:D-topology}, taking the \v{C}ech nerve exhibits $D(S)$ as a $\kappa$-small limit of a diagram whose vertices are given by the $D(X_{i_1} \times_S \dotsb \times_S X_{i_n})$, where the map from $D(S)$ is given by $D(X_{i_1} \times_S \dotsb \times_S X_{i_n} \to S)$, which preserves the monoidal unit. Using \cite[Lemma D.4.7(i)]{HM6FF}, we find that the monoidal unit is equivalent to a $\kappa$-small colimit of objects of the form $(X_{i_1} \times_S \dotsb \times_S X_{i_n} \to S)_\sharp(1)$.

		For the converse, we use \Cref{prp:monoidal crit for cons family} to see that the family of quasi-admissible maps $\{X_{i_1} \times_S \dotsb \times_S X_{i_n} \to S\}_{n, i_1, \dotsc, i_n}$ is a $D$-pseudocover, whence $\{X_i \to S\}_i$ is also a $D$-pseudocover.
	\end{proof}
\end{prp}

\subsection{Locally Cartesian Pullback Formalisms} \label{S:LCPBF}

In this section we will consider a certain property of pullback formalisms, which can be seen as a strong version of the projection formula with respect to Cartesian monoidal structures. This property is only used to prove \Cref{lem:pseudotop gluing}, which underlies many of the results of \Cref{S:red to fib gluing}, but it is not needed to understand the statements of the results from that section.

If $D$ is a pullback formalism on a pullback context $\mathcal C$, the projection formula for quasi-admissible maps states that if $f : X \to Y$ is a quasi-admissible map in $\mathcal C$, then for $A \in D(X)$ and $B \in D(Y)$, we have an equivalence
\[
	f_\sharp(A \otimes f^* B) \to f_\sharp A \otimes B
.\]
Since $f^*$ has a left adjoint, it preserves limits, so it is symmetric monoidal with respect to the Cartesian monoidal structures, and we may also consider the natural map
\[
	f_\sharp(A \times f^* B) \to p_\sharp A \times f^* B
.\]
In fact, given maps $f_\sharp A \to B' \gets B$ in $D(Y)$, we obtain a natural map
\[
	f_\sharp(A \times_{f^* B'} f^* B) \to f_\sharp A \times_{B'} B
\]
as in Definition \ref{defn:LC adj}.

The property of $D$ being locally Cartesian is a strong form the projection formula for Cartesian monoidal structures, in which this map is required to be an equivalence:
\begin{defn}
	Given a pullback context $\mathcal C$, and a presheaf of categories $D$ on $\mathcal C$, say that $D$ is \emph{locally Cartesian} (LC) if it sends every quasi-admissible map to a right Cartesian functor (\Cref{defn:LC adj}), and for every $S \in \mathcal C$, $D(S)$ has universal colimits.
\end{defn}

Our main example of LC pullback formalism is the following, in which the strong form of the Cartesian projection formula holds because the functors $f_\sharp$ are slice projections:
\begin{exa} \label{exa:completed Huniv is LC}
	Following \cite[Example 1.2.12]{Fundamentals}, if $\mathcal C$ is a presentable category with universal colimits, then if $\mathcal C$ is a pullback context, we have a pullback formalism $\bar H^\univ = \bar H^\univ_{\mathcal C}$ on $\mathcal C$ given by sending any map $f : X \to Y$ to the functor $\mathcal C_{/Y} \to \mathcal C_{/X}$ given by base change along $f$. Note that this does not depend on the quasi-admissibility structure on $\mathcal C$, since this actually holds for the quasi-admissibility structure on $\mathcal C$ in which all maps are quasi-admissible.
	
	In fact, $\bar H^\univ$ is a LC pullback formalism, since every slice $\mathcal C_{/S}$ of $\mathcal C$ has universal colimits, and base change functors are right Cartesian by \Cref{cor:slice proj is LC}.

	In particular, if $\mathcal C$ is a small pullback context, we can consider the pullback formalism $\hat H^\univ_{\mathcal C} = \hat H^\univ$ given by restricting $\bar H^\univ_{\Psh(\mathcal C)}$ along the Yoneda embedding, and $\hat H^\univ$ is a LC pullback formalism on $\mathcal C$.
\end{exa}

%
The main reason that LC pullback formalisms will be useful for us is because they admit a particular strategy for checking that maps are equivalences. Roughly speaking, if $D$ is a LC pullback formalism on a pullback context $\mathcal C$, and $S \in \mathcal C$ is an object, then we can reduce the problem of checking that a map $\zeta : X \to Y$ in $D(S)$ is an equivalence to checking that for quasi-admissible maps $\sigma : S' \to S$, certain pullbacks of $\sigma^* \zeta$ are equivalences:
\begin{lem} \label{lem:check equiv on fibres for LCPF}
	Let $D$ be a LC pullback formalism on a pullback context $\mathcal C$, $\phi : D \to D'$ be a morphism of pullback formalisms, and $S \in \mathcal C$. Suppose that we have an equivalence $\varinjlim_a Y_a \to Y$ in $D(S)$ where $\{Y_a\}_a$ is a small diagram, and that for each $a$, $Y_a \simeq (\sigma_a)_\sharp(M_a)$ for some quasi-admissible $\sigma_a : S_a \to S$, and $M_a \in D(S_a)$.

	For any map $\zeta : X \to Y$ in $D(S)$, write $\zeta_a$ for the base change of $\sigma_a^* \zeta$ along the map $M_a \to \sigma_a^* Y$ corresponding to
	\[
		(\sigma_a)_\sharp M_a \simeq Y_a \to \varinjlim_a Y_a \to Y
	.\]

	If $\phi(\zeta_a)$ is an equivalence for all $a$, then $\phi(\zeta)$ is an equivalence.
	\begin{proof}
		For any quasi-admissible map $\sigma$ to $S$, since the adjunction $\sigma_\sharp \dashv \sigma^*$ is LC (for $D$), we have that the base change of $\zeta$ along a map $\sigma_\sharp(V) \to Y$ is $\sigma_\sharp$ of the base change of $\sigma^*(\zeta)$ along a map $V \to \sigma^* Y$. Therefore, for each $a$, we have that
		\[
			Y_a \times_Y \zeta \simeq (\sigma_a)_\sharp \zeta_a
		.\]
		Since $D(S)$ has universal colimits, and $\varinjlim_a Y_a \to Y$ is an equivalence, we have that $\zeta$ is a colimit of maps of the form $(\sigma_a)_\sharp(\zeta_a)$.

		Since $\phi$ respects quasi-admissibility, we have that
		\[
			\phi((\sigma_a)_\sharp(\zeta_a)) \simeq (\sigma_a)_\sharp(\phi(\zeta_a))
		.\]
		Since $\phi(\zeta_a)$ is an equivalence, so is $\phi((\sigma_a)_\sharp(\zeta_a))$, and since $\phi$ also preserves colimits, $\phi(\zeta)$ is a colimit of equivalences $\phi((\sigma_a)_\sharp(\zeta_a))$, so it is an equivalence.
	\end{proof}
\end{lem}

Unfortunately, it is not so easy to produce LC pullback formalisms apart from by considering objectwise LC localizations of $\hat H^\univ$. Results such as \Cref{prp:LC iff slice localizations} and \Cref{prp:characterize slice proj} show that satisfying the LC condition places strong restrictions on our adjunctions. The following result shows that $H^\univ$ is often not LC:

\begin{lem} \label{lem:qadm is adm iff Huniv is LC}
	Let $\mathcal C$ be a quasi-small pullback context. If quasi-admissible maps have quasi-admissible diagonals, then $H^\univ$ is a LC pullback formalism, and the converse holds if quasi-admissible maps are closed under retracts.\footnote{It is possible to slightly weaken this condition.}
\end{lem}

Fortunately, we will often be able to use \Cref{lem:res local equiv} to reduce to the case of an LC pullback formalism. This fact will be crucial for proving the results of \Cref{S:red to fib gluing}.

\begin{rmk} \label{rmk:smoothness retracts}
	When quasi-admissibility is defined by some lifting property, it is easy to see that quasi-admissible maps are closed under retracts. For example, for schemes locally of finite presentation over a base, smoothness is equivalent to formal smoothness, so smoothness of maps is closed under retracts.
\end{rmk}

\begin{rmk} \label{rmk:the mistake}
	\Cref{lem:qadm is adm iff Huniv is LC} shows that the usual pullback formalism of unstable motivic homotopy on schemes does \emph{not} satisfy the Cartesian projection formula, as is implicitly assumed in the proof of \cite[Theorem 3.2.21]{A1htpysch}. This mistake was repeated several times in the literature, although it was later fixed in many cases. Also see \cite[Remark 3.24]{SixAlgSt}.
\end{rmk}

\begin{proof}[Proof of \Cref{lem:qadm is adm iff Huniv is LC}]
	For any $S \in \mathcal C$, $H^\univ(S)$ has universal colimits since it is a presheaf category, so we just need to consider the condition that for every quasi-admissible $\sigma$, $\sigma_\sharp$ is left Cartesian.

	Note that for any quasi-admissible $\sigma : S' \to S$, \Cref{prp:LC iff slice localizations} says that since $\sigma_\sharp$ is conservative, it is left Cartesian if and only if for all $P \in \Psh(\mathcal C_{S'})$, the functor
	\[
		\Psh(\mathcal C_{S'})_{/P} \to \Psh(\mathcal C_S)_{/\sigma_\sharp P}
	\]
	is an equivalence.

	Since $\sigma_\sharp$ preserves colimits, and $P$ is a colimit of objects of the form $\yo(X)$ for $X \in \mathcal C_{S'}$, \cite[Theorem 6.1.3.9]{htt} says that this functor is a limit of functors of the form
	\[
		\Psh(\mathcal C_{S'})_{/\yo(X)} \to \Psh(\mathcal C_S)_{/\sigma_\sharp \yo(X)}
	\]
	for $X \in \mathcal C_{S'}$. Thus, it is equivalent to show that if the diagonal of every quasi-admissible map is quasi-admissible, then functors of this form are equivalences, and that the converse holds if quasi-admissible maps are closed under retracts.

	Indeed, this functor is equivalent to the colimit-preserving functor
	\[
		\Psh((\mathcal C_{S'})_{/X}) \to \Psh((\mathcal C_S)_{/X})
	\]
	induced by 
	\[
		(\mathcal C_{S'})_{/X} \to (\mathcal C_S)_{/X}
	.\]

	Thus, \Cref{lem:crit for equiv of psh cats} says that 
	\[
		\Psh(\mathcal C_{S'})_{/\yo(X)} \to \Psh(\mathcal C_S)_{/\sigma_\sharp \yo(X)}
	\]
	is an equivalence if and only if
	\[
		(\mathcal C_{S'})_{/X} \to (\mathcal C_S)_{/X}
	\]
	is fully faithful and every object in the codomain is a retract of an object in the domain.

	If the diagonal of every quasi-admissible map is quasi-admissible, then this functor is actually equivalent to the identity of the category $\mathcal C_X$, so we find that
	\[
		\Psh(\mathcal C_{S'})_{/\yo(X)} \to \Psh(\mathcal C_S)_{/\sigma_\sharp \yo(X)}
	\]
	is an equivalence.

	Conversely, assume that for all quasi-admissible maps $S' \to S$, the functor
	\[
		\mathcal C_{S'} \simeq (\mathcal C_{S'})_{/S'} \to (\mathcal C_S)_{/S'}
	\]
	is fully faithful and every object in the codomain is a retract of an object in the domain. If quasi-admissible maps are closed under retracts, it follows that for any $Y \to S'$ in $\mathcal C_S$, the map $Y \to S'$ is quasi-admissible. In particular, for any quasi-admissible $X \to S$, we can take $S' \to S$ to be the map $X \times_S X \to S$, and $Y \to S'$ to be the diagonal $X \to X \times_S X$. This shows that for any quasi-admissible map $X \to S$, the diagonal $X \to X \times_S X$ is quasi-admissible, as desired.
\end{proof}

\subsection{Pseudotopologies} \label{S:pseudotop}

The notion of pseudotopologies was introduced in \cite{Fundamentals}. These can be thought of as a structure that encodes general invariance properties, such as descent, homotopy invariance, and compatibility with colimits -- see \cite[Examples 2.1.4 and 2.1.5]{Fundamentals}. We begin with a brief overview of some notions from \cite[\S2.1, 2.3, 3.2]{Fundamentals} (see there for a more detailed introduction).

\begin{defn}
	Let $\mathcal C$ be a locally small category.
	\begin{enumerate}

		\item A \emph{pseudosieve} on $S \in \mathcal C$ is a map to $\yo(X)$ in $\Psh(\mathcal C)$.

		\item A \emph{pseudotopology} $\tau$ on $\mathcal C$, consists of, for each $X \in \mathcal C$, the specification of a collection of pseudosieves on $X \in \mathcal C$, called ($\tau$-)acyclic pseudosieves, such that for any acyclic pseudosieve $\tilde Y \to \yo(Y)$, and map $X \to Y$, the base change of $g$ along $\yo(X \to Y)$ is an acyclic pseudosieve on $X$. We may also view $\tau$ simply as the collection of maps in $\Psh(\mathcal C)$ that are $\tau$-acyclic pseudosieves, so we sometimes say that a map $P \to \yo(X)$ is \emph{in $\tau$} to mean that it is a $\tau$-acyclic pseudosieve.

		\item A pseudotopology $\tau$ on $\mathcal C$ is \emph{small} if for every $X \in \mathcal C$, the collection of $\tau$-acyclic pseudosieves on $S$ is small.

		\item Given a pseudotopology $\tau$ on $\mathcal C$, a map $P \to Q$ in $\Psh(\mathcal C)$ is said to be \emph{$\tau$-acyclic} if for any $Y \in \mathcal C$ and map $\yo(Y) \to Q$, the base change $P \times_Q \yo(Y) \to \yo(Y)$ is a $\tau$-acyclic pseudosieve.

		\item Given a pseudotopology $\tau$ on $\mathcal C$, we define the category $\Psh^\tau(\mathcal C)$ of \emph{$\tau$-local presheaves} to be the full subcategory of objects of $\Psh(\mathcal C)$ that are local with respect to the $\tau$-acyclic pseudosieves.

		\item Given a pseudotopology $\tau$ on $\mathcal C$, we define the collection of \emph{$\tau$-local equivalences} in $\Psh(\mathcal C)$ to be the maps $P \to Q$ in $\Psh(\mathcal C)$ such that for any $\tau$-local presheaf $R$ on $\mathcal C$, the map
			\[
				\Psh(\mathcal C)(P \to Q, R)
			\]
			is an equivalence.

	\end{enumerate}
	
\end{defn}

Although pseudotopologies are more general than Grothendieck topologies, we still have the following good properties of the corresponding ``pseudosheafification'' functors given in \cite[Proposition 2.1.9]{Fundamentals}.
\begin{prp}
	If $\tau$ is a small pseudotopology on a small category $\mathcal C$, then $\Psh^\tau(\mathcal C)$ is a presentable category with universal colimits, and the inclusion $\Psh^\tau(\mathcal C) \subseteq \Psh(\mathcal C)$ has an accessible left adjoint which is a locally Cartesian localization (see \Cref{defn:LC adj} or \cite[\S C.1]{Fundamentals}).
\end{prp}

The following result is given by \cite[Lemma 2.1.13]{Fundamentals}.
\begin{lem} \label{lem:res pseudotop}
	Let $F : \mathcal C \to \mathcal D$ be a morphism of locally small pullback contexts,\footnote{This means that $\mathcal C$ and $\mathcal D$ are equipped with quasi-admissibility structures, and $F : \mathcal C \to \mathcal D$ is a functor that preserves quasi-admissible maps and base changes along quasi-admissible maps.}
	Let $\rho$ be a pseudotopology on $\mathcal D$ such that for every $S \in \mathcal C$, every $\rho$-acyclic pseudosieve on $F(S)$ is \emph{weakly $\yo$-quasi-admissible} in the sense of \cite[Definition 2.1.11]{Fundamentals}, \ie{} it is of the form $\varinjlim \yo Y \to \yo(F(S))$ for some small diagram $Y : L \to \mathcal D_{F(S)}$.

	If $\mathcal C$ is small so that there is a colimit-preserving extension $F_! : \Psh(\mathcal C) \to \Psh(\mathcal D)$, then there is a (unique) pseudotopology $\tau$ on $\mathcal C$ such that every map $P \to Q$ in $\Psh(\mathcal C)$ is $\tau$-acyclic if and only if it is sent to a $\rho$-acyclic map by $F_!$.
\end{lem}

If $\tau$ is a pseudotopology on a locally small pullback context $\mathcal C$, for every object $S \in \mathcal C$, \Cref{lem:res pseudotop} allows us to view $\tau$ as a pseudotopology on $\mathcal C_S$ and $\mathcal C_{/S}$, and the following result lets us compare $\tau$-local equivalences in $\Psh(\mathcal C_S)$ to $\tau$-local equivalences in $\Psh(\mathcal C_{/S})$. This is useful because it often lets us reduce questions about equivalences in the pullback formalism $H^\tau$ to questions about $\tau$-local equivalences in the LC pullback formalism $\hat H^\univ$, where we can use \Cref{lem:check equiv on fibres for LCPF}.
\begin{lem} \label{lem:res local equiv}
	Let $\mathcal C'$ be a full anodyne pullback subcontext of a small pullback context $\mathcal C$. Let $\tau$ be a pseudotopology on $\mathcal C$ such that for every $S' \in \mathcal C'$, every $\tau$-acyclic pseudosieve on $S'$ is weakly $\yo$-quasi-admissible.

	If $\tau'$ is the induced pseudotopology on $\mathcal C'$ coming from \Cref{lem:res pseudotop}, then the restriction functor $\Psh(\mathcal C) \to \Psh(\mathcal C')$ sends $\tau$-acyclic maps to $\tau'$-local equivalences.
	\begin{proof}
		By \cite[Proposition 5.1.6.10]{htt}, we may view $\Psh(\mathcal C')$ as a full subcategory of $\Psh(\mathcal C)$ via the (fully faithful) left adjoint of the restriction functor. Since $\mathcal C' \to \mathcal C$ is anodyne, \cite[Lemma 2.1.12]{Fundamentals} implies that every $\tau$-acyclic map in $\Psh(\mathcal C)$ to an object of $\Psh(\mathcal C')$ is some $\tau'$-acyclic map in $\Psh(\mathcal C')$. We conclude by \Cref{lem:compat loc ind}.
	\end{proof}
\end{lem}

\begin{exa} \label{exa:counter}
	In the setting of \Cref{lem:res local equiv}, if not all $\tau$-acyclic pseudosieves on objects of $\mathcal C'$ are weakly $\yo$-quasi-admissible, then we can still try to define a notion of $\tau$-local equivalences in $\Psh(\mathcal C')$ by simply considering all maps whose images in $\Psh(\mathcal C)$ are $\tau$-local equivalences, and we can ask if every $\tau$-acyclic pseudosieve restricts to a $\tau$-local equivalence. 

	Our main application of \Cref{lem:res local equiv} is to show that if $\tau$ is a pseudotopology on $\mathcal C$, then for any $S \in \mathcal C$, restriction along $\mathcal C_S \subseteq \mathcal C_{/S}$ sends $\tau$-acyclic pseudosieves to $\tau$-local equivalences. In particular, the following should be true: if $\mathcal U \subseteq \yo(S) \in \Psh(\mathcal C_{/S})$ is a $\tau$-covering sieve, \ie{} $\mathcal U \to \yo(S)$ is a monomorphism that is a $\tau$-acyclic pseudosieve, then the map
	\[
		\varinjlim_{\substack{X \to S \text{ in $\mathcal U$} \\ \text{quasi-admissible}}} \yo(X) \to \yo(S)
	\]
	should be a $\tau$-local equivalence.

	In particular, this map should be sent to an effective epimorphism in $\Psh^\tau(\mathcal C_{/S})$. If $\tau$ is actually a Grothendieck topology, then by \Cref{lem:eff epi from colim vs eff epi from diagram,lem:characterize covering families by eff epi}, we see that the map above is sent to an effective epimorphism in $\Psh^\tau(\mathcal C_{/S})$ if and only if the subsieve $\mathcal U'$ of $\mathcal U$ generated by the quasi-admissible $X \to S$ in $\mathcal U$ is a $\tau$-covering sieve.

	It is not too difficult to come up with examples where this fails if not all $\tau$-covering sieves of $S$ contain $\tau$-covering families that consisting of quasi-admissible maps. In all of the following examples, $\mathcal C$ is the category of schemes, $k$ is a field, and $\tau$ is some Grothendieck topology such that every $\tau$-covering family is jointly surjective.
	\begin{description}

		\item[Smooth maps and cdh covers] Suppose that the quasi-admissible maps are the smooth maps, and $\tau$ contains the proper cdh topology, so for any closed immersion $Z \to X$, if $\tilde X \to X$ is the blow-up of $X$ at $Z$, then $\{Z, \tilde X \to X\}$ is a $\tau$-covering family. Suppose $S$ is the scheme $\Spec k[x,y]/(xy)$, and let $Z \to S$ be given by $\Spec k[x,y]/(x) = Z(x) \subseteq S$, so that the blow-up of $Z$ in $S$ is precisely the closed subscheme $\tilde S = \Spec k[x,y]/(y) = Z(y) \subseteq S$.\footnote{Thanks to Noah Olander for this example.}

			If we let $\mathcal U$ be the sieve generated by $Z,\tilde S \to S$, then $\mathcal U$ is a $\tau$-covering sieve, and $\mathcal U'$ is the sieve on $S$ generated by all smooth maps to $S$ that factor through $Z(x)$ or $Z(y)$, but since smooth maps are open, it follows that no such smooth map can contain the point $Z(x,y)$ in its image, so $\mathcal U'$ is not a covering sieve.

		\item[Smooth maps and constructible covers] Suppose that the quasi-admissible maps are the smooth maps, and that $\tau$ contains the constructible topology, \ie{} for any closed immersion $Z \to X$, the family $\{Z, X \setminus Z \to X\}$ is $\tau$-covering. Let $\Spec k \cong s \to S$ be some closed immersion over $k$, and let $\mathcal U$ be the sieve generated by $\{s, S \setminus s \to S\}$.

			If $S$ is connected, then $s \to S$ is not open, so since smooth maps are open, no smooth map to $S$ can factor through the inclusion $s \to S$. Therefore, $\mathcal U'$ is the sieve on $S$ generated by all smooth maps to $S$ that land in $S \setminus s$. Hence, $\mathcal U'$ is the sieve generated by $S \setminus s \to S$, which is not a covering sieve.

		\item[Proper maps and Zariski covers] Suppose that the quasi-admissible maps are the \emph{proper} maps, and that every Zariski cover is a $\tau$-cover. Let $\mathcal U$ be the sieve generated by all open subsets $U \subsetneq S$ -- this is a covering sieve as long as $S$ contains at least 2 closed points. In this case, $\mathcal U'$ is generated by the proper maps $X \to S$ that are not surjective.

			If $S = \proj^1_k$ and $k$ is algebraically closed, the only closed subsets $Z \subsetneq S$ are finite sets of closed points, so $\mathcal U'$ is a sieve consisting of maps $X \to S$ whose images are finite sets of closed points. It is easy to see that this is not a covering sieve since the generic point is not in the image of any of these maps.

	\end{description}
\end{exa}

Given a pullback context $\mathcal C$, we will need the definition of $\yo$-quasi-admissible maps in $\Psh(\mathcal C)$ from \cite[Definition 2.2.1]{Fundamentals}, which is somewhat cumbersome to state. Rather than recalling the full definition here, we will mention some results that describe and produce $\yo$-quasi-admissible maps:
\begin{lem}[$\yo$-quasi-admissible maps] \label{lem:yo-qadm}
	Let $\mathcal C$ be a locally small pullback context.
	\begin{enumerate}

		\item $\yo : \mathcal C \to \Psh(\mathcal C)$ sends quasi-admissible maps to $\yo$-quasi-admissible maps.

		\item If $\mathcal C$ is small, then for any sieve $\mathcal U$ on $S \in \mathcal C$, the map $\mathcal U \to \yo(S)$ is quasi-admissible if and only if $\mathcal U$ is generated by quasi-admissible maps to $S$.

		\item If $\mathcal C$ is small, then for any $S \in \mathcal C$, every $\yo$-quasi-admissible map to $\yo(S)$ is weakly $\yo$-quasi-admissible, \ie{} it is of the form
			\[
				\varinjlim \yo X \to \yo(S)
			\]
			for some small diagram $X : K \to \mathcal C_S$.

		\item For any small diagram $X : K^\triangleright \to \mathcal C$ that sends every edge to a quasi-admissible map, the map
			\[
				\varinjlim_{a \in K} \yo(X(a)) \to \yo(X(\infty))
			\]
			is $\yo$-quasi-admissible.

		\item For any pullback formalism $D \in \PF(\mathcal C)$, we can view $D$ as a limit-preserving presheaf on $\Psh(\mathcal C)$ that sends every $\yo$-quasi-admissible map to a right adjoint functor.

	\end{enumerate}
	\begin{proof}\hfill
		\begin{enumerate}

			\item \cite[Proposition 2.2.7]{Fundamentals}.

			\item \cite[Lemma 2.4.5]{Fundamentals}.

			\item \cite[Lemma 2.2.4]{Fundamentals}.

			\item Follows from \cite[Proposition 2.2.5]{Fundamentals} and the first point.

			\item \cite[Definition 2.2.1]{Fundamentals}.

		\end{enumerate}
		
	\end{proof}
\end{lem}

We now recall \cite[Definition 3.2.2]{Fundamentals}:
\begin{defn}
	A pseudotopology on a pullback context is said to be \emph{quasi-admissible} if all acyclic pseudosieves are $\yo$-quasi-admissible.
\end{defn}

\begin{exa} \label{exa:red pseudotop}
	For any locally small category $\mathcal C$, we can consider the pseudotopology $\red$ on $\mathcal C$, where the $\red$-acyclic pseudosieves are the maps of the form $\emptyset \to \yo(S)$ when $S \in \mathcal C$ admits a map to an initial object. Note that for any quasi-admissibility structure on $\mathcal C$, $\red$ is a quasi-admissible pseudotopology by \Cref{lem:yo-qadm}, specifically \cite[Lemma 2.2.5 and Proposition 2.2.7, or Lemma 2.4.5]{Fundamentals}.

	A presheaf on $\mathcal C$ is $\red$-local if and only if it sends initial objects of $\mathcal C$ to terminal objects. In fact, if $D$ is a pullback formalism on $\mathcal C$, then $D$ is $\red$-invariant (\ie{} the $\red$-acyclic pseudosieves are $D$-acyclic) if and only if $D(I) \simeq \pt$ whenever $I$ is an initial object of $\mathcal C$.
\end{exa}
For the remainder of this section, we fix a quasi-admissible pseudotopology $\tau$ on a locally small pullback context $\mathcal C$. 

\begin{defn}
	For $D \in \PF(\mathcal C)$, write $\mathcal W_S^D$ for the collection of maps in $D(S)$ of the form $\sigma_\sharp(u_\sharp 1 \to 1)$, where $\sigma : S' \to S$ is quasi-admissible, $u$ is a $\tau$-acyclic pseudosieve, and $u_\sharp 1 \to 1$ is the counit of $u_\sharp \dashv u^*$ evaluated at the monoidal unit $1 \in D(S')$.

	Say $\tau$ is \emph{$D$-small} if $\mathcal W_S^D$ is generated under small colimits by a small subcollection for all $S \in \mathcal C$.
\end{defn}

\begin{rmk}
	If $\mathcal C$ is quasi-small, and $\tau$ is small, then $\mathcal W_S^D$ is small for all $S \in \mathcal C$ and $D \in \PF(\mathcal C)$. In particular, $\tau$ is $D$-small for all $D \in \PF(\mathcal C)$.
\end{rmk}

The following result is given in \cite[Proposition 3.2.5]{Fundamentals}:
\begin{prp} \label{prp:invariance localization}
	For any pullback formalism $D$ on $\mathcal C$ such that $\tau$ is $D$-small, there is a morphism of pullback formalisms $L_\tau : D \to D^\tau$ such that the natural functor $\PF(\mathcal C)_{D^\tau/} \to \PF(\mathcal C)_{D/}$ is fully faithful with essential image given by those pullback formalisms that are $\tau$-invariant, \ie{} they send $\tau$-acyclic pseudosieves to fully faithful functors.

	Furthermore, if $D$ is geometrically generated, then for any $S \in \mathcal C$, the functor $L^\tau : D(S) \to D^\tau(S)$ is a localization along $\mathcal W_S^D$.
\end{prp}

The following result is given in \cite[Corollary 3.2.6]{Fundamentals}, and allows us to construct pullback formalisms of the form $S \mapsto \Psh^\tau(\mathcal C_S)$:
\begin{prp} \label{prp:univ invar}
	Suppose $\mathcal C$ is quasi-small, and $\tau$ is small. The full subcategory of $\PF(\mathcal C)$ consisting of $\tau$-invariant pullback formalisms (as in \Cref{prp:invariance localization}) has an initial object $H^\tau$, and for any $S \in \mathcal C$, the functor $H^\univ(S) \to H^\tau(S)$ is an accessible locally Cartesian localization of $\Psh(\mathcal C_S)$ along all maps that lie over $\tau$-acyclic pseudosieves in $\Psh(\mathcal C)$.
\end{prp}

\begin{exa} \label{exa:completed invar PF}
	Recall the definition of $\hat H^\univ$ from \Cref{exa:completed Huniv is LC}.
	If $\mathcal C$ is small and $\tau$ is a small pseudotopology, then there is morphism of pullback formalisms $\hat L_\tau : \hat H^\univ \to \hat H^\tau$ such that for any $S \in \mathcal C$, the functor $\hat L_\tau : \hat H^\univ(S) \to \hat H^\tau(S)$ is equivalent to the localization $\Psh(\mathcal C_{/S}) \to \Psh^\tau(\mathcal C_{/S})$ along all maps in $\Psh(\mathcal C_{/S})$ whose image under $\Psh(\mathcal C_{/S}) \to \Psh(\mathcal C)$ is a $\tau$-acyclic pseudosieve.

	In fact, we have that $\hat L_\tau$ extends to a morphism of pullback formalisms on the pullback context $\Psh(\mathcal C)$ in which all maps are quasi-admissible.
	\begin{proof}
		Equip $\Psh(\mathcal C)$ with the quasi-admissibility structure in which all maps are quasi-admissible. Define a pseudotopology $\hat \tau$ on $\Psh(\mathcal C)$ by setting the $\hat \tau$-acyclic pseudosieves to be the maps of the form $\yo(f)$, for $f$ a $\tau$-acyclic map. By our choice of quasi-admissibility structure, $\hat \tau$ is a quasi-admissible pseudotopology by \Cref{lem:yo-qadm} (\cite[Proposition 2.2.7]{Fundamentals}).

		Recall the pullback formalism $\bar H^\univ = \bar H^\univ_{\Psh(\mathcal C)}$ of \Cref{exa:completed Huniv is LC}. Note that for any $P \in \Psh(\mathcal C)$, the collection $\mathcal W_P^{\bar H^\univ}$ of maps in $\bar H^\univ(P)$ consists of maps in $\Psh(\mathcal C)_{/P} = \bar H^\univ(P)$ that lie over $\tau$-acyclic maps in $\Psh(\mathcal C)$.

		Since $\mathcal C$ is small, any map in $\mathcal W_P^{\bar H^\univ}$ is a small colimit of maps lying over $\tau$-acyclic pseudosieves in $\Psh(\mathcal C)$, so since $\tau$ is small, $\mathcal W_P^{\bar H^\univ}$ is generated under small colimits by a small subcollection.

		Note that with our choice of quasi-admissibility structure on $\Psh(\mathcal C)$, the pullback formalism $\bar H^\univ$ is actually geometrically generated, so by \Cref{prp:invariance localization}, there is a morphism of pullback formalisms on $\Psh(\mathcal C)$
		\[
			L_{\hat \tau} : \bar H^\univ \to (\bar H^\univ)^{\hat \tau}
		\]
		such that for any $P \in \Psh(\mathcal C)$, the functor $\bar H^\univ(P) \to (\hat H^\univ)^\tau(P)$ is the localization of $\bar H^\univ(P) \simeq \Psh(\mathcal C)_{/P}$ along maps in $\mathcal W_P^{\bar H^\univ}$, which we have already seen is generated under small colimits by maps lying over $\tau$-acyclic pseudosieves in $\Psh(\mathcal C)$.

		Thus, if we define
		\[
			\hat L_\tau : \hat H^\univ \to \hat H^\tau
		\]
		to be $L_{\hat \tau}$ restricted to $\mathcal C$, we get that $\hat L_\tau$ actually defines a morphism of pullback formalisms on $\Psh(\mathcal C)$ (by construction), and for any $S \in \mathcal C$, the functor $\hat L_\tau : \hat H^\univ(S) \to \hat H^\tau(S)$ is the localization $\Psh(\mathcal C_{/S}) \to \Psh^\tau(\mathcal C_{/S})$ of $\Psh(\mathcal C_{/S})$ along maps lying over $\tau$-acyclic maps in $\Psh(\mathcal C)$.
	\end{proof}
\end{exa}

\section{Basic Techniques for Establishing Gluing} \label{S:basic gluing}


The ``localization property'' that we discussed in the introduction can be seen as a particular case of ``gluing''. This term was already used in the original localization theorem of \cite[Theorem 3.2.21]{A1htpysch}. In this section, we begin our study of the gluing property defined in \cite[\S4]{6FF}. As the applications of the gluing property have already been studied in \cite[\S4]{6FF}, we will focus only on techniques for establishing it.

Fix a pullback context $\mathcal C$ with a \emph{strict initial object}, which is an initial object $\emptyset$ such that every map to $\emptyset$ is invertible.

We recall that if $i : Z \to S$ is a map in a pullback context $\mathcal C$, we say that $i$ is \emph{closed} if it has a quasi-admissible complement. A \emph{complement} of $i$ is a monomorphism $j : U \to S$ such that a map $X \to S$ factors through $j$ if and only if $\emptyset$ is a pullback of $i$ with $X \to S$. The notion of complements is discussed further in \cite[Definition 4.1]{6FF}.

Given a pullback formalism $D$ on $\mathcal C$, the \emph{gluing square for $i$ $\square_i$} is defined to be the following natural square of endofunctors of $D(S)$:
\[
	\begin{tikzcd}
		j_\sharp j^* \ar[d] \ar[r] & \id \ar[d] \\
		j_\sharp j^* i_* i^* \ar[r] & i_* i^*
	\end{tikzcd}
	\qquad\text{or equivalently}\qquad
	\begin{tikzcd}
		\cls{U} \otimes - \ar[d] \ar[r] & \id \ar[d] \\
		\cls{U} \otimes i_* i^* \ar[r] & i_* i^*
	\end{tikzcd}
.\]
We say that \emph{$D$ has gluing for $i$} if this square is coCartesian. When $D$ is not clear from context, we may write $\square_i^D$ to denote this square.

Assume now that $D$ is \emph{reduced}, \ie{} $D(\emptyset) \simeq \pt$. Using \Cref{exa:red pseudotop}, this is equivalent to the condition that $D \in \PF^\red(\mathcal C)$.

In this case, \cite[Lemma 4.6]{6FF} shows that all not-necessarily-commutative squares as below with the same bottom-left corner, and top and right arrows, commute in an essentially unique way and are equivalent to $\square_i$:
\[
	\begin{tikzcd}
		j_\sharp j^* \ar[d] \ar[r] & \id \ar[d] \\
		j_\sharp \pt \ar[r] & i_* i^*
	\end{tikzcd}
	\qquad\text{or equivalently}\qquad
	\begin{tikzcd}
		\cls{U} \otimes - \ar[d] \ar[r] & \id \ar[d] \\
		\cls{U} \otimes \pt \ar[r] & i_* i^*
	\end{tikzcd}
.\]

In fact, we will mostly consider gluing for reduced pullback formalisms, and the only non-reduced pullback formalisms for which we will consider $\square_i$ are $H^\univ$ and $\hat H^\univ$. Note that in these cases, $\square_i$ and the above squares still agree when evaluated on reduced presheaves.

We also recall from \cite[Proposition 4.13]{6FF} that if $D$ has gluing for $i$, then the following are equivalent:
\begin{enumerate}

	\item $i_* : D(Z) \to D(S)$ is conservative.

	\item $i_* : D(Z) \to D(S)$ is fully faithful.

	\item $D(Z) \xrightarrow{i_*} D(S) \xrightarrow{j^*} D(U)$ is a fibre sequence of categories.

\end{enumerate}
In this case $i$ is said to be \emph{$D$-closed}. This property has very strong consequences that are studied in \cite[\S4.2 and 4.3]{6FF}, and is used for constructing 6-functor formalisms in \cite[\S6]{6FF}.

The following result is given in \cite[Lemma 4.4, Remark 4.7, and Lemma 4.16]{6FF}:
\begin{lem} \label{lem:gluing locality on base}
	Suppose that $i$ admits a $D$-pseudocover by quasi-admissible maps $S' \to S$ such that $D$ has gluing for $i \times_S S'$ (\resp{} $i \times_S S'$ is $D$-closed). Then $D$ has gluing for $i$ (\resp{} $i$ is $D$-closed).
\end{lem}

\begin{rmk} \label{rmk:gluing on fundamental}
	As in \cite[Remark 4.9]{6FF}, for any $X \in \mathcal C_S$, we have that the boundary of $\square_i(\cls{X})$ is equivalent to any (a priori) not-necessarily-commutative square
	\[
		\begin{tikzcd}
			\cls{X \times_S U} \ar[d] \ar[r] & \cls{X} \ar[d] \\
			\pt \otimes \cls{U} \ar[r] & i_* \cls{X \times_S Z}
		\end{tikzcd}
	\]
	such that the top arrow is equivalent to $\cls{X \times_S U \to X}$, and the right arrow is the unit $\cls{X} \to i_* i^* \cls{X}$.
\end{rmk}

\subsection{Reduced Presheaves} \label{S:red}

As before, we fix a pullback context $\mathcal C$ with a strict initial object. Recall the quasi-admissible pseudotopology $\red$ of \Cref{exa:red pseudotop}: the $\red$-acyclic pseudosieves are precisely those of the form $\emptyset \to \yo(I)$ when $I$ is an initial object of $\mathcal C$. Assume $\mathcal C$ is quasi-small, so that the initial pullback formalism $H^\univ$ of \Cref{exa:Huniv} exists.

\begin{rmk} \label{rmk:red PF}
	Using \Cref{prp:univ invar}, we have that the following are equivalent:
	\begin{enumerate}

		\item $D$ is reduced, \ie{} $D(I) \simeq \pt$ for any initial object $I \in \mathcal C$.

		\item $D$ is a $\red$-invariant pullback formalism.

		\item There is a morphism of pullback formalisms $H^\red \to D$.

		\item There is a contractible space of morphisms of pullback formalisms $H^\red \to D$.

	\end{enumerate}
\end{rmk}
Most of the  pullback formalisms we will be interested in studying are reduced.

Fix a reduced pullback formalism $D$ on $\mathcal C$. For any $S \in \mathcal C$, say a map in $H^\univ(S)$ is a \emph{$D$-equivalence} if it is sent to an equivalence by the functor $H^\univ(S) \to D(S)$.

It turns out that when $H^\red \to D$ is a localization of pullback formalisms, and $i_* : H^\univ(Z) \to H^\univ(S)$ preserves $D$-equivalences between reduced presheaves, we have especially strong tools for establishing that $D$ has gluing for $i$. This section serves to establish some of the basic results of this type. The results of this section will also be used to establish the important reduction given in \Cref{S:red to fib gluing}.

\begin{rmk}
	Many of these results are easier to prove if we instead assume that $i_* : H^\univ(Z) \to H^\univ(S)$ preserves \emph{all} $D$-equivalences, but we are not always able to show this. This often does hold when the $D$-equivalences are given by some notion of homotopy, but when we also have local equivalences coming from a Grothendieck topology, the condition that $i_*$ preserves local equivalences is generally established using \cite[Lemma 3.1.6]{locspalg}, which only guarantees that $i_*$ preserves locally equivalences between \emph{reduced} presheaves.
\end{rmk}

The main application of this section is the following:
\begin{prp} \label{prp:intrinsic gens gluing}
	Let $\mathcal K$ be a collection of closed maps in $\mathcal C$, and let $\mathcal G \subseteq \mathcal C$ be a full subcategory such that for any map $i : Z \to S$ in $\mathcal K$, the following conditions hold:
	\begin{enumerate}

		\item Any base change of $i$ along a quasi-admissible map is still in $\mathcal K$.

		\item The unit $1 \in D(S)$ is a small colimit of objects of the form $p_\sharp(1)$, where $p$ is a quasi-admissible map $X \to S$ such that either $X \in \mathcal G$, or $X \times_S Z$ is initial.

		\item If $S \in \mathcal G$, then
			\begin{enumerate}

				\item the functor $i_* : H^\univ(Z) \to H^\univ(S)$ preserves $D$-equivalences between reduced presheaves,

				\item the functor $H^\univ(S) \to D(S)$ is a localization, and

				\item for any quasi-admissible map $X \to S$, if $X \in \mathcal G$, then this functor sends $\square_i^{H^\univ}(\yo X)$ to a coCartesian square in $D(S)$.

			\end{enumerate}

	\end{enumerate}
	Then $D$ has gluing for all maps in $\mathcal K$.
\end{prp}

We will prove \Cref{prp:intrinsic gens gluing} at the end of this section.

For now, we note the following result that will be useful in the proof of \Cref{prp:intrinsic gens gluing}:
\begin{lem} \label{lem:strong gens gluing}
	Let $\mathcal K$ be a collection of closed maps in $\mathcal C$ that is stable under quasi-admissible base change, and let $\mathcal G \subseteq \mathcal C$ be a full subcategory such that for any map $i : Z \to S$ in $\mathcal K$, the unit $1 \in D(S)$ is a small colimit of objects of the form $p_\sharp(1)$, where $p$ is a quasi-admissible map $X \to S$ such that either $X \in \mathcal G$, or $X \times_S Z$ is initial.

	Then for any map $i : Z \to S$ in $\mathcal K$, if $D$ is geometrically generated at $S$,\footnote{\ie{} $H^\univ(S) \to D(S)$ is essentially surjective} we have that $D(S)$ is generated under small colimits by objects of the form $p_\sharp(1)$, where $p$ is a quasi-admissible map $X \to S$ such that either $X \in \mathcal G$, or $X \times_S Z$ is initial.
	\begin{proof}
		Since $D$ is geometrically generated at $S$, it suffices to show that for any quasi-admissible $\sigma : S' \to S$, $\sigma_\sharp(1) \in D(S)$ is a small colimit of objects of the form $p_\sharp(1)$ as above. Indeed, since $\mathcal K$ is stable under quasi-admissible base change, we have that the base change $i' : Z' \to S'$ of $i$ is in $\mathcal K$, so $1 \in D(S')$ is a small colimit of objects of the form $p'_\sharp(1)$ where $p'$ is a quasi-admissible map $X \to S'$ such that either $X \in \mathcal G$, or $X \times_{S'} Z' = X \times_S Z$ is initial.

		It follows that $\sigma_\sharp(1)$ is a small colimit of objects of the form $p_\sharp(1)$, where $p = \sigma \circ p'$ with $p'$ as above.
	\end{proof}
\end{lem}

For the remainder of the section, fix a closed map $i : Z \to S$ in $\mathcal C$, and write $j : U \to S$ for its quasi-admissible complement.

\begin{lem} \label{lem:reduced pushforward of empty along closed map}
	If $I$ is an (strict) initial object of $\mathcal C$, the functor $i_* : H^\univ(Z) \to H^\univ(S)$ sends $\yo(I)$ to the presheaf represented by $U \in \mathcal C_S$.

	In particular, for any $X \in \mathcal C_S$, if $X \times_S Z$ is initial, then $\square_i^{H^\univ}(\yo X)$ is coCartesian.
	\begin{proof}
		The functor $i_* : H^\univ(Z) \to H^\univ(S)$ is the functor $\Psh(\mathcal C_Z) \to \Psh(\mathcal C_S)$ given by precomposing by the functor $\mathcal C_S \to \mathcal C_Z$ given by base change along $i$. Since $I$ is a strict initial object of $\mathcal C$, it follows that $i_*$ sends $\yo(I)$ to the presheaf that sends $X \in \mathcal C_S$ to $\pt$ if $X \times_S Z$ is initial, and $\emptyset$ otherwise. This is precisely the presheaf represented by the complement of $i$.

		If $X \in \mathcal C_S$ is such that $X \times_S Z$ is initial, the arrow $X \times_S U \to X$ is an equivalence, so the top horizontal arrow of $\square_i(\yo X)$ is an equivalence, and since the space of maps $\yo(U) \simeq j_\sharp j^* i_* i^* \yo(X) \to i_* \yo(X \times_S Z)$ is contractible, we have just shown that the bottom arrow of $\square_i(\yo X)$ is an equivalence, so $\square_i(\yo X)$ is coCartesian.
	\end{proof}
\end{lem}

We will need the following assumption:
\begin{ass}
	Assume that $i_* : H^\univ(Z) \to H^\univ(S)$ preserves $D$-equivalences between reduced presheaves.
\end{ass}

\begin{rmk}
	The condition that $i_* : H^\univ(Z) \to H^\univ(S)$ preserves $D$-equivalences between reduced presheaves is equivalent to the condition that $i_* : H^\red(Z) \to H^\red(S)$ preserves $\phi$-equivalences, where $\phi : H^\red \to D$ is the unique morphism of pullback formalisms, and $\phi$-equivalences are maps that are sent to equivalences by $\phi$.
\end{rmk}

\begin{rmk}
	Let $\phi : D^\flat \to D$ be a morphism of pullback formalisms on $\mathcal C$. We can check the condition that $i_* : D^\flat(Z) \to D^\flat(S)$ preserves $\phi$-equivalences locally:

	If there is a quasi-admissible $D$-pseudocover of $S$ such that for each base change $i : Z' \to S'$ along a map $S' \to S$ in the pseudocover, $i'_* : D^\flat(Z') \to D^\flat(S')$ preserves $\phi$-equivalences, then $i_* : D^\flat(Z) \to D^\flat(S)$ preserves $\phi$-equivalences.
	\begin{proof}
		By quasi-admissible base change we have a commutative square
		\[
			\begin{tikzcd}
				D^\flat(Z) \ar[d] \ar[r, "i_*"] & D^\flat(S) \ar[d] \\
				\prod_\sigma D^\flat(Z') \ar[r] & \prod_\sigma D^\flat(S')
			\end{tikzcd}
		,\]
		where the products are indexed over elements $\sigma : S' \to S$ of the $D$-pseudocover, and $Z' \to S'$ is the base change of $i$ along $\sigma$. Given a morphism $f$ in $D^\flat(Z)$ that is a $\phi$-equivalence, we have that all of the $(Z' \to Z)^*(f)$ are $\phi$-equivalences, so by our assumption we know that the image of $f$ in each $D^\flat(S')$ is a $\phi$-equivalence. Thus, $\phi i^* f$ is sent to an equivalence in $D(S')$ for each $S' \to S$ in the pseudocover, so it must be an equivalence.
	\end{proof}
\end{rmk}

We will also need an additional assumption:
\begin{ass}
	Assume that the functor $H^\univ(S) \to D(S)$ is a localization.
\end{ass}

In order to emphasize the fact that $H^\univ(S) \to D(S)$ is a localization, we will denote this functor by $L$, and we view $D(S)$ as a full subcategory $L(H^\univ(S))$ of $H^\univ(S)$.

\begin{lem} \label{lem:gluing colimits}

	If $X_0 \in \Psh(\mathcal C_S) \simeq H^\univ(S)$ is a \emph{reduced} presheaf such that $LX_0 \in D(S)$ is a small colimit in $D(S)$ of tensor products of objects $X \in D(S)$ such that $L\square_i(X)$ is coCartesian, then $\square_i^D(LX_0)$ is coCartesian.
	\begin{proof}
		Note that since $i^*$ and $\cls{U} \otimes -$ preserve all $D$-equivalences, and $i_*$ preserves $D$-equivalences between reduced presheaves, we have that $\square_i$ preserves $D$-equivalences between reduced presheaves.
		Thus, for any reduced presheaf $X$ on $\mathcal C_S$, we have an equivalence
		\[
			L\square_i^{H^\univ}(X) \simeq L\square_i^{H^\univ}(LX)
		.\]
		In particular, it suffices to assume $X_0 = LX_0$.

		Using the fact that $\yo(U)$ is $(-1)$-truncated in $H^\univ(S)$, we have that $\yo(U) \times -$, which is equivalent to $\cls{U} \otimes -$, commutes with binary products. It follows that $\square_i$ is compatible with binary products as a functor from $H^\univ(S)$ to commutative squares in $H^\univ(S)$, so since $L : H^\univ(S) \to D(S)$ is monoidal, we have that for any $X,Y \in H^\red(S) \subseteq H^\univ(S)$, 
		\[
			L\square_i(LX \otimes LY) \simeq L\square_i(L(X \times Y)) \simeq L\square_i(X \times Y) \simeq L(\square_i(X) \times \square_i(Y)) \simeq L \square_i(X) \otimes L \square_i(Y)
		.\]
		We also have that if $1_{D(S)}$ is the monoidal unit of $D(S)$, then since $1_{D(S)} \simeq L\pt$ is a reduced presheaf,
		\[
			L\square_i(1_{D(S)}) \simeq L\square_i(L\pt) \simeq L\square_i(\pt)
		,\]
		and this is coCartesian because $\square_i(\pt)$ is always coCartesian simply because the unit $\pt \to i_* i^* \pt$ is always an equivalence.

		Therefore, since the monoidal structure of $D(S)$ is compatible with pushouts in each variable, \emph{it suffices to assume that $X_0$ is $L$ of a small colimit of objects $X \in D(S) \subseteq H^\univ(S)$ such that $L\square_i(X)$ is coCartesian.}

		Since $L : H^\univ(S) \to D(S)$ is a localization, and restricts to a localization on $H^\red(S) \subseteq H^\univ(S)$, it follows that $X_0 \in D(S)$ is equivalent to $L$ of a small colimit in $H^\univ(S)$ of reduced presheaves $X \in H^\red(S) \subseteq H^\univ(S)$ such that $L\square_i^{H^\univ}(X)$ is coCartesian.

		Note that the initial reduced presheaf defines a cone over this diagram, allowing us to extend the diagram to a small weakly contractible category and only change the colimit up to $D$-equivalence (in fact, it only changes it up to an $L_\red$-equivalence). Furthermore, by \Cref{lem:reduced pushforward of empty along closed map}, we know that $\square_i^{H^\univ}$ is coCartesian at the initial reduced presheaf, so we have reduced to the case that $X_0$ is equivalent to $L$ of a small weakly contractible colimit in $H^\univ(S)$ of reduced presheaves $X \in H^\red(S) \subseteq H^\univ(S)$ such that $L \square_i^{H^\univ}(X)$ is coCartesian.

		Since weakly contractible colimits of reduced presheaves are reduced, and $\square_i^{H^\univ}$ preserves $D$-equivalences between reduced presheaves, it follows that $L \square_i(X_0)$ is equivalent to $L \square_i$ of a small colimit of objects $X$ such that $L \square_i X$ is coCartesian. Therefore, $L \square_i(X_0)$ is coCartesian since $\square_i^{H^\univ}$ is colimit-preserving as a functor from $H^\univ(S)$ to squares in $H^\univ(S)$ (this is because $i_* : H^\univ(Z) \to H^\univ(S)$ is colimit-preserving).

		Finally, since $i_*$ preserves $D$-equivalences between reduced presheaves, and $X_0 \in D(S) \subseteq H^\univ(S)$ is a reduced presheaf, it follows from \cite[Lemma F.5(2)]{TwAmb} that $L$ sends the unit map $X_0 \to i_* i^* X_0$ in $H^\univ(S)$ to the corresponding unit map in $D(S)$, whence $L \square_i(X_0) \simeq \square_i^D(X_0)$.
	\end{proof}
\end{lem}

\begin{cor} \label{cor:gluing locality globally on base}
	Suppose that $D(S)$ is generated under small colimits by objects of the form $\cls{X_1 \times_S \dotsb \times_S X_n}$ for $n \geq 1$ and $X_1, \dotsc, X_n \in \mathcal C_S$ such that for all $k$, $X_k \times_S Z$ is initial, or $L\square^{H^\univ}_i(\cls{X_k})$ is coCartesian. Then $D$ has gluing for $i$.
	\begin{proof}
		Note that if $X \in \mathcal C_S$ satisfies that $X \times_S Z$ is initial, then $\square^{H^\univ}_i(\cls{X})$ is coCartesian by \Cref{lem:reduced pushforward of empty along closed map}. Thus, the result follows from \Cref{lem:gluing colimits}, since
		\[
			\cls{X_1 \times_S \dotsb \times_S X_n} \simeq \cls{X_1} \otimes \dotsb \otimes \cls{X_n}
		.\]
	\end{proof}
\end{cor}

\begin{proof}[Proof of \Cref{prp:intrinsic gens gluing}]
	For any map $i : Z \to S$ in $\mathcal K$, \Cref{prp:cover is pseudocover} shows that $i$ has a $D$-pseudocover by quasi-admissible maps $p : X \to S$ such that $X \in \mathcal G$, and the base change of $i$ along $p$ is in $\mathcal K$.
	It follows from \Cref{lem:gluing locality on base} that we can reduce to showing that $D$ has gluing for $i$ when $S \in \mathcal G$.

	In this case, since $H^\univ(S) \to D(S)$ is essentially surjective, \Cref{lem:strong gens gluing} says that $D(S)$ is generated under small colimits by objects of the form $p_\sharp(1)$, where $p : X \to S$ is a quasi-admissible map such that $X \in \mathcal G$ or $X \times_S Z$ is initial. By \Cref{cor:gluing locality globally on base}, we may reduce to showing that $\square_i^{H^\univ}(\cls X)$ is sent to a coCartesian square in $D(S)$ for all quasi-admissible maps $X \to S$ with $X \in \mathcal G$.
\end{proof}
	
\subsection{Stabilizing}  \label{S:stabilizing}

The procedure of ``stabilizing'' pullback formalisms was studied in \cite[\S4]{Fundamentals}. Intuitively, if $D$ is a pullback formalism on a pullback context $\mathcal C$, and for $S \in \mathcal C$ we think of $D(S)$ as ``$S$-parametrized homotopy types'', then we want to construct a stabilized version of $D$ that is still a pullback formalism, and which sends $S \in \mathcal C$ to an appropriate category of ``$S$-parametrized genuine spectra''. As an example, in the case of the pullback formalism $H^\diff$ of \cite[\S5.2]{Fundamentals}, first studied in \cite[\S II]{TwAmb}, when $G$ is a compact Lie group, the category $H^\diff(\B G)$ of $\B G$-parametrized homotopy types can be identified with the $G$-equivariant homotopy category, and the stabilization $\SH^\diff(\B G)$ is the category of genuine $G$-spectra.

In \Cref{S:red to fib gluing}, we will provide tools for establishing the gluing property for pullback formalisms of ``unstable homotopy'', such as $H^\diff, H^\alg, H^\hol$ from \cite[\S5]{Fundamentals}. In this section, we will be interested in deducing the gluing property for the stabilized pullback formalism from the same property for the unstabilized pullback formalism. In fact, the stabilization process decomposes into two steps: first we use \cite[Proposition 4.2.2]{Fundamentals} to produce a pullback formalism of ``pointed homotopy types'', and then we use \cite[Theorem 4.1.16]{Fundamentals} to adjoin $\otimes$-inverses of certain objects. In the case of genuine $G$-spectra, this corresponds to first taking pointed $G$-equivariant homotopy types, and then freely adjoining $\wedge$-inverses of representation spheres.

Thus, given a pullback formalism $D$, we will achieve our goal by showing how to deduce gluing for the associated ``pointed'' pullback formalism $D_\bullet$, and for the pullback formalism $D[A^{\otimes -1}]$ given by formally adjoining $\otimes$-inverses of certain objects.


For the rest of the section, we \textbf{fix a pullback context $\mathcal C$ with a strict initial object, and $D$ a reduced pullback formalism on $\mathcal C$}.

\subsubsection{Pointing}

In \cite[Proposition 4.2.2]{Fundamentals} we produced a pullback formalism $D_\bullet$ that can be thought of as a ``pointed version'' of $D$, which admits a morphism $D \to D_\bullet$ such that for each $S \in \mathcal C$, the functor $D(S) \to D_\bullet(S)$ is equivalent to the pointing functor
\[
	D(S) \xrightarrow{\pt \coprod -} D(S)_{\pt/}
.\]

In particular, $D_\bullet(S)$ is the category of pointed objects in $D(S)$, that is, objects $F$ of $D(S)$ equipped with a pointing $\pt \to F$.

Let us first consider the gluing property for such pointed objects:
\begin{rmk} \label{rmk:gluing at ptd obj}
	Let $i : Z \to S$ be a closed map in $\mathcal C$, let $j : U \to S$ be a quasi-admissible complement of $i$, and let $\pt \xrightarrow{x_0} F$ be a map in $D(S)$. We know that $\square_i(F)$ is
	\[
		\begin{tikzcd}
			F \otimes \cls U \ar[d] \ar[r] & F \ar[d] \\
			\pt \otimes \cls U \ar[r] & i_* i^* F
		\end{tikzcd}
	.\]

	This is the bottom rectangle in the following diagram:
	\[
		\begin{tikzcd}
			\pt \otimes \cls U \ar[d, "x_0 \otimes \cls U"'] \ar[r] & \pt \ar[d] & \\
			F \otimes \cls U \ar[d] \ar[r] & \cofib(x_0 \otimes \cls U) \ar[d] \ar[r] & F \ar[d] \\
			\pt \otimes \cls U \ar[r] & \pt \ar[r] & i_* i^* F
		\end{tikzcd}
	,\]
	where the top square is coCartesian, and the left rectangle is coCartesian. It follows that the square in the bottom left corner is coCartesian, so the bottom rectangle is coCartesian if and only if the bottom right square is coCartesian. Thus, $D$ has gluing for $i$ at $F$ if and only if 
	\[
		\cofib(x_0 \otimes \cls U) \to F \to i_* i^* F
	\]
	is a cofibre sequence.
\end{rmk}

Now we show that the gluing property propagates along pointing. Note that when stabilizing, we generally start with a pullback formalism $D$ such that $D(S)$ has the Cartesian monoidal structure for all $S \in \mathcal C$, so that the monoidal structure  of $D_\bullet(S)$ is given by ``smash product''. In this case, for any map $f$ in $\mathcal C$, the functor $f^*$ automatically preserves terminal objects, as required in the following result:
\begin{lem} \label{lem:gluing for ptd}
	For any closed map $i : Z \to S$ in $\mathcal C$, if $D$ has gluing for $i$, and $i^* : D(S) \to D(Z)$ preserves terminal objects, then $D_\bullet$ has gluing for $i$.
	\begin{proof}
		Since $i^* : D(S) \to D(Z)$ preserves terminal objects, the induced functor $D(S)_{\pt/} \to D(Z)_{\pt/}$ is equivalent to $i^* : D_\bullet(S) \to D_\bullet(Z)$. It follows that the functor $D_\bullet(S) \simeq D(S)_{\pt/} \to D(S)$ sends the unit $\id \to i_* i^*$ in $D_\bullet(S)$ to the unit $\id \to i_* i^*$ in $D(S)$.

		Furthermore, for any $F \in D_\bullet(S)$, the map $F \otimes \cls{U} \to F$ in $D_\bullet(S)$ is sent to the map $(F \otimes \cls{U})/(\pt \otimes \cls{U}) \to F$ in $D(S)$. Now, since $\pt$ is a zero object in $D_\bullet(S)$, $\pt \otimes - \simeq \pt$, so we have that $D_\bullet$ has gluing for $i$ at $F$ if and only if
		\[
			F \otimes \cls U \to F \to i_* i^* F
		\]
		is a cofibre sequence.

		Note that by the usual description of the monoidal structure of $D_\bullet(S) \simeq D(S)_{\pt/}$ in terms of ``smashing'', and since $\cls{U;S}_{D_\bullet}$ is given by adjoining a basepoint to $\cls{U;S}_D$, we have that the functor $D_\bullet(S) \simeq D(S)_{\pt/} \to D(S)$ sends the above sequence to the sequence
		\[
			(F \otimes \cls U)/(\pt \otimes \cls U) \to F \to i_* i^* F
		.\]
		Since this functor is conservative and preserves cofibre sequences, it also reflects cofibre sequences, so we conclude by \Cref{rmk:gluing at ptd obj}.
	\end{proof}
\end{lem}

\subsubsection{Formal \texorpdfstring{$\otimes$}{⊗}-inversions}

Let us now recall the notions needed to consider formal $\otimes$-inversions. See \cite[\S4.1]{Fundamentals} for more details.
\begin{defn} \label{defn:otimes-gen}
	Given a symmetric monoidal category $\mathcal A \in \CAlg(\Cat)$, and a small collection of objects $A$ of $\mathcal A$, we say an object $a \in \mathcal A$ is \emph{$\otimes$-generated} by $A$ if for every symmetric monoidal functor $F : \mathcal A \to \mathcal B$, if $F$ sends every object of $A$ to a $\otimes$-invertible object of $\mathcal B$, then it sends $a$ to a $\otimes$-invertible object of $\mathcal B$.
\end{defn}

\begin{rmk}
	Although \Cref{defn:otimes-gen} is quite abstract, in applications, one usually verifies this condition by showing that there is some $a' \in \mathcal A$ which is some tensor product of elements of $A$, such that $a \otimes a'$ is also a tensor product of elements of $A$.
\end{rmk}

We fix a family $A = \{A_S\}_{S \in \mathcal C}$ such that for each $S \in \mathcal C$, $A_S$ is a small collection of objects in $D(S)$, and for any map $f : X \to Y$ in $\mathcal C$, the functor $f^* : D(Y) \to D(X)$ sends every object in $A_Y$ to an object that is $\otimes$-generated by $A_X$.

\begin{defn} \label{defn:A-lifts}
	Say a map $f : X \to Y$ in $\mathcal C$ \emph{has $A$-lifts} if every element of $A_X$ is $\otimes$-generated by $f^*(A_Y)$.

	Say an object $S \in \mathcal C$ is $A$-good if every quasi-admissible map to $S$ has $A$-lifts, and $A_S$ is $\otimes$-generated by the collection of symmetric objects in $A_S$ in the sense of \cite[Remark 2.20]{robalo} or \cite[Definition 4.1.11]{Fundamentals}.\footnote{An object $x$ of a symmetric monoidal category is said to be symmetric if for some $n \geq 2$, the cyclic permutation of $x^{\otimes n}$ is equivalent to the identity.}
\end{defn}

Now, suppose that $\mathcal G$ is a collection of $A$-good objects in $\mathcal C$ such that every object in $\mathcal C$ admits a $D$-pseudocover by quasi-admissible maps from objects of $\mathcal G$. Then \cite[Theorem 4.1.16]{Fundamentals} shows that there is a \emph{formal $\otimes$-inversion} $\Sigma_A^\infty : D \to D[A^{\otimes -1}]$ -- this is a morphism of pullback formalisms such that the slice projection $\PF(\mathcal C)_{D[A^{\otimes -1}]/} \to \PF(\mathcal C)_{D/}$ is fully faithful with essential image given by those morphisms $D \to E$ such that for every $S \in \mathcal C$, the functor $D(S) \to E(S)$ sends $A_S$ to $\Pic E(S)$.

\begin{prp} \label{prp:stabilize gluing}
	Suppose that $\mathcal K$ is a collection of closed maps in $\mathcal C$ that is stable under base change along quasi-admissible maps from objects of $\mathcal G$, and for any map $i \in \mathcal K$ with codomain in $\mathcal G$,
	\begin{enumerate}

		\item the domain of $i$ is in $\mathcal G$,

		\item $i$ is $D$-closed, and

		\item $i$ has $A$-lifts.

	\end{enumerate}
	Assume that for any object $S \in \mathcal G$, the category $D(S)$ has a zero object. Then every map in $\mathcal K$ is $D[A^{\otimes -1}]$-closed.
\end{prp}

The proof of \Cref{prp:stabilize gluing} will make use of the following notion, which is a dual version of the projection formula satisfied by quasi-admissible maps:
\begin{defn}
	A symmetric monoidal functor $i^* : \mathcal B \to \mathcal A$ is said to have a \emph{linear right adjoint} if it has a right adjoint $i_*$ such that the natural map
	\[
		i_* \otimes - \to i_*(- \otimes i^*)
	\]
	is an equivalence.
\end{defn}

\begin{lem} \label{lem:gluing when have right proj form}
	If $i : Z \to S$ is a closed map in $\mathcal C$ such that $i^* : D(S) \to D(Z)$ has a linear right adjoint, and $D(S)$ is pointed, then $D$ has gluing for $i$ if and only if
	\[
		j_\sharp(1) \to 1 \to i_*(1)
	\]
	is a cofibre sequence, where $j$ is a quasi-admissible complement of $i$.
	\begin{proof}
		Since $D(S)$ is pointed, we know that $D$ has gluing for $i$ if and only if for all $F \in D(S)$,
		\[
			\cls{U} \otimes  F \to F \to i_* i^* F
		\]
		is a cofibre sequence.

		Since $i_*$ is a linear right adjoint of $i^*$, this is equivalent to
		\[
			F \otimes (j_\sharp 1 \to 1 \to i_* 1)
		,\]
		so $j_\sharp 1 \to 1 \to i_* 1$ is a cofibre sequence if and only if $j_\sharp j^* F \to F \to i_* i^* F$ is a cofibre sequence for all $F \in D(S)$.
	\end{proof}
\end{lem}

\begin{proof}[Proof of \Cref{prp:stabilize gluing}]
	By \Cref{thm:D-topology} we have that every $D$-pseudocover consisting of quasi-admissible maps is also a $D[A^{\otimes-1}]$-pseudocover, so by \Cref{lem:gluing locality on base}, it suffices to show that every map in $\mathcal K$ to objects of $\mathcal G$ is $D[A^{\otimes -1}]$-closed. Thus, let $i : Z \to S$ be a map in $\mathcal K$ where $S \in \mathcal G$. It follows that $Z \in \mathcal G$, and that $i$ is $D$-closed and has $A$-lifts.

	Let $\mathcal C' \subseteq \mathcal C$ be the full anodyne pullback subcontext consisting of objects that admit quasi-admissible maps to objects of $\mathcal G$. It follows from \cite[Remark 4.1.15]{Fundamentals} that every object of $\mathcal C'$ is $A$-good, and by \cite[Proposition 4.1.13]{Fundamentals}, it suffices to show the result after replacing $\mathcal C$ by $\mathcal C'$, so that we may apply \cite[Proposition 4.1.12]{Fundamentals} to study the formal $\otimes$-inversion $D \to D[A^{\otimes -1}]$.

	Since $D(S)$ is pointed, and $i$ is $D$-closed, \cite[Theorem 4.19(2)]{6FF} shows that $i_* : D(Z) \to D(S)$ is a linear right adjoint of $i^* : D(S) \to D(Z)$. By \cite[Proposition 4.13]{6FF}, since $D(Z)$ and $D(S)$ have zero objects, we have that $i_* : D(Z) \to D(S)$ is colimit-preserving, so since $i$ has $A$-lifts, \cite[Proposition 4.1.12(5)]{Fundamentals} says that $i_* : D[A^{\otimes -1}](Z) \to D[A^{\otimes -1}](S)$ is a linear right adjoint of $i^*$, and the natural map $\Sigma_A^\infty i_* \to i_* \Sigma_A^\infty$ is an equivalence. 
	Since we also know that the natural map $j_\sharp \Sigma_A^\infty \to \Sigma_A^\infty j_\sharp$ is an equivalence, and $\Sigma_A^\infty$ preserves monoidal units, it follows that $\Sigma_A^\infty$ functor sends the sequence
	\[
		j_\sharp 1 \to 1 \to i_* 1
	\]
	in $D(S)$ to the same sequence in $D[A^{\otimes -1}](S)$. Since $D$ has gluing for $i$, this is a cofibre sequence in $D(S)$, so since $\Sigma_A^\infty : D(S) \to D[A^{\otimes -1}](S)$ is a colimit-preserving functor between categories that have zero objects, the same sequence in $D[A^{\otimes -1}](S)$ is also a cofibre sequence.

	Thus, we may conclude that $D[A^{\otimes - 1}]$ has gluing for $i$ by \Cref{lem:gluing when have right proj form}.

	Finally, we will show that $i_*$ is conservative. Using \cite[Proposition 4.1.12(2)]{Fundamentals}, and considering the proof of \cite[Proposition E.0.4]{Fundamentals} or \cite[\S6.1]{sixopsequiv}, we have that $i^* : D[A^{\otimes -1}](S) \to D[A^{\otimes -1}](Z)$ is a colimit of a diagram that sends each vertex to the functor $i^* : D(S) \to D(Z)$. Therefore, the right adjoint $i_* : D[A^{\otimes -1}](Z) \to D[A^{\otimes -1}](S)$ is a limit of a diagram that sends each vertex to the functor $i_* : D(Z) \to D(S)$. We know that this functor is fully faithful by \cite[Proposition 4.13]{6FF}, and limits of fully faithful functors are fully faithful (in particular, conservative), so $i_* : D[A^{\otimes -1}](Z) \to D[A^{\otimes -1}](S)$ is conservative.
\end{proof}

\section{Establishing Gluing for Unstable Homotopy Theories} \label{S:main gluing}

We think of pullback formalisms of the form $H^\tau$ as unstable homotopy theories. For example, unstable motivic homotopy theory is a pullback formalism of this form, and \cite[Theorem II.4.4.16]{TwAmb} shows that usual equivariant homotopy theory is also given by a pullback formalism of this form. By the results of \Cref{S:stabilizing}, in order to show gluing for \emph{stable} homotopy theories, it is sufficient to first show it for unstable ones.

In this section, we will explore more specialized techniques for establishing gluing for these types of pullback formalisms.

Fix a pullback context $\mathcal C$ with a \emph{strict initial object}, which is an initial object $\emptyset$ such that every map to $\emptyset$ is invertible.

We will make use of the following definition:
\begin{defn} \label{defn:fibre gluing}
	Let $i : Z \to S$ be a closed map in $\mathcal C$, and let $D$ be a pullback formalism on $\mathcal C$. For any $F \in D(S)$, and map $t : V \to i_* i^* F$ in $D(S)$, define
	\[
		\Theta_i(F,t) : \theta_i(F,t) \to V
	\]
	to be the base change of the coCartesian gap of $\square_i(F)$ along $t$. Sometimes we will abuse notation and write $\Theta_i(F,t) : \theta_i(F,t) \to V$ when $t$ is actually a map $i^* V \to i^* F$, by identifying it with its adjunct map $V \to i_* i^* F$.
\end{defn}

For LC pullback formalisms, it is easy to reduce the problem of showing gluing to the problem of showing that certain maps $\Theta_i(F,t)$ are equivalences using \Cref{lem:check equiv on fibres for LCPF}. In \Cref{S:red to fib gluing}, we will show that when $D$ is of the form $H^\tau$, we can still reduce to the study of objects $\theta_i(F,t)$. In fact, these results will reduce the problem of gluing in $H^\tau$ to the ``contractibility'' of certain presheaves $\hat \theta_i(F,t) \in \Psh(\mathcal C_{/S})$.

The main utility of this is that we have more tools available to manipulate the presheaves $\hat \theta_i(X,t)$. The most important one is given by \Cref{thm:gluing is adm invar}, which lets us work locally in an ``\'{e}tale neighbourhood'' of $t$ in $X$. We then use this result to reduce $X,t$ to cases where it is easy to establish that $\hat \theta_i(X,t)$ is contractible in the appropriate sense (\cf{} \cite[Lemma 4.17]{sixopsequiv} and \cite[Lemma II.5.2.11]{TwAmb}). 

Before coming to our main results, let us go over some basic properties of the maps $\Theta_i(F,t)$.

\begin{rmk} \label{rmk:pullback gluing}
	Let $\sigma : S' \to S$ be a map in $\mathcal C$, and let $i : Z \to S$ be a closed map in $\mathcal C_S$ such that the base change $i' : Z' \to S'$ of $i$ along $\sigma$ exists and is closed. Then the Beck-Chevalley transformations induce a map $\sigma^* \circ \square_i \to \square_{i'} \circ \sigma^*$. If $\sigma$ is quasi-admissible, then base change implies that this morphism is invertible.

	Furthermore, if $F \in D(S)$, and $t : V \to i_*i^* F$, we get a map $t' : \sigma^*(W) \to i'_* (i')^* \sigma^* F$ as the composite
	\[
		\sigma^*(t) \xrightarrow{\sigma^*(t)} \sigma^*(i_* i^* F) \to i'_* (i')^* \sigma^* F
	,\]
	When $\sigma$ is quasi-admissible, the second map is invertible, and we get an equivalence
	\[
		\sigma^* \Theta_i(F,t) \to \Theta_{i'}(\sigma^* F, t')
	.\]

	We will often simply write $\sigma^* t, \sigma^{-1}(i)$ instead of $t', i'$, so we have an equivalence
	\[
		\sigma^* \Theta_i(F,t) \xrightarrow{\sim} \Theta_{\sigma^{-1}(i)}(\sigma^* F, \sigma^* t)
	.\]
\end{rmk}

\begin{rmk} \label{rmk:gluing relative fundamental class}
	Let $D$ be a LC pullback formalism on $\mathcal C$. Then for any closed $i : Z \to S$, $F \in D(S)$, $t_0 : V \to i_* i^* F$, quasi-admissible $\sigma : S' \to S$, and $t : \sigma_\sharp(W) \to V$, we have that
	\[
		\Theta_i(F,t_0 \circ t) \simeq \sigma_\sharp \Theta_{i'}(\sigma^* F, t')
	,\]
	where $i'$ is the base change of $i$ along $\sigma$, and $t'$ corresponds to the map $W \to (i')_* (i')^* \sigma^* F \simeq \sigma^* (i_* i^* F)$ adjunct to $t_0 \circ t$.
\end{rmk}

\begin{lem} \label{lem:gluing locality and excision on base}
	Let $\phi : D \to D'$ be a morphism of pullback formalisms on $\mathcal C$. Let $i : Z \to S$ be a closed map in $\mathcal C$, and let $\{\sigma_k : S_k \to S\}$ be a quasi-admissible $D'$-pseudocover of $i$ (recall \Cref{defn:covers dense generators}).

	Using the notation of \Cref{rmk:pullback gluing}, for any $t : A \to i_* i^* F$ in $D(S)$, we have that $\phi \Theta_i(F,t)$ is invertible if and only if $\phi \Theta_{\sigma_k^{-1}(i)}(\sigma_k^* F, \sigma_k^* t)$ is invertible for all $k$.
	\begin{proof}
		Note that for any quasi-admissible $\sigma : S' \to S$, using the notation of Remark \ref{rmk:pullback gluing}, we have
		\[
			\sigma^* \phi \Theta_i(F,t) \simeq \phi \sigma^* \Theta_i(F,t) \simeq \phi \Theta_{\sigma^{-1}(i)}(\sigma^* F, \sigma^* t)
		,\]
		and in particular, if $\phi \Theta_{\sigma^{-1}(i)}(\sigma^* F, \sigma^* t)$ is invertible, so is $\sigma^* \phi \Theta_i(F,t)$.

		Thus, by hypothesis we have that for each $k$, $\sigma_k^* \phi \Theta_i(F,t)$ is invertible.

		Now, if $\sigma' : S' \to S$ is disjoint from $i$, then $\sigma'$ factors as
		\[
			S' \xrightarrow{\varsigma} S \setminus Z \xrightarrow{\sigma} S
		,\]
		where $\sigma$ is quasi-admissible since $i$ is closed. Note that $\sigma^* i$ is $\emptyset \to S \setminus Z$, so by \cite[Lemma 4.4]{6FF}, we have that $\Theta_{\sigma^{-1}(i)}(\sigma^* F, \sigma^* t)$ is invertible, and therefore
		\[
			(\sigma')^* \phi \Theta_i(F,t) \simeq \varsigma^* \sigma^* \phi \Theta_i(F,t) \simeq \varsigma^* \phi \Theta_{\sigma^{-1}(i)}(\sigma^*F, \sigma^* t)
		\]
		is invertible.

		Since we have a $D'$-pseudocover of $S$ given by $\{\sigma_k\}_k \cup i^\complement$, it therefore follows that $\phi \Theta_i(F,t)$ is invertible.
	\end{proof}
\end{lem}

\subsection{Reducing to Local Gluing} \label{S:red to fib gluing}

Let $\mathcal C$ be a small pullback context with a strict initial object, and let $\tau$ be a small quasi-admissible pseudotopology on $\mathcal C$ that is reduced -- \ie{} if $I$ is an initial object of $\mathcal C$, then $\emptyset \to \yo(I)$ is a $\tau$-acyclic pseudosieve, and any map to $I$ is invertible.

Before coming to the statement of our main result, we introduce the following notation:
\begin{nota} \label{nota:completed fibre gluing}
	Let $i : Z \to S$ be a map in $\mathcal C$ with complement $U \to S$. Recalling the notation of \Cref{exa:completed Huniv is LC}, we write
	\[
		\hat \square_i \coloneqq \square^{\hat H^\univ}_i \simeq \square^{\bar H^\univ_{\Psh(\mathcal C)}}_{\yo(i)}
	,\]
	which is the square of endofunctors in $\Psh(\mathcal C_{/S}) \simeq \Psh(\mathcal C)_{/\yo(S)}$ given by
	\[
		\hat \square_i(F) = \begin{tikzcd}
			F \times_{\yo(S)} \yo(U) \ar[d] \ar[r] & F \ar[d] \\
			F(\emptyset) \times \yo(U) \ar[r] & i_* (F \times_{\yo(S)} \yo(Z))
		\end{tikzcd}
	,\]
	where $i_* : \Psh(\mathcal C_{/Z}) \to \Psh(\mathcal C_{/S})$ is the right adjoint of
	\[
		\Psh(\mathcal C_{/S}) \simeq \Psh(\mathcal C)_{/\yo(S)} \xrightarrow{- \times_{\yo(S)} \yo(Z)} \Psh(\mathcal C)_{/\yo(Z)} \simeq \Psh(\mathcal C_{/Z})
	.\]

	Similarly, we define
	\[
		\hat \Theta_i : \hat \theta_i \to i_* (- \times_{\yo(S)} \yo(Z))
	\]
	to be the coCartesian gap of $\hat \square_i$, and for any map $t : V \to i_* i^* F$ in $\Psh(\mathcal C_{/S})$, we write
	\[
		\hat \Theta_i(F,t) : \hat \theta_i(F,t) \to V
	\]
	for the base change of $\hat \Theta_i$ along $t$.

	For any $X \in \mathcal C_{/S}$, a map $t : Z \to X$ over $S$ can be seen as a map $t : \pt \to i_* i^* X$ in $\hat H^\univ(S)$, and we write $\hat \theta_i(X,t)$ for $\theta^{\hat H^\univ}_i(\yo(X),t)$, so
	\[
		\hat \theta_i(X,t) \simeq \hat \theta_i(X) \times_{i_*(\yo(X) \times_{\yo(S)} \yo(Z))} \pt
	.\]
\end{nota}

\begin{rmk}
	It is not really necessary for $\mathcal C$ to be small in order to make the definitions of \Cref{nota:completed fibre gluing}.
\end{rmk}

\begin{defn}
	Given a $S \in \mathcal C$, a presheaf in $\Psh(\mathcal C_{/S})$ is said to be \emph{$\tau$-locally contractible} if it is $\tau$-locally equivalent to a terminal presheaf.
\end{defn}

\begin{rmk} \label{rmk:reduce by anodyne}
	Let $F : \mathcal C' \to \mathcal C$ be an anodyne morphism of small pullback contexts, \ie{} it is a functor between pullback contexts that preserves quasi-admissible maps and base changes along quasi-admissible maps, and such that for any $S' \in \mathcal C'$, the induced functor $\mathcal C'_{S'} \to \mathcal C_{F(S')}$ is an equivalence.

	Using \cite[Lemmas 2.1.13 and 2.2.4]{Fundamentals}, there is a canonical quasi-admissible pseudotopology $\tau'$ on $\mathcal C'$ such that the morphism
	\[
		H^{\tau'} \to F^* H^\tau \coloneqq H^\tau \circ F^\op
	\]
	is an equivalence.

	In particular, if $i : Z \to S$ is a closed map in $\mathcal C'$, such that $F(i)$ is also closed in $\mathcal C$, then $H^\tau$ has gluing for $F(i)$ if and only if $H^{\tau'}$ has gluing for $i$.

	Thus, in all of our results, instead of considering $\tau$-local contractibility in categories of the form $\Psh(\mathcal C_{/F(S)})$, it suffices to show $\tau'$-local contractibility in categories of the form $\Psh(\mathcal C'_{/S})$.
\end{rmk}

\begin{thm} \label{thm:global pseudotop gluing}
	Let $\mathcal G \subseteq \mathcal C$ be a full subcategory, and let $\mathcal K$ be a collection of closed maps in $\mathcal C$. Assume the following:
	\begin{enumerate}

		\item $\mathcal G$ satisfies that
			\begin{enumerate}

				\item for any quasi-admissible map $X \to S$, if $S \in \mathcal G$, then $X \in \mathcal G$, and

				\item every object of $\mathcal C$ admits a $H^\tau$-pseudocover consisting of quasi-admissible maps from objects of $\mathcal G$.

			\end{enumerate}

		\item Any map $i : Z \to S$ in $\mathcal K$ satisfies that
			\begin{enumerate}

				\item any base change of $i$ along a quasi-admissible map is still in $\mathcal K$, and

				\item if $S \in \mathcal G$, then
					\begin{enumerate}

						\item the functor $i_* : H^\univ(Z) \to H^\univ(S)$ preserves $\tau$-local equivalences between reduced presheaves, and

						\item for any quasi-admissible maps $X \to S$, if $X \in \mathcal G$, then for any $t : Z \to X$ over $S$, the presheaf
							\[
								\hat \theta_i(X, t) \in \Psh(\mathcal C_{/S})
							\]
							is $\tau$-locally contractible.

					\end{enumerate}

			\end{enumerate}

	\end{enumerate}
	Then $H^\tau$ has gluing for all $i \in \mathcal K$.
\end{thm}

This result follows easily from \Cref{prp:cover is pseudocover} and the following more general result:

\begin{prp} \label{prp:global pseudotop gluing}
	Let $\mathcal G \subseteq \mathcal C$ be a full subcategory, and let $\mathcal K$ be a collection of closed maps in $\mathcal C$ such that for any $i : Z \to S$ in $\mathcal K$, the following conditions hold:
	\begin{enumerate}

		\item Any base change of $i$ along a quasi-admissible map is still in $\mathcal K$.

		\item The terminal object of $\Psh(\mathcal C_S)$ is $\tau$-locally equivalent to a small colimit of presheaves that are $\tau$-locally equivalent to presheaves of the form $\yo(X)$ such that $X \to S$ is a quasi-admissible map and either $X \in \mathcal G$ or $X \times_S Z$ is initial.

		\item If $S \in \mathcal G$, then
			\begin{enumerate}

				\item the functor $i_* : H^\univ(Z) \to H^\univ(S)$ preserves $\tau$-local equivalences between reduced presheaves, and

				\item for any quasi-admissible maps $X,S' \to S$, if $X,S' \in \mathcal G$, then for any $t$ corresponding to a map $Z \times_S S' \to X$ over $S$, the presheaf
					\[
						\hat \theta_{i \times_S S'}(X \times_S S', t) \in \Psh(\mathcal C_{/S'})
					\]
					is $\tau$-locally contractible.

			\end{enumerate}

	\end{enumerate}
	Then $H^\tau$ has gluing for all $i \in \mathcal K$.
\end{prp}

The main ingredients needed to prove \Cref{prp:global pseudotop gluing} are \Cref{prp:intrinsic gens gluing}, and the following:
\begin{prp} \label{prp:pseudotop gluing}
	Let $i : Z \to S$ be a closed map in $\mathcal C$, and let $X \in \mathcal C_S$. Suppose that for every quasi-admissible map $\sigma : S' \to S$, and map $t : Z \times_S S' \to X$ over $S$, the presheaf
	\[
		\hat \theta_{i \times_S S'}(X \times_S S',t) \in \Psh(\mathcal C_{/S'})
	\]
	is $\tau$-locally contractible. Then $L_\tau \square_i(\cls{X})$ is coCartesian.
\end{prp}

In some cases when we need to have more control of the map $X \to S$, it may be useful to directly use \Cref{prp:pseudotop gluing} rather than applying \Cref{thm:global pseudotop gluing} or even \Cref{prp:global pseudotop gluing}.

\begin{proof}[Proof of \Cref{prp:global pseudotop gluing} using \Cref{prp:pseudotop gluing}]
	By \Cref{prp:intrinsic gens gluing}, it suffices to show that for any quasi-admissible map $X \to S$ where $X,S \in \mathcal G$, and map $i : Z \to S$ in $\mathcal K$, the square $L_\tau \square_i(\yo X)$ is coCartesian.

	By \Cref{prp:pseudotop gluing}, it suffices to show that for every quasi-admissible map $\varsigma : S'' \to S$, and map $t' : Z \times_S S'' \to X$ over $S$, the presheaf
	\[
		\hat \theta_{i \times_S S''}(X \times_S S'', t') \in \Psh(\mathcal C_{/S''})
	\]
	is $\tau$-locally contractible. By \Cref{prp:cover is pseudocover,thm:D-topology}, we have a $\hat H^\tau$-pseudocover of $i \times_S S''$ consisting of quasi-admissible maps $\sigma' : S' \to S''$ such that $S' \in \mathcal G$, so by \Cref{lem:gluing locality and excision on base}, it suffices to show that for each such $\sigma'$,
	\[
		\hat \theta_{i \times_S S'}(X \times_S S', t) \in \Psh(\mathcal C_{/S'})
	\]
	is $\tau$-locally contractible, where $t = \sigma'^*(t')$.
\end{proof}

\Cref{prp:pseudotop gluing} follows from the following more general result:
\begin{lem} \label{lem:pseudotop gluing}
	Let $i : Z \to S$ be a closed map in $\mathcal C$, and let $F \in \Psh(\mathcal C_S) \simeq H^\univ(S)$. Suppose that there is a small diagram $\{F_a\}_a$ in $\Psh(\mathcal C_S)$, and a \emph{universal}\footnotemark{} $\tau$-local equivalence
	\footnotetext{\ie{} all base changes of this map are still $\tau$-local equivalences. For example, this always holds if this map is actually an equivalence, or more generally if it is a $\gamma$-local equivalence for some Grothendieck topology $\gamma$ contained in $\tau$.}
	\[
		\varinjlim_a F_a \to i_* i^* F
	\]
	in $\Psh(\mathcal C_S)$ such that for each $a$, $F_a \in H^\univ(S) \subseteq \hat H^\univ(S)$ is equivalent to $(\sigma_a)_\sharp(M_a)$ for some\footnotemark{} map $\sigma_a : S_a \to S$, and object $M_a \in \Psh(\mathcal C_{/S_a}) \simeq \bar H^\univ_{\Psh(\mathcal C)}(\yo(S_a))$.
	\footnotetext{The functor $(\sigma_a)_\sharp$ is computed for the pullback formalism $\bar H^\univ_{\Psh(\mathcal C)}$ on $\Psh(\mathcal C)$, where all maps are quasi-admissible. In particular, $(\sigma_a)_\sharp$ is simply the slice projection $\Psh(\mathcal C)_{/\yo(S_a)} \to \Psh(\mathcal C)_{/\yo(S)}$, so $\sigma_a$ does not need to be a quasi-admissible map of $\mathcal C$.}

	Suppose that for each $a$, the base change $i_a$ of $i$ along $\sigma_a$ exists. The composite
	\[
		(\sigma_a)_\sharp(M_a) \simeq I F_a \to I \varinjlim_a F_a \to I i_* i^* F \to i_* i^* I F
	\]
	corresponds to a map
	\[
		t_a : M_a \to \sigma_a^* i_* i^* I F \simeq (i_a)_* i_a^* I \sigma_a^* F 
	.\]

	Suppose that for each $a$, the map of presheaves
	\[
		\hat \Theta_{i_a}(I \sigma_a^* F, t_a) : \hat \theta_{i_a}(I \sigma_a^* F, t_a) \to M_a
	\]
	in $\Psh(\mathcal C_{/S_a})$ is a $\tau$-local equivalence. Then $L_\tau \square_i(F)$ is coCartesian.
\end{lem}

We will need to make some remarks before proving \Cref{lem:pseudotop gluing}.

Write $I : H^\univ \to \hat H^\univ$ for the unique morphism of pullback formalisms $H^\univ \to \hat H^\univ$.
\begin{rmk} \label{rmk:right BC Huniv}
	For any map $f : X \to Y$ in $\mathcal C$, and $M \in H^\univ(X)$, the Beck-Chevalley transformation
	\[
		I f_* M \to f_* I M
	\]
	in $\hat H^\univ(Y) \simeq \Psh(\mathcal C_{/Y})$ restricts to an equivalence on $\mathcal C_Y \subseteq \mathcal C_{/Y}$.
	\begin{proof}
		Write $R$ for a right adjoint $I$. Then the Beck-Chevalley transformation is given by the following composite:
		\[
			I f_* \to I f_* R I \simeq I R f_* I \to f_* I
		.\]
		The first arrow is an equivalence since $I : H^\univ(X) \to \hat H^\univ(X)$ is fully faithful, so $\id_{D(X)} \to R I$ is an equivalence, and the last arrow is an $R$-equivalence since the counit $I R \to \id$ is always an $R$-equivalence.
	\end{proof}
\end{rmk}

\begin{rmk} \label{rmk:restrict completed gluing}
	Let $i : Z \to S$ be a closed map in $\mathcal C$, and let $F \in H^\univ(S)$. Then
	\[
		\hat \Theta_i(I F)|_{\mathcal C_S} \simeq \Theta_i(F)
	.\]
	\begin{proof}
		By \Cref{rmk:right BC Huniv}, we know that the Beck-Chevalley map $I i_* F \to i_* I F$ restricts to an equivalence on $\mathcal C_S$.
		By \cite[Lemma F.5(2)]{TwAmb}, we have that the composite
		\[
			I \to I i_* i^* \to i_* I i^* \simeq i_* i^* I
		\]
		is equivalent to the unit map $I \to i_* i^* I$.
		Since the restriction functor $\hat H^\univ(S) \simeq \Psh(\mathcal C_{/S}) \to \Psh(\mathcal C_S) \simeq H^\univ(S)$ preserves products and pushouts, it follows that
		\[
			\hat \Theta_i(I F)|_{\mathcal C_S} \simeq \Theta_i(F)
		.\]
	\end{proof}
\end{rmk}

\begin{proof}[Proof of \Cref{lem:pseudotop gluing}]
	We need to show that $L_\tau \Theta_i(F)$ is an equivalence. Write
	\[
		R : \hat H^\univ(S) \simeq \Psh(\mathcal C_{/S}) \to \Psh(\mathcal C_S) \simeq H^\univ(S)
	\]
	for the restriction functor along $\mathcal C_S \subseteq \mathcal C_{/S}$. \Cref{rmk:restrict completed gluing} shows that $\Theta_i(F) \simeq R \hat \Theta_i(I F)$, so it suffices to show that $L_\tau R \hat \Theta_i(I F)$ is an equivalence.

	Consider the map $\lambda : Y \to i_* i^* I F$ given by the composite
	\[
		Y \coloneqq I \varinjlim_a F_a \to I i_* i^* F \to i_* I i^* F \simeq i_* i^* I F
	.\]
	Since $RI \simeq \id$, and by \Cref{rmk:right BC Huniv}, we have that $R \lambda$ is equivalent to the map $\varinjlim_a F_a \to i_* i^* F$.

	Since $R$ is limit-preserving, it preserves pullbacks, and since $\varinjlim_a F_a \to i_* i^* F$ is a universal $\tau$-local equivalence, we have that the base change $\zeta : X \to Y$ of $\hat \Theta_i(I F)$ along $\lambda$ is $L_\tau R$-equivalent to $\hat \Theta_i(I F)$, whence it suffices to show that $L_\tau R \zeta$ is an equivalence.

	As in \Cref{lem:res pseudotop}, we may view $\tau$ as a pseudotopology on $\mathcal C_S$ and $\mathcal C_{/S}$, so the $\tau$-acyclic maps in $\Psh(\mathcal C_S)$ and $\Psh(\mathcal C_{/S})$ are precisely the maps lying over $\tau$-acyclic maps in $\Psh(\mathcal C)$. By \cite[Proposition 5.5.4.15(4)]{htt} we have that the collections of $\tau$-local equivalences in $\Psh(\mathcal C_S)$ and $\Psh(\mathcal C_{/S})$ are precisely the strongly saturated classes generated by the $\tau$-acyclic maps, and these are precisely the collections of maps sent to equivalences by $L_\tau$ and $\hat L_\tau$. Thus, by \Cref{lem:res local equiv}, it suffices to show that $\hat L^\tau \zeta$ is an equivalence.

	By \Cref{exa:completed Huniv is LC}, $\hat H^\univ$ extends to a LC pullback formalism $\bar H^\univ_{\Psh(\mathcal C)}$ on $\Psh(\mathcal C)$ equipped with the quasi-admissibility structure consisting of all maps, and by \Cref{exa:completed invar PF}, $\hat L^\tau$ extends to a morphism of pullback formalisms on $\Psh(\mathcal C)$,\footnote{The only reason we need to extend our pullback formalisms to the pullback context $\Psh(\mathcal C)$ where all maps are quasi-admissible, is so that we can allow $\sigma_a$ to be a map that is not necessarily quasi-admissible with respect to the quasi-admissibility structure of $\mathcal C$. This level of generality is never needed in the present work, but we have decided to include it in case it may be useful in the future.}
	so \Cref{lem:check equiv on fibres for LCPF} says that to check that $\hat L_\tau \zeta$ is an equivalence, it suffices to show that for each $a$, we have that $\hat L_\tau(\zeta_a)$ is an equivalence in $\hat H^\tau(S_a)$, where $\zeta_a$ is the base change of $\sigma_a^* \zeta$ along the map $\bar t_a : M_a \to \sigma_a^* Y$ that is adjunct to the map
	\[
		(\sigma_a)_\sharp M_a \simeq I F_a \to I \varinjlim_a F_a = Y
	.\]
	Note that since $t_a$ given by precomposing $\sigma_a^* i_*i^* IF \simeq i_* i^* I \sigma_a^* F$ with the adjunct of the composite
	\[
		\sigma_\sharp(M_a) \simeq IF_a \to I \varinjlim_a F_a = Y \xrightarrow{\lambda} i_* i^* I F
	,\]
	it is equivalent to the following composite
	\[
		M_a \xrightarrow{\bar t_a} \sigma_a^* Y \xrightarrow{\sigma_a^* \lambda} \sigma_a^* i_* i^* I F \simeq (i_a)^* i_a^* I \sigma_a^* F
	.\]

	By \Cref{rmk:pullback gluing}, and since $\sigma_a^* I \simeq I \sigma_a^*$, we have an equivalence
	\[
		\sigma_a^* \hat \Theta_i(I F) \simeq \hat \Theta_{i_a}(\sigma_a^* I F) \simeq \hat \Theta_{i_a}(I \sigma_a^* F)
	.\]

	Thus, we have Cartesian squares
	\[
		\begin{tikzcd}
			\hat \theta_{i_a}(I \sigma_a^* F, t_a) \ar[d, "\zeta_a"'] \ar[r] & \sigma_a^* X \ar[d, "\sigma_a^* \zeta"] \ar[r] & \sigma_a^* \hat \theta_i(I F) \ar[d, "\sigma_a^* \hat \Theta_i(I F)"] \ar[r, no head, "\sim"] & \hat \theta_{i_a}(I \sigma_a^* F) \ar[d, "\hat \Theta_{i_a}(I \sigma_a^* F)"] \\
			M_a \ar[r, "\bar t_a"'] & \sigma_a^* Y \ar[r, "\sigma_a^* \lambda"'] & \sigma_a^* i_* i^* I F \ar[r, no head, "\sim"] & (i_a)_* i_a^* I \sigma_a^* F
		\end{tikzcd}
	,\]
	so $\zeta_a \simeq \hat \Theta_{i_a}(I \sigma_a^* F, t_a)$.
\end{proof}

\subsection{Excision} \label{S:adm excision}


In this section, $\mathcal C$ may be any small category that has a strict initial object, and is equipped with a Grothendieck topology $\gamma$, and collections of maps $\mathcal K, \Sigma$ such that diagonals of maps in $\Sigma$ exist, and $\Sigma$ is stable under taking diagonals.

Assume that for any Cartesian square
\[
	\begin{tikzcd}
		Z \ar[d, equals] \ar[r, "t'"] & X' \ar[d] \\
		Z \ar[r, "t"'] & X
	\end{tikzcd}
\]
in $\mathcal C$, if $t \in \mathcal K$ and $X' \to X$ is in $\Sigma$, then $t' \in \mathcal K$, and there is some covering sieve of $X$ generated by $X' \to X$ and maps $Y \to X$ that are disjoint from $t$.\footnote{This means that $Y \times_X Z$ is an initial object.}

\begin{exa}
	Suppose that $\mathcal C$ is the category of schemes, and $\gamma$ is the Nisnevich topology. We can let $\Sigma$ be the collection of all \'{e}tale morphisms between schemes, and $\mathcal K$ can be the collection of all closed immersions from $Z$.
\end{exa}

\begin{thm} \label{thm:gluing is adm invar}
	Suppose that $\gamma$ is reduced. Let $i : Z \to S$ be a map in $\mathcal C$ that admits a complement, and let
	\[
		\begin{tikzcd}
			Z \ar[d, equals] \ar[r, "t'"] & X' \ar[d] \\
			Z \ar[r, "t"'] & X
		\end{tikzcd}
	\]
	be a Cartesian square in $\mathcal C_{/S}$, where $X' \to X$ is in $\Sigma$, and $Z \to X$ is in $\mathcal K$. Also assume that $X,X',Z$ represent sheaves on $\mathcal C_{/S}$.

	Then the sheafification of
	\[
		\hat \theta_i(X', t') \to \hat \theta_i(X, t)
	\]
	is $\infty$-connective.

	If $X',X \to S$ are truncated maps, then this map is actually a $\gamma$-local equivalence.
	\begin{proof}
		Write $\jmath : \mathcal C_{/S} \to \Shv(\mathcal C_{/S})$ for the sheafified Yoneda embedding.


		Let $\tilde{\mathcal K}$ be the collection of all maps in $\Shv(\mathcal C_{/S})$ of the form $\jmath(f)$, where $f$ is a map $Z \to Y$ in $\mathcal K$ over $S$ such that $Y$ represents a sheaf on $\mathcal C_{/S}$. Similarly define $\tilde \Sigma$ to be the collection of all maps in $\Shv(\mathcal C_{/S})$ of the form $\jmath(f)$, where $f$ is a map $Y' \to Y$ in $\Sigma$, and $Y,Y'$ represent sheaves on $\mathcal C_{/S}$.


		Since objects that represent sheaves are stable under limits, it follows that $\tilde{\Sigma}$ is closed under taking diagonals, and $\tilde{\mathcal K}$ has the necessary stability under base changes along maps in $\tilde{\Sigma}$ for us to apply \Cref{lem:geom gluing connective} in the topos $\Shv(\mathcal C_{/S})$. We also use that $\gamma$ is reduced so that $\jmath$ preserves disjointness of maps.

		Since $\tilde{\Sigma}$ is closed under taking diagonals, it follows that, in the notation of \Cref{defn:topos gluing} where we think of $i$ as the geometric morphism $\Shv(\mathcal C_{/S})_{/\jmath(Z)} \to \Shv(\mathcal C_{/S})$, the map
		\[
			\theta_i(\jmath(X'), t') \to \theta_i(\jmath(X), t)
		\]
		is $\infty$-connective, and that it is an equivalence if $X',X \to S$ are truncated maps.

		Note that since $\gamma$ is reduced, and $Z$ represents a sheaf on $\mathcal C_{/S}$, \Cref{exa:compl of map} shows that $\jmath$ sends the complement of $Z$ to an open complement of the geometric morphism $\Shv(\mathcal C_{/Z}) \simeq \Shv(\mathcal C_{/S})_{/\jmath(Z)} \to \Shv(\mathcal C_{/S})$. Therefore, 
		using the fact that $X,X',Z$ represent sheaves on $\mathcal C_{/S}$, one concludes that the map above is given by sheafifying the map
		\[
			\hat \theta_i(X', t') \to \hat \theta_i(X, t)
		.\]
	\end{proof}
\end{thm}

\section{Applications to Motivic Homotopy Theories on Stacks} \label{S:examples}

We will now consider some applications of our general techniques for showing gluing. More specifically, we will be interested in the contexts of \cite[\S5]{Fundamentals}. The results for the algebraic and differentiable contexts of \cite[\S5.1 and \S5.2]{Fundamentals} have already been addressed in \cite{SixAlgSt, sixopsequiv} and \cite{TwAmb}, but we have chosen to include them to show how our general tools can be applied to diverse geometric contexts.

The general situation in which we are interested is the following: we have a small pullback context $\mathcal C$ equipped with a quasi-admissible pseudotopology $\tau$, and a collection of maps $\mathcal K$ for which we want to show that $H^\tau$ has gluing. Recall that $H^\tau$ is the pullback formalism that sends an object $S \in \mathcal C$ to the category $\Psh^\tau(\mathcal C_S)$ of $\tau$-invariant presheaves on $\mathcal C_S$. In fact, we also have some ``stabilized'' version $\SH^\tau$ of $H^\tau$ that is given by formally adjoining certain $\wedge$-inverses to the pointed version $H^\tau_\bullet$ of $H^\tau$, and we would like to show that maps in $\mathcal K$ are actually $\SH^\tau$-closed. Using \Cref{lem:gluing for ptd,prp:stabilize gluing}, this reduces to showing that maps in $\mathcal K$ are $H^\tau$-closed, that is, if $i : Z \to S$ is in $\mathcal K$, then $H^\tau$ has gluing for $i$, and $i_* : H^\tau(Z) \to H^\tau(S)$ is conservative.

The fact that $i_* : H^\tau(Z) \to H^\tau(S)$ is conservative is generally shown by working locally on $S$ using a suitable quasi-admissible $H^\tau$-pseudocover, and then giving some geometric argument showing roughly that any quasi-admissible map to $Z$ is ``$\tau$-locally'' given by the base change along $i$ of quasi-admissible maps to $S$. It follows that $H^\tau(Z)$ is generated under colimits by the essential image of $i^*$, so $i_*$ is conservative.

We are left with the task of showing that $H^\tau$ has gluing for maps in $\mathcal K$. For this, we pick a suitable collection $\mathcal G$ of well-behaved objects in $\mathcal C$ for which we can apply \Cref{thm:global pseudotop gluing}, allowing us to reduce to showing that for any $i : Z \to S$ in $\mathcal K$ with $S \in \mathcal G$, if $X \in \mathcal G$, $X \to S$ is quasi-admissible, and $t : Z \to X$ is a map over $S$, then $\hat \theta_i(X,t)$ is $\tau$-locally contractible (recall \Cref{nota:completed fibre gluing}).

The choice of $\mathcal G$ can sometimes be delicate, as it needs to be large enough for us to be able to apply \Cref{thm:global pseudotop gluing}, but small enough for us to be able to show that the presheaves $\hat \theta_i(X,t)$ are $\tau$-locally contractible. In some cases, it may be difficult to find a suitable collection $\mathcal G$, and we may use a different combination of results from \Cref{S:red to fib gluing}, but we always end up reducing to the $\tau$-local contractibility of presheaves of the form $\hat \theta_i(X,t)$. One example of a more delicate reduction of this form is given in our proof for the algebraic case in \Cref{thm:alg gluing}.

When $X \to S$ is a quasi-admissible map between objects of $\mathcal G$, $i : Z \to S$ is in $\mathcal K$, and $t : Z \to X$ is a map in $\mathcal C_{/S}$, to show that $\hat \theta_i(X,t)$ is $\tau$-locally contractible, we proceed as follows:
\begin{enumerate}

	\item Using the geometric properties of maps in $\mathcal K$ to objects of $\mathcal G$, we can show that locally near $i$, $t$ factors through a section of $X \to S$. Thus, \Cref{lem:gluing locality and excision on base} lets us reduce to showing that $\hat \theta_i(X,t)$ is $\tau$-locally contractible when $t$ factors through a section of $X \to S$.

	\item Next, we use the properties of objects in $\mathcal G$ once again to show that we can relate $X \to S$ to some $V \to S$ via suitable ``\'{e}tale'' maps, where $V \to S$ is a pointed object of $\mathcal C_S$ that is a ``strict $\tau$-homotopy equivalence relative the section $0 : S \to V$'' (generally this is because $V \to S$ is a type of vector bundle). Then \Cref{thm:gluing is adm invar} shows that $\hat \theta_i(X,t)$ is $\tau$-locally contractible if and only if $\hat \theta_i(V, 0 \circ i)$ is $\tau$-locally contractible.\footnote{Here, the ``\'{e}tale'' maps should be given by some collection of quasi-admissible maps all of whose iterated diagonals are also quasi-admissible, and such that $\tau$ contains some notion of excision with respect to \'{e}tale maps, as described in \Cref{S:adm excision}.}

	\item The fact that $\hat \theta_i(V, 0 \circ i)$ is $\tau$-locally contractible can be shown using an explicit null-homotopy (\cf{} \cite[Lemma 4.17]{sixopsequiv} and \cite[Lemma II.5.2.11]{TwAmb}).

\end{enumerate}

\subsection{Differentiable Motivic Homotopy}

In this section we will consider the motivic homotopy theory of differentiable stacks as introduced in \cite{TwAmb} and discussed in \cite[\S5.2]{Fundamentals}.

We write $\Mfld$ for the category of manifolds. This is equipped with the Grothendieck topology of open coverings, and $\Shv(\Mfld)$ is equipped with the structure of a pullback context, as well as a quasi-admissible pseudotopology $\tau_\diff$ allowing us to define the pullback formalism $H^\diff \coloneqq H^{\tau_\diff}$ on $\Shv(\Mfld)$, and for suitable pullback subcontexts $\DiffStk^\prpr \subseteq \Shv(\Mfld)$, we have a stabilized version $\SH^\diff$.

Our goal in this section is to prove \Cref{thm:diff gluing}, which shows that closed embeddings in $\DiffStk^\prpr$ are $H^\diff$-closed and $\SH^\diff$-closed.

We now recall and supplement the relevant notions from \cite[\S5.2]{Fundamentals} (see there or \cite[Part II]{TwAmb} for more details):
\begin{description}

	\item[Quasi-admissibility structure] The category $\Shv(\Mfld)$ is given the structure of a quasi-small pullback context where the quasi-admissible maps are the representable submersions.

	\item[Pseudotopologies] We have the following small quasi-admissible pseudotopologies on $\Shv(\Mfld)$:
		\begin{description}

			\item[Open covers] An \emph{open cover} in $\Shv(\Mfld)$ is given by a family of open embeddings $\{U_i \to X\}_i$ that pull back to an open cover of $X'$ for any manifold $X'$ and map $X' \to X$. This defines the Grothendieck topology $\tau_{\open^\diff}$ on $\Shv(\Mfld)$.

			\item[Homotopy invariance] The quasi-admissible pseudotopology $\tau_\reals$ on $\Shv(\Mfld)$ has acyclic pseudosieves given by product projections $\yo(\reals \times X \to X)$.

			\item[Motivic equivalence] The quasi-admissible pseudotopology $\tau_\diff$ on $\Shv(\Mfld)$ is the union $\tau_\reals \cup \tau_{\open^\diff}$.

		\end{description}
		
	\item[Unstable homotopy theory] We write $H^\diff$ to denote $H^{\tau_\diff}$, as well as its restriction to any full anodyne pullback subcontext of $\Shv(\Mfld)$. This is a pullback formalism that sends $S \in \Shv(\Mfld)$ to the category of $\reals$-invariant sheaves on the site of representable submersions over $S$.

	\item[Differentiable stacks] Define $\DiffStk^\prpr$ to be the full subcategory of $\Shv(\Mfld)$ consisting of those objects that admit open covers by global quotients of the form $M/G$ for $G$ a compact Lie group acting on a manifold $M$. Note that if $X \to Y$ is a representable submersion, and $Y \in \DiffStk^\prpr$, then $X \in \DiffStk^\prpr$, so $\DiffStk^\prpr$ is a full anodyne pullback subcontext of $\Shv(\Mfld)$.

	\item[Stable homotopy theory] Following \cite[\S II.4.3]{TwAmb}, \cite[Theorem 5.2.9]{Fundamentals} constructs the stabilized version $\SH^\diff$ of $H^\diff$ on $\DiffStk^\prpr$ characterized by the condition that the unique map $H^\diff \to \SH^\diff$ is initial among pointed pullback formalisms $D$ under $H^\diff$ satisfying that Thom spaces of vector bundles over any $S \in \DiffStk^\prpr$ are sent to $\otimes$-invertible objects in $D(S)$.

\end{description}

\begin{defn}
	A map $X \to Y$ in $\Shv(\Mfld)$ is a \emph{closed embedding} if for any representable submersion $N \to Y$, where $N$ is a manifold, the base change $X \times_Y N \to N$ is a closed embedding of manifolds.
\end{defn}

\begin{rmk}
	Closed embeddings in $\DiffStk^\prpr$ are closed under quasi-admissible base change, and are closed in the sense that they have quasi-admissible complements.
\end{rmk}

Our goal is to prove the following, which recovers \cite[Theorem II.5.2.14]{TwAmb}:
\begin{thm} \label{thm:diff gluing}
	Any embedding in $\DiffStk^\prpr$ is $H^\diff$-closed and $\SH^\diff$-closed.
\end{thm}

\begin{lem} \label{lem:diff cl emb}
	Let $i : Z \to S$ be a closed embedding in $\DiffStk^\prpr$. Then $i_* : H^\univ(Z) \to H^\univ(S)$ preserves $\tau_\diff$-local equivalences between reduced presheaves.
	\begin{proof}
		Note that $\tau_\reals$-local equivalences are always generated by $\reals$-indexed homotopy equivalences, and it is clear that $i_*$ preserves these.

		It remains to show that $i_*$ sends $\tau_{\open^\diff}$-local equivalences to $\tau_\diff$-local equivalences. Indeed, for any open cover of $Z$, \cite[Corollary II.3.2.3]{TwAmb} says that we have an open cover of $i(Z)$ in $S$ whose pullback along $i$ is the original open cover of $Z$, so we may conclude by \cite[Lemma 3.1.6]{locspalg}.
	\end{proof}
\end{lem}

\begin{lem} \label{lem:cl emb conservative}
	Let $i : Z \to S$ be a closed embedding in $\DiffStk^\prpr$, and let $\{X_j \to Z\}_j$ be a $H^\diff$-pseudocover such that for each $j$, the map $X_j \to Z$ is the base change along $i$ of some representable submersion $M/G \to S$, where $G$ is a Lie group acting properly\footnotemark{} (and smoothly) on a manifold $M$. Then $i_* : H^\diff(Z) \to H^\diff(S)$ is conservative.
	\footnotetext{This is automatic if $G$ is compact.}
	\begin{proof}
		By quasi-admissible base change, we may reduce to the case that $S = M/G$ for $G$ a Lie group acting properly on a manifold $M$. It suffices to show that for any quasi-admissible map $X \to Z$, there is a quasi-admissible $\tilde X \to S$ whose base change along $i$ is $X \to S$. The equivariant tubular neighborhood Theorem \cite[Theorem 4.4]{equivdiff} says that $i$ factors as
		\[
			Z \to N \to S
		,\]
		where $Z \to N$ is the zero section of the normal bundle of $i$, and $N \to S$ is a representable open embedding. In particular, for any quasi-admissible map $\tilde X \to N$, we have that $\tilde X \to S$ is quasi-admissible and $\tilde X \times_S N = \tilde X$, so it suffices to replace $S$ by $N$. Therefore, can assume that $i$ has a retraction, and we set $\tilde X \to S$ to be the base change of $X \to S$ along this retraction.
	\end{proof}
\end{lem}

\begin{proof}[Proof of \Cref{thm:diff gluing}]
	First we will show that closed embeddings are $H^\diff$-closed.
	Recalling the construction of $\SH^\diff$ as the stabilization of $H^\diff$, it will then follow from \Cref{lem:gluing for ptd}, \Cref{prp:stabilize gluing}, and \cite[Proposition II.4.3.9]{TwAmb}, that all closed embeddings are also $\SH^\diff$-closed.

	\Cref{lem:cl emb conservative} shows that if $i$ is a closed embedding, then $i_*$ is conservative, so it only remains to show that $H^\diff$ has gluing for closed embeddings.

	By \Cref{thm:global pseudotop gluing,lem:diff cl emb}, we may reduce to showing that if $G$ is a compact Lie group, $Z,S,X$ are $G$-manifolds, $i : Z \to S$ is a $G$-equivariant closed embedding, $X \to S$ is a $G$-equivariant submersion, and $t : Z \to X$ is a $G$-equivariant map over $S$, then $\hat \theta_{i/G}(X/G,t/G) \in \Psh(\DiffStk^\prpr_{/(S/G)})$ is $\tau_\diff$-locally contractible.


	Note that $t$ is necessarily injective, so for any $z \in Z$, since $t$ and $i$ are injective, the map $X \to S$ induces an isomorphism from the stabilizer of $t(z)$ to the stabilizer of $i(z)$. As in \cite[Step 2 of the proof of Proposition II.3.7.5]{TwAmb}, we may use \cite[Proposition 2.7]{Pflaum_2019} to obtain a $G$-invariant open cover of $i(Z)$ in $S$ where $p$ admits an equivariant section that extends $t$. Thus, there is a quasi-admissible $H^\diff$-pseudocover of $i/G$ on which $t/G$ extends through $i/G$ by a section of $(X \to S)/G$. Hence, by \Cref{lem:gluing locality and excision on base}, it suffices to assume $t = si$, where $s : S \to X$ is a section of $X \to S$.

	Note that $s$ is automatically a closed embedding, so by the equivariant tubular neighbourhood theorem \cite[Theorem 4.4]{equivdiff}, we have that $s$ factors as the zero section $0 : S \to V$ of a $G$-equivariant vector bundle on $S$, followed by a $G$-invariant open embedding $V \to X$. By \Cref{thm:gluing is adm invar}, we have that $\hat \theta_{i/G}(V/G, 0/G \circ i/G) \to \hat \theta_{i/G}(X/G, t/G)$ is a $\tau_{\open^\diff}$-local equivalence, so it is also a $\tau_\diff$-local equivalence. This is seen to be $\tau_\diff$-locally contractible by means of an explicit nullhomotopy (\cf{} \cite[Lemma 4.17]{sixopsequiv} and \cite[Lemma II.5.2.11]{TwAmb}).
\end{proof}

\subsection{Algebraic Motivic Homotopy}
Let $\AlgStk$ be the category of qcqs algebraic stacks. The category $\AlgStk$ is equipped with some quasi-admissibility structure making it a quasi-small pullback context, and where all quasi-admissible maps are representable smooth morphisms, and all smooth affine morphisms are quasi-admissible.

As in \cite[\S5.1]{Fundamentals}, we define a pseudotopology $\tau_\alg$ on $\AlgStk$, where the $\tau_\alg$-acyclic pseudosieves are the pseudosieves that are either Nisnevich covering sieves generated by quasi-admissible maps, or vector bundle torsors. We write
\[
	H^\alg \coloneqq H_{\AlgStk}^{\tau_\alg}
.\]
This is the pullback formalism that sends an algebraic stack $S$ to the category of sheaves $F$ on the Nisnevich site $\AlgStk_S$ of stacks smooth over $S$, which are also homotopy invariant in the sense that if $V \to X$ is a vector bundle torsor, then $F(X) \to F(V)$ is an equivalence.

\begin{rmk}
	For the sake of brevity, we only consider the case of \emph{classical stacks}, but our techniques also work in the derived setting: see \cite[Theorem 3.15 and its proof]{SixAlgSt}.
\end{rmk}

We will mainly be interested in the following full anodyne pullback subcontexts of $\AlgStk$:
\begin{exa} \label{exa:alg}
	\hfill
	\begin{enumerate}

		\item $\AlgStk^\nice \subseteq \AlgStk$ is the full subcategory consisting of (qcqs) algebraic stacks with separated diagonal whose stabilizers are nice groups in the sense of \cite[Definition 2.1(i)]{SixAlgSt}, which we recall here for convenience: an fppf affine group scheme $G$ over an affine scheme $S$ is \emph{nice} if it is an extension of a finite \'{e}tale group scheme of order prime to the residue characteristic of $S$, by a group scheme of multiplicative type.

		\item $\AlgStk^\lred \subseteq \AlgStk$ is the full subcategory consisting of (qcqs) algebraic stacks that admit quasi-projective Nisnevich covers by global quotients of the form $X/G$, where $G$ is a linearly reductive group scheme over an affine scheme, and $X/G \to \B G$ is quasi-projective.

	\end{enumerate}
	Note that since the quasi-admissible maps of $\AlgStk$ are and representable, $\AlgStk^\nice$ is a full anodyne pullback subcontext of $\AlgStk$. If the quasi-admissible morphisms of $\AlgStk$ are all quasi-projective, then $\AlgStk^\lred$ is a full anodyne pullback subcontext of $\AlgStk$.

	Also note that by \cite[Theorem 2.12]{SixAlgSt}, the category $\AlgStk^\nice$ is precisely the category of \emph{nicely scalloped algebraic stacks} of \cite{SixAlgSt}.
\end{exa}

If $\mathcal C \subseteq \AlgStk$ is a full anodyne pullback subcontext, \cite[Theorem 5.1.11]{Fundamentals} (or \cite{SixAlgSt}) constructs the stabilized version $\SH^\alg$ of $H^\alg$ in the following cases: 
\begin{enumerate}

	\item The quasi-admissible maps in $\AlgStk$ are the representable smooth morphisms, $\mathcal C \subseteq \AlgStk^\nice$.

	\item The quasi-admissible maps in $\AlgStk$ are the quasi-projective smooth morphisms, and $\mathcal C \subseteq \AlgStk^\lred$.

\end{enumerate}

Under these hypotheses, we have the following:
\begin{thm} \label{thm:alg gluing}
	All closed immersions in $\mathcal C$ are $H^\alg$-closed and $\SH^\alg$-closed.
\end{thm}

The proof of \Cref{thm:alg gluing} will be given at the end of the section.

\begin{lem} \label{lem:alg loc str}
	Suppose $S \in \AlgStk$ admits a quasi-projective morphism $S \to \B G$ for some linearly reductive group scheme $G$ over an affine scheme. Then $S$ admits an affine Nisnevich cover $\{S_i \to S\}_i$ such that for any indices $i_1, \dotsc, i_n$ with $n \geq 1$, there is an \emph{embeddable} linearly reductive group scheme $G_{i_1}$ over an affine scheme, a quasi-projective morphism $S_{i_1, \dotsc, i_n} \coloneqq S_{i_1} \times_S \dotsb \times_S S_{i_n} \to \B G_{i_1}$, and a vector bundle torsor $V \to S_{i_1, \dotsc, i_n}$ such that $V \to \B G_{i_1}$ is affine.

	In particular, there is a small diagram $\tilde S : K \to \Psh(\AlgStk_{/S})$ such that
	\begin{enumerate}

		\item $\varinjlim \tilde S \to \yo(S)$ is a $\tau_\alg$-local equivalence, and

		\item for each $a$, $\tilde S(a) \in \Psh(\AlgStk_{/S})$ is $\tau_\alg$-locally equivalent to $\yo$ of a smooth affine morphism $S_a \to S$, such that

		\item there exists an affine morphism $S_a \to Y_a/G = \B(Y_a \times_B G)$, where $Y_a$ is an affine $G$-scheme such that $G$ is embeddable over $Y_a$ (and linearly reductive).

	\end{enumerate}
	Furthermore, $S$ admits an $H^\alg$-pseudocover by affine smooth maps from stacks that admit affine morphisms to stacks of the form $\B G'$, where $G'$ is an embeddable linearly reductive group scheme over an affine scheme.
	\begin{proof}
		Write $B$ for the affine base scheme of $G$. Using \cite[Corollary 6.2]{AHR}, we can find an affine Nisnevich cover $\{B_i \to B\}_i$ of $B$ such that $G_i \coloneqq G \times_B B_i$ is embeddable (and linearly reductive). As in \cite[Remark 6.3 or 2.5]{AHR}, we have that for each $i$, $B_i$ has the $G$-resolution property of \cite[Definition 2.7]{sixopsequiv}.

		Since $S \to \B G$ is quasi-projective, \cite[Lemma A.0.8]{Fundamentals} shows that it is of the form $(X \to B)/G$, where $X \to B$ is $G$-quasi-projective in the sense of \cite[Definition 2.5]{sixopsequiv}. In particular, for each $i$, $S_i \coloneqq S \times_{\B G} \B G_i \to \B G_i$ is of the form $(X \times_B B_i \to B_i)/G_i$, and $X_i \coloneqq X \times_B B_i \to B_i$ is $G$-quasi-projective as a morphism of $G$-schemes (since it is $G_i$-quasi-projective as a morphism of $G_i$-schemes).

		Therefore, \cite[Lemma 2.13]{sixopsequiv} shows that for indices $i_1, \dotsc, i_n$ with $n \geq 1$, we also have that $X_{i_1, \dotsc, i_n} \coloneqq X_{i_1} \times_X \dotsb \times_X X_{i_n}$ admits a $G$-quasi-projective morphism to $B_{i_1}$, and therefore \cite[Proposition 2.20]{sixopsequiv} shows that there is a vector bundle torsor $V \to S_{i_1, \dotsc, i_n} = X_{i_1, \dotsc, i_n}/G$ such that $V \to B_{i_1}/G = \B G_{i_1}$ is affine.

		The statement about colimits in $\Psh(\AlgStk_{/S})$ follows form the fact that affine Nisnevich covering sieves and vector bundle torsors give $\tau_\alg$-acyclic pseudosieves, and by taking \v{C}ech nerves. The final statement then follows from \Cref{prp:cover is pseudocover}.
	\end{proof}
\end{lem}

In view of \Cref{exa:alg,lem:alg loc str}, we fix a full subcategory $\mathcal C^\alg \subseteq \AlgStk$ such that every object of $\mathcal C^\alg$ admits an $H^\alg$-pseudocover by quasi-admissible maps from stacks that admit quasi-projective morphisms to stacks of the form $\B G$, where $G$ is a linearly reductive group scheme over an affine scheme. \Cref{exa:alg} shows that $\mathcal C^\alg$ can be taken to be any full subcategory of $\AlgStk^\nice$ or $\AlgStk^\lred$, and also describes when it actually defines an anodyne pullback subcontext of $\AlgStk$.

\Cref{lem:alg loc str} immediately shows that any object of $\mathcal C^\alg$ actually admits an $H^\alg$-pseudocover by quasi-admissible maps from stacks that admit \emph{affine} morphisms to stacks of the form $\B G$, where $G$ is an \emph{embeddable} linearly reductive group scheme over an affine scheme.

\begin{lem} \label{lem:alg loc gluing}
	Let $X \to S$ be an affine quasi-admissible morphism, and let $i : Z \to S$ be a closed immersion.

	If $S \in \mathcal C^\alg$, then for any map $t : Z \to X$ over $S$, we have that
	\[
		\hat \theta_i(X,t) \in \Psh(\AlgStk_{/S})
	\]
	is $\tau_\alg$-locally contractible.
	\begin{proof}
		By \Cref{lem:gluing locality and excision on base}, it suffices to assume that $S$ admits an affine morphism to $\B G$, where $G$ is some embeddable linearly reductive group scheme over an affine scheme. \Cref{thm:gluing is adm invar} implies that it suffices to replace $X$ by any affine \'{e}tale neighbourhood of $t$. Note that $X$ has the resolution property by \cite[Remark A.2]{SixAlgSt}, so \cite[Theorem 2.22]{sixopsequiv} says that, after replacing $X$ by an affine \'{e}tale neighbourhood of $t$, there is a $H^\alg$-pseudocover of $i$ consisting of a quasi-admissible map of the form $S' \to S$, where $S' \to \B G$ is affine, and where $t \times_S S'$ extends through $i \times_S S'$ by a section $S' \to X \times_S S'$ of $X \times_S S' \to S'$. Note that $X \times_S S' \to \B G$ is still affine. By \Cref{lem:gluing locality and excision on base}, it suffices to replace $S$ by $S'$, so we may assume that $t = s \circ i$, where $s : S \to X$ is a section of $X \to S$.

		By \cite[Proposition 2.25]{sixopsequiv}, in an open neighborhood of $s$, there is an \'{e}tale map $X \to V$ where $V$ is a vector bundle on $S$, and $S \to X \to V$ is the zero section. Thus, \Cref{thm:gluing is adm invar} tells us that it suffices to show that $\hat \theta_i(V, 0 \circ i)$ is $\tau_\alg$-contractible, but this is easily shown by means of an explicitly nullhomotopy (\cf{} \cite[Lemma 4.17]{sixopsequiv} and \cite[Lemma II.5.2.11]{TwAmb}).
	\end{proof}
\end{lem}

\begin{prp} \label{prp:qproj gluing}
	Suppose that for every quasi-admissible map $X \to S$, if $S \in \mathcal C^\alg$, then the map is quasi-projective, and $X \in \mathcal C^\alg$. Then all closed immersions in $\mathcal C^\alg$ are $H^\alg$-closed.
	\begin{proof}
		Note that by \Cref{lem:gluing locality on base}, it suffices to show that if $S \to \B G$ is an affine morphism, where $G$ is a linearly reductive embeddable group scheme, then any closed immersion $i : Z \to S$ is $H^\alg$-closed, \ie{} we need to show that $i_* : H^\alg(Z) \to H^\alg(S)$ is conservative, and that $H^\alg$ has gluing for $i$.

		To see that $i_* : H^\alg(Z) \to H^\alg(S)$ is conservative, it suffices to show that $H^\alg(Z)$ is generated under small colimits by the essential image of $i^*$. Since any quasi-admissible map to $Z$ admits a quasi-projective morphism to $\B G$, \Cref{lem:alg loc str} shows that $H^\alg(Z)$ is generated under small colimits by objects of the form $\cls{X}$, where $X \to Z$ is affine. Now we conclude by \cite[Corollary 2.23 and Lemma 2.3]{sixopsequiv}.

		Now we show that $H^\alg$ has gluing for $i$. First we note that since all quasi-admissible maps to $S$ and $Z$ are smooth and quasi-projective, so by \cite[Lemma A.0.8]{Fundamentals}, \cite[Theorem 4.12]{sixopsequiv} says that $i_* : H^\univ(Z) \to H^\univ(S)$ preserves $\tau_\alg$-local equivalences between reduced presheaves. As before, \Cref{lem:alg loc str} shows that $H^\alg(S)$ is generated under small colimits by objects of the form $\cls{X}$ such that $X \to S$ is affine. By \Cref{cor:gluing locality globally on base}, we reduce to showing that for any such $X \to S$, $\square_i^{H^\univ}(\cls{X})$ is sent to a coCartesian square in $H^\alg(S)$.

		Since any quasi-admissible map to $S$ is from an object of $\mathcal C^\alg$, we conclude by \Cref{prp:pseudotop gluing,lem:alg loc gluing}.
	\end{proof}
\end{prp}

\begin{proof}[Proof of \Cref{thm:alg gluing}]
	In the second case, the fact that all closed immersions are $H^\alg$-closed follows immediately from \Cref{prp:qproj gluing}. It follows from \Cref{lem:gluing for ptd} that they are also $H^\alg_\bullet$-closed. Finally, it follows from \cite[A.3.1]{SixAlgSt} and \Cref{prp:stabilize gluing} that closed immersions are $\SH^\alg$-closed.

	In the first case, we use \cite[Theorem 2.12]{SixAlgSt} and \Cref{lem:gluing locality on base} to reduce to the case of closed immersions to global quotients of the form $X/G$, where $G$ is a nice embeddable group scheme over an affine scheme, and $X$ is an affine $G$-scheme. As in \cite[A.2.2(a,b)]{SixAlgSt}, we find that in this case, $H^\alg$ and $\SH^\alg$ coincide with the versions where the quasi-admissible morphisms are the quasi-projective smooth morphisms, so we reduce to the previous case.
\end{proof}

\subsection{Holomorphic Motivic Homotopy}

In this section we will consider the motivic homotopy theory of complex analytic stacks as introduced in \cite[\S5.3]{Fundamentals}.

We write $\An_\comps$ for the category of complex spaces. This is equipped with the Grothendieck topology of open coverings, and $\Shv(\An_\comps)$ is equipped with the structure of a pullback context, as well as a quasi-admissible pseudotopology $\tau_\hol$ allowing us to define the pullback formalism $H^\hol \coloneqq H^{\tau_\hol}$ on $\Shv(\An_\comps)$, and for suitable pullback subcontexts $\HolStk^{\fin,\red} \subseteq \Shv(\An_\comps)$, we have a stabilized version $\SH^\hol$. The relevant notions are introduced in \cite[\S5.3]{Fundamentals}, and are reviewed in \Cref{S:hol gluing}.

Our goal in this section is to prove \Cref{thm:hol gluing}, which shows that embeddings between suitable objects of $\Shv(\An_\comps)$ are $H^\hol$-closed and $\SH^\hol$-closed.

In the present work we restrict ourselves to objects of $\Shv(\An_\comps)$ with finite stabilizers, although eventually we hope to generalize to the case of reductive stabilizers.

\subsubsection{Preliminaries on complex geometry}

We will need to recall some definitions following \cite[0.23]{Fischer} expressed in terms of the usual terminology for maps of locally ringed spaces.
\begin{defn} \label{defn:emb imm hol}
	A map $f : X \to Y$ in $\An_\comps$ is an \emph{embedding} if it is a closed immersion of locally ringed spaces (as in \stackscite{01HK}), and it is an \emph{immersion at $x \in X$} if it becomes an embedding after restricting to neighbourhoods of $x$ and $f(x)$.
\end{defn}

\begin{rmk} \label{rmk:invariance of sub locus}
	Let $G$ be a complex group, and let $p : X \to S$ be a $G$-equivariant holomorphic map of complex spaces with $G$-action. Then for any $n \geq 0$, the ``relative dimension $n$ submersive locus'' of $p$,
	\[
		U_n = \{x \in X \mid \text{$p$ is a submersion at $x$ of relative dimension $n$}\}
	,\]
	is a $G$-invariant open of $X$.
	\begin{proof}
		Following \cite[2.18]{Fischer}, recall that for $x \in X$, we have that $p$ is a submersion at $x$ of relative dimension $n$ if and only if there are open neighborhoods $U$ of $x$ and $V$ of $p(x)$ with $p(U) \subseteq V$, an open subset $Z \subseteq \comps^n$, and a map $U \to Z$ such that the induced map $U \to Z \times V$ is a biholomorphism. It follows immediately that $U_n$ is open, and since $p$ is $G$-equivariant, if this condition holds at $x$, then it does at $gx$ as well since we can simply take $gU, gV$ as the neighborhoods of $gx$ and $p(gx) = gp(x)$, and the map $gU \to Z$ is simply $gU \xrightarrow{g^{-1}} U \to Z$, since then $gU \to Z \times gV$ is equivalent to $U \to Z \times V$, so it is a biholomorphism.
	\end{proof}
\end{rmk}

\begin{lem} \label{lem:extn of subm is subm}
	Let
	\[
		\begin{tikzcd}
			X' \ar[d] \ar[r] & X \ar[d] \\
			S' \ar[r] & S
		\end{tikzcd}
	\]
	be a Cartesian square of complex spaces. Let $x' \in X'$ be a point, and write $s',x,s$ for its images in $S', X, S$ respectively.

	Suppose $X \to S$ is flat at $x$, and $X' \to S'$ is a submersion at $x'$. Then $X \to S$ is a submersion at $x$.
	\begin{proof}
		Since $X \to S$ is flat at $x$, \cite[3.21]{Fischer} says that it suffices to show that the fibre over the image $s$ of $x$ in $S$ is a manifold at $x$.
		Note that the fibre in $X$ over $s$ is the fibre in $X'$ over $s'$, so since $X' \to S'$ is a submersion at $x'$, this fibre is a manifold at $x$ as desired (using \cite[3.21]{Fischer} again).
	\end{proof}
\end{lem}

\begin{lem} \label{lem:imm between hol subm}
	Let $Z \to X$ be a holomorphic map over a complex space $S$. Let $z_0 \in Z$ that is sent to $x_0 \in X$. Assume $Z \to X$ is an immersion at $z_0$, and assume the maps $Z,X \to S$ are submersions at $z_0,x_0$. Then the kernel of $\mathcal O_X \to \mathcal O_Z$ is generated by a regular sequence in an open neighbourhood of $x_0$.
	\begin{proof}
		Write $s_0$ for the image of $z_0$ (or $x_0$) in $S$. By working locally near $z_0,x_0,s_0$, we may replace the map $Z \to X$ by a map $S \times \comps^n \to S \times \comps^n \times \comps^k$ over $S$, and $x_0$ corresponds to $(s_0,0,0)$, and $z_0$ corresponds to $(s_0,0)$. It follows that the map $\mathcal O_{X,x_0} \to \mathcal O_{Z,z_0}$ of local $\mathcal O_{S,s_0}$-algebras is some surjection of convergent power series rings $\mathcal O_{S,s_0}\{u_1, \dotsc, u_{n + k}\} \to \mathcal O_{S,s_0}\{t_1, \dotsc, t_n\}$. By \Cref{lem:regular local}, it suffices to show that the kernel of this map is a regular ideal.

		Indeed, by Nakayama's Lemma and \stackscite{00MG}, we may show this after quotienting out by the maximal ideal of $\mathcal O_{S,s_0}$, so we need to show that the kernel of a surjective homomorphism of local $\comps$-algebras $f : \comps\{u_1, \dotsc, u_{n + k}\} \to \comps\{t_1, \dotsc, t_n\}$ is regular.

		Indeed, since the map
		\[
			\comps^{n + k} \cong (u_1, \dotsc, u_{n + k})/(u_1, \dotsc, u_{n + k})^2 \to (t_1, \dotsc, t_n)/(t_1, \dotsc, t_n)^2 \cong \comps^n
		\]
		is surjective, and $f$ is a morphism of local $\comps$-algebras, after applying some linear change of variables to $u_1, \dotsc, u_{n + k}$, we have $p_1, \dotsc, p_{n + k} \in (t_1, \dotsc, t_n)^2$ such that for $1 \leq i \leq n$, $f(u_i) = t_i + p_i(t_1, \dotsc, t_n)$, and for $n < i \leq n + k$, we have $f(u_i) = p_i(t_1, \dotsc, t_n)$.

		For $1 \leq i \leq n$,
		\[
			f(u_i - p_i(u_1, \dotsc, u_n)) = t_i + p_i(t_1, \dotsc, t_n) - p_i(t_1 + p_1(t_1, \dotsc t_n), \dotsc, t_n + p_n(t_1, \dotsc, t_n))
		.\]
		Applying $\partial/\partial t_j$ gives
		\[
			\delta_{ij} + \frac{\partial p_i}{\partial t_j}(t_1, \dotsc, t_n) - (1 - \frac{\partial p_i}{\partial t_j}(t_1, \dotsc, t_n)) \frac{\partial p_i}{\partial t_j}(t_1 + p_1(t_1, \dotsc t_n), \dotsc, t_n + p_n(t_1, \dotsc, t_n))
		.\]
		Since $p_1, \dotsc, p_n \in (t_1, \dotsc, t_n)^2$, this evaluates to $\delta_{ij}$ at $0$. By applying the analytic Inverse Function Theorem (\eg{} \cite[Theorem 1.21]{GLS2007}) in $\comps\{t_1, \dotsc, t_n\}$, we get a change of variables such that $f(u_i) = t_i$ for $1 \leq i \leq n$, and $f(u_i)  = p_i \in (t_1, \dotsc, t_n)^2$ as before. Note that
		\[
			\comps\{u_1, \dotsc, u_n, u_{n + 1} - p_{n + 1}(u_1, \dotsc, u_n), \dotsc, u_{n + k} - p_{n + k}(u_1, \dotsc, u_n)\} = \comps\{u_1, \dotsc, u_{n + k}\}
		,\]
		so after a suitable change of variables, the kernel is the ideal $(u_{n + 1}, \dotsc, u_{n + k})$, which is clearly regular.
	\end{proof}
\end{lem}

\begin{lem} \label{lem:extn hol submersion}
	Let
	\[
		\begin{tikzcd}
			V \ar[d, "\nu"'] \ar[r, "t"] & \hat V \ar[d, "\hat \nu"] \\
			X_Z \ar[d] \ar[r] & X \ar[d] \\
			Z \ar[r, "i"'] & S
		\end{tikzcd}
	\]
	be a commutative diagram of complex spaces where the horizontal arrows are embeddings, and the bottom square is Cartesian. Write $\mathcal N, \hat{\mathcal N}$ for the conormal sheaves of $\nu, \hat \nu$ respectively.

	Let $p \in V$ be a point. By abuse of notation, we also write $p$ for the image of this point in $\hat V, X_Z, X$, and we write $z$ for the image of this point in $Z,S$. If
	\begin{itemize}

		\item $\nu$ and $\hat \nu$ are immersions at $p$,

		\item $X \to S$ is a submersion at $p$,

		\item the map $(t^* \hat{\mathcal N} \to \mathcal N)_p$ is invertible,

		\item $\hat V \to X$ is finitely presented at $p$, and

		\item $V \to Z$ is a submersion at $p$,

	\end{itemize}
	then $\hat V \to S$ is a submersion at $p$.
	\begin{proof}
		Consider the diagram
		\[
			\begin{tikzcd}
				\mathcal O_{V,p} & \ar[l] \mathcal O_{\hat V, p} \\
				\mathcal O_{X_Z, p} \ar[u] & \ar[l] \mathcal O_{X,p} \ar[u] \\
				\mathcal O_{Z,z} \ar[u] & \ar[l] \mathcal O_{S,z} \ar[u]
			\end{tikzcd}
		.\]
		Since $X \to S$ is a map of complex spaces which is a submersion at $p$, the bottom right arrow is a flat map of Noetherian local rings. 
		Furthermore, $\nu$ is an immersion at $p$, $V \to Z$ is a submersion at $p$, and $X_Z \to Z$ is a submersion at $p$ since it is a base change of $X \to S$, so \Cref{lem:imm between hol subm} says that the top left arrow is a surjection whose kernel is a quasi-regular ideal.

		By Lemma \ref{lem:check flatness of extn of qreg imm}, we have that $\mathcal O_{S,z} \to \mathcal O_{\hat V,p}$ is flat, so we conclude by Lemma \ref{lem:extn of subm is subm}.
	\end{proof}
\end{lem}

\begin{rmk}
	A holomorphic map between complex spaces is a submersion of relative dimension $0$ at a point $p$ if and only if it is biholomorphic at $p$.
\end{rmk}

\begin{rmk} \label{rmk:averaging}
	Note that when $G$ is a finite discrete group, we can use the ``averaging trick'' to get that for every complex $G$-representation $V$, we have a surjection
	\[
		\frac1{\lvert G \rvert}\sum_{g \in G} g \cdot : V \to V^G
	,\]
	which is a retract of the inclusion $V^G \to V$.

	If $X$ is a complex space on which $G$ acts, then for any $G$-equivariant coherent $\mathcal O_X$-modules $M,N \in \Coh^G(X)$, $\Coh(X)(M,N)$ naturally has the structure of a complex $G$-representation such that $\Coh^G(X)(M,N) \simeq \Coh(X)(M,N)^G$. In particular, there is a natural retraction of the inclusion $\Coh^G(X)(M,N) \to \Coh(X)(M,N)$.

	It follows that if $M$ is projective in $\Coh(X)$, then it is projective in $\Coh^G(X)$. In particular, if $X$ is a Stein space, and $M$ is finite locally free, then $M$ is projective in $\Coh^G(X)$.
\end{rmk}

\begin{prp} \label{prp:hol sec has bihol nbhd in vb}
	Let $G$ be a finite discrete group, let $p : X \to S$ be a $G$-equivariant submersion of complex $G$-spaces, and let $s : S \to X$ be a $G$-equivariant section of $p$. If $X$ is Stein, then after replacing $X$ by a $G$-invariant open neighborhood of $s$,\footnotemark{} there is a $G$-equivariant locally biholomorphic map $h : X \to V$ over $S$, where $V$ is a $G$-equivariant holomorphic vector bundle on $S$, and such that $s$ is the vanishing locus of $h$. That is, we have a Cartesian square of complex $G$-spaces over $S$ of the following form:
	\footnotetext{Using \cite{SteinNbhd}, this open neighbourhood can chosen to be Stein if $S$ is Stein.}
	\[
		\begin{tikzcd}
			S \ar[d, equals] \ar[r, "s"] & X \ar[d, "h"] \\
			S \ar[r, "0"'] & V
		\end{tikzcd}
	.\]
	\begin{proof}
		Let $\mathcal I$ be the ideal sheaf of $s$. Since $s$ is quasi-regular, we have that $s^* \mathcal I$ is locally finite free, so $p^* s^* \mathcal I$ is locally finite free. Since $G$ is finite discrete, and $X$ is Stein, we can use \Cref{rmk:averaging} to see that $p^* s^* \mathcal I$ is projective in $\Coh^G(X)$, so we have an equivariant lift $f : p^* s^* \mathcal I \to \mathcal I$ of the surjection
		\[
			p^* s^* \mathcal I \to s_* s^* p^* s^* \mathcal I = s_* s^* \mathcal I
		\]
		along the surjection $\mathcal I \to s_* s^* \mathcal I$.

		By \Cref{lem:sec has flat nbhd in conormal}, the equivariant map $S \to Z(f)$ through which $s$ factors is a clopen embedding, and since $p$ is flat,
		\[
			(\Sym s^* \mathcal I)_{p(x)} \to \mathcal O_{X,x}
		\]
		is flat for any $x$ in the image of $s$. It follows that if $V$ is the conormal bundle of $S$, then $f$ induces an equivariant map $X \to V$ over $S$ which is flat at points in the image of $s$, and such that we have a clopen section $S \to S \times_V X = Z(f)$.

		Now, since $S \times_V X \to X$ is the inclusion of a $G$-invariant closed subspace of $X$ and $S \to S \times_V X$ is the inclusion of a $G$-invariant open subspace, then after shrinking $X$ to a $G$-invariant open neighbourhood of $s$, we may assume that $S \to S \times_V X$ is an equivalence (it is clear that some such non-invariant open open neighbourhood exists, and we can make it invariant since $G$ is finite). Since $X \to V$ is flat at points of $s$, \Cref{lem:extn of subm is subm} shows that it is a submersion of relative dimension $0$ at points of $s$. Thus, by \Cref{rmk:invariance of sub locus}, we may again replace $X$ by a $G$-invariant open neighborhood of $s$ where $X \to V$ is a local biholomorphism.
	\end{proof}
\end{prp}

\begin{prp} \label{prp:lifting qsections}
	Let $G$ be a finite discrete group, let $X \to S$ be a $G$-equivariant submersion of Stein $G$-spaces, and let $i : Z \to S$ be a $G$-equivariant embedding. Let
	\[
		\nu : V \to X_Z \coloneqq X \times_S Z
	\]
	be a $G$-equivariant map that is an immersion at $p \in V$, and such that the composite $V \to Z$ is a submersion at $p$. Then there is a $G$-equivariant immersion $\hat \nu : \hat V \to X$ whose base change along $X_Z \to X$ is the restriction of $\nu$ to a $G$-invariant open neighborhood of $p$, and such that $\hat V \to S$ is a submersion.
	\begin{proof}
		Note that it suffices to replace $X$ by any $G$-invariant Stein open neighbourhood of the image of $p$, so by \Cref{lem:imm between hol subm} and \cite[Lemma B.0.2]{Fundamentals},\footnote{Note that $p$ has a neighbourhood basis in $V$ consisting of neighbourhoods of the form $(V \to X)^{-1}(U)$, where $U$ is a Stein open neighbourhood of the image of $p$ in $X$.}
		we may assume that $\nu$ is an embedding, and the kernel $\mathcal I_\nu$ of $\mathcal O_{X_Z} \to \mathcal O_V$ is generated by a regular sequence. In particular, the conormal sheaf $\nu^* \mathcal I_\nu$ is finite free, and $\mathcal I_\nu$ is of finite type.

		Since $\nu^* \mathcal I_\nu$ is finite free, it is given as $(V \to \pt)^*(\mathcal N_0)$, for $\mathcal N_0$ a $G$-equivariant vector bundle over the point, that is, finite dimensional complex representation of $G$. In particular, there is a $G$-equivariant finite free $\mathcal O_X$-module $\mathcal N \simeq (X \to \pt)^*(\mathcal N_0)$ such that $(V \to X)^*(\mathcal N) \simeq \nu^* \mathcal I_\nu$.

		Recall that by \stackscite{08KS}, and since $\nu$ is an embedding, the functor $(X_Z \to X)_* : \Coh(X_Z) \to \Coh(X)$ is fully faithful and exact, so we can make sense of the following solid diagram in $\Coh(X)$:
		\[
			\begin{tikzcd}
				{} & & \mathcal O_X \ar[d, two heads] \\
				\mathcal N \ar[d] \ar[r, dashed] \ar[urr, dashed, "f"] & \mathcal I_\nu \ar[d, two heads] \ar[r] & \mathcal O_{X_Z} \\
				(V \to X)^* \mathcal N \ar[r, no head, "\sim"] & \nu^* \mathcal I_\nu
			\end{tikzcd}
		.\]
		Since $X$ is Stein, $\mathcal N$ is projective in $\Coh(X)$, so the dashed completions of this diagram exist, and by \Cref{rmk:averaging}, we can assume that $f$ is actually a map in $\Coh^G(X)$.

		Thus, $(X_Z \to X)^* f$ is a map $\mathcal N_Z \coloneqq (X_Z \to X)^* \mathcal N \to \mathcal O_{X_Z}$ that factors through a map $\mathcal N_Z \to \mathcal I_\nu$ in $\Coh^G(X_Z)$ such that $\mathcal N_Z \to \mathcal I_\nu \to \nu_* \nu^* \mathcal I_\nu$ is surjective. It follows from \Cref{lem:describe closed as vanishing locus in a nbhd} that
		\[
			V \to Z(f) \times_X X_Z = Z(f) \times_S Z
		\]
		is an open embedding. After shrinking $X$ again, we get that $V \to Z(f) \times_S Z$ is an equivalence, and \Cref{lem:extn hol submersion} says that $Z(f) \to S$ is a submersion at the image of $p$, so we conclude by \Cref{rmk:invariance of sub locus}.
	\end{proof}
\end{prp}

\begin{lem} \label{lem:restricting lin equiv maps on Stein}
	Let $G$ be a finite group, let $Z,S$ be complex $G$-spaces, and let $i : Z \to S$ be a $G$-equivariant embedding. If $S$ is Stein, then for any finite dimensional complex representation $V$ of $G$, the induced map
	\[
		\An_\comps(S,V) \to \An_\comps(Z,V)
	\]
	induces a surjection from the set of $G$-equivariant maps $S \to V$ to the set of $G$-equivariant maps $Z \to V$.
	\begin{proof}
		Since $V$ is a finite dimensional complex vector space, we have that
		\[
			\An_\comps(S,V) \to \An_\comps(Z,V)
		\]
		is equivalent to
		\[
			\mathcal O_S(S)^n \to \mathcal O_Z(Z)^n
		\]
		for some $n$, which is
		\[
			\Gamma(S, \mathcal O_S \to i_* \mathcal O_Z)^n
		.\]
		Since $i$ is an embedding, the map $\mathcal O_S \to i_* \mathcal O_Z$ is a surjection, so since $S$ is Stein, it is also surjective after taking global sections, which implies that the morphism of complex representations of $G$
		\[
			\An_\comps(S,V) \to \An_\comps(Z,V)
		\]
		is surjective.

		We conclude by \Cref{rmk:averaging}.

	\end{proof}
\end{lem}

\begin{lem} \label{lem:lifting submersions}
	Let $G$ be a finite group, let $i : Z \to S$ be a $G$-equivariant embedding of Stein $G$-spaces, and let $p : X \to Z$ be a $G$-equivariant submersion of Stein $G$-spaces. Then for any $x \in X$, there is a $G$-equivariant submersion $\tilde p : \tilde X \to S$ such that $Z \times_S \tilde X \to Z$ factors through $p : X \to Z$ by the inclusion of a $G$-invariant open neighborhood of $x$.
	\begin{proof}
		Since $p$ is a submersion of Stein spaces, its fibres are Stein manifolds. Since $G$ is finite, $p^{-1}(Gp(x))$ is a finite disjoint union of Stein manifolds, and we have a short exact sequence
		\[
			0 \to T_x p^{-1}(Gp(x)) \to T_x X \to T_{p(x)} Z \to 0
		.\]

		By \cite{Heinzner1988}, there is some finite dimensional complex $G$-representation $V$ and $G$-equivariant embedding $\psi : p^{-1}(Gp(x)) \to V$.

		Since $p^{-1}(Gp(x)) \to X$ is a $G$-equivariant embedding, and $X$ is Stein, \Cref{lem:restricting lin equiv maps on Stein} says there is a $G$-equivariant map $f : X \to V$ extending $\psi$. Observe that the induced map $(p,f) : X \to Z \times V$ is an immersion at $x$: indeed, by \cite[Proposition 2.4]{Fischer}, it suffices to show that the kernel of
		\[
			T_x X \to T_{p(x)} Z \times T_{f(x)} V
		\]
		is zero. Suppose $v \in T_x X$ is sent to $0$ in $T_{p(x)} Z$. We have already observed that $v$ must be in the subspace $T_x p^{-1}(Gp(x))$ of $T_x X$, but then if $v$ is also sent to $0$ in $T_{f(x)} V$, it must be in the kernel of $T_x p^{-1}(Gp(x)) \to T_{\psi(x)} V$, but this map is injective since $\psi$ is an embedding.

		Thus, \Cref{prp:lifting qsections} says that there is a $G$-equivariant embedding $\tilde X \to S \times V$ whose base change along $Z \times V \to S \times V$ is the restriction of $X \to Z \times V$ to a $G$-invariant open neighborhood of $x$, and such that $\tilde X \to S$ is a submersion. In particular, $\tilde X \times_S Z \to Z$ factors through $p$ by the inclusion of a $G$-invariant open neighborhood of $x$.
	\end{proof}
\end{lem}

\subsubsection{The gluing property for complex analytic stacks} \label{S:hol gluing}

Recall that $\An_\comps$ is equipped with the Grothendieck topology of open coverings. We recall and supplement the relevant notions from \cite[\S5.3]{Fundamentals} (see there for more details):
\begin{description}

	\item[Quasi-admissibility structure] The category $\Shv(\An_\comps)$ is given the structure of a quasi-small pullback context where the quasi-admissible maps are the representable submersions.

	\item[Pseudotopologies] We have the following small quasi-admissible pseudotopologies on $\Shv(\An_\comps)$:
		\begin{description}

			\item[Open covers] An \emph{open cover} in $\Shv(\An_\comps)$ is given by a family of open subspaces $\{U_i \to X\}_i$ that pull back to an open cover of $X'$ for any complex space $X'$ and map $X' \to X$. This defines the Grothendieck topology $\tau_{\open^\an}$ on $\Shv(\An_\comps)$.

			\item[Analytic Nisnevich covers] The \emph{analytic Nisnevich topology} $\tau_{\Nis^\an}$ on $\Shv(\An_\comps)$ is the Grothendieck topology generated by $\tau_{\open^\an}$, and families of the form $\{U,\tilde X \to X\}$, where $U \to X$ is an open subspace, and $\tilde X \to X$ is a representable local biholomorphism that is an equivalence away from $U$.

			\item[Homotopy invariance] The quasi-admissible pseudotopology $\tau_\htpy$ on $\Shv(\An_\comps)$ has acyclic pseudosieves given by vector bundle torsors.

			\item[Motivic equivalence] The quasi-admissible pseudotopology $\tau_\hol$ on $\Shv(\An_\comps)$ is the union $\tau_\htpy \cup \tau_{\Nis^\an}$.

		\end{description}
		
	\item[Unstable homotopy theory] We write $H^\hol$ to denote $H^{\tau_\hol}$, as well as its restriction to any full anodyne pullback subcontext of $\Shv(\An_\comps)$. This is a pullback formalism which sends any $S \in \Shv(\An_\comps)$ to the category of sheaves $F$ on the site of representable submersions over $S$, which are also homotopy invariant in the sense that for any vector bundle torsor $V \to X$, the map $F(X) \to F(V)$ is an equivalence.

	\item[Complex analytic stacks] Define $\HolStk^\fin$ (\resp{} $\HolStk^{\fin,\red}$) to be the full subcategory of $\Shv(\An_\comps)$ consisting of those objects that admit analytic Nisnevich covers by global quotients of the form $X/G$ for $G$ a finite discrete group acting on a (\resp{} reduced) complex space $X$. Note that if $X \to Y$ is a representable submersion, and $Y \in \HolStk^\fin$ (\resp{} $\HolStk^{\fin,\red}$), then $X \in \HolStk^\fin$ (\resp{} $\HolStk^{\fin,\red}$), so $\HolStk^\fin$ and $\HolStk^{\fin,\red}$ are full anodyne pullback subcontexts of $\Shv(\An_\comps)$.

	\item[Stable homotopy theory] \cite[Theorem 5.3.10]{Fundamentals} constructs the stabilized version $\SH^\hol$ of $H^\hol$ on $\HolStk^{\fin,\red}$ characterized by the condition that the unique map $H^\hol \to \SH^\hol$ is initial among pointed pullback formalisms $D$ under $H^\hol$ satisfying that Thom spaces of vector bundles over any $S \in \HolStk^{\fin,\red}$ are sent to $\otimes$-invertible objects in $D(S)$.

\end{description}

\begin{defn}
	Say a map $X \to Y$ in $\Shv(\An_\comps)$ is an \emph{embedding} if for any $Y' \in \An_\comps$, and map $Y' \to Y$, the map $X \times_Y Y' \to Y'$ is an embedding of complex spaces (see \Cref{defn:emb imm hol}).
\end{defn}

Our goal is to prove the following:
\begin{thm} \label{thm:hol gluing}
	Any embedding in $\HolStk^\fin$ is $H^\hol$-closed, and there is a unique extension of $\SH^\hol$ to $\HolStk^\fin$ such that every embedding is $\SH^\hol$-closed.
\end{thm}

First we need some preliminary results about embeddings.
\begin{lem} \label{lem:hol emb are incisions}
	Let $i : Z \to S$ be an embedding in $\Shv(\An_\comps)$, and suppose that $S$ is of the form $X/G$ for some finite discrete group $G$ acting on a complex space $X$. Then for any analytic Nisnevich cover of $Z$, there is an open cover of $S$ that pulls back to a refinement of it, and consists of maps of the form $(U \to X)/G$, where $U \to X$ is the inclusion of a $G$-invariant Stein open subpsace.

	\begin{proof}
		We may use \cite[Lemma 5.3.9]{Fundamentals} to reduce to the case of an open cover of $Z$. If we write $\bar Z \to X$ for the $G$-equivariant embedding corresponding to $i$, then we must show that for any $G$-invariant open cover of $\bar Z$, there is a $G$-invariant open cover of $X$ that pulls back to a refinement of it. This follows easily from \cite[Lemma B.0.2]{Fundamentals}, which shows that
		\begin{itemize}

			\item every point of $X \setminus \bar Z$ has a $G$-invariant Stein open neighbourhood that is disjoint from $\bar Z$, and

			\item every point $x$ of $\bar Z$ has a neighbourhood basis in $\bar Z$ consisting of neighbourhoods of the form $U \cap \bar Z$, where $U$ is a $G$-invariant Stein open neighbourhood of $x$ in $X$.

		\end{itemize}
	\end{proof}
\end{lem}

\begin{lem} \label{lem:hol emb}
	Let $i : Z \to S$ be an embedding in $\Shv(\An_\comps)$, where $S$ is a global quotient $X/G$ for some finite group $G$ acting on a complex space $X$.
	Then $i_* : H^\univ(Z) \to H^\univ(S)$ preserves $\tau_\hol$-local equivalences between reduced presheaves.
	\begin{proof}
		Recall that $\tau_\hol = \tau_{\Nis^\an} \cup \tau_\htpy$, where $\tau_{\Nis^\an}$ is given by open covers and \'{e}tale excision, and $\tau_\htpy$ is given by vector bundle torsors.

		By \cite[Lemma 5.3.9]{Fundamentals}, it suffices to replace $\tau_{\Nis^\an}$ by $\tau_{\open^\an}$, and $\tau_\htpy$ by the pseudotopology whose acyclic pseudosieves are maps of the form $\comps \times X \to X$.

		It is easy to show that maps of the form $i_*(\comps \times X \to X)$ are $\tau_\htpy$-local equivalences using explicit $\comps$-indexed homotopies, so we find that $i_*$ preserves $\tau_\htpy$-local equivalences.

		It remains to show that $i_*$ sends $\tau_{\Nis^\an}$-local equivalences to $\tau_\hol$-local equivalences. In fact, since embeddings are stable under base change, \Cref{lem:hol emb are incisions} says that we may conclude by \cite[Lemma 3.1.6]{locspalg}.
	\end{proof}
\end{lem}

\begin{proof}[Proof of \Cref{thm:hol gluing}]
	First we show that $H^\hol$ has gluing for embeddings to objects of $\HolStk^\fin$.

	Let $\mathcal G$ be the collection of objects of $\Shv(\An_\comps)$ of the form $X/G$, where $G$ is a finite discrete group, and $X$ is a Stein $G$-space. Note that by \cite[Lemma B.0.2 or 5.3.9(2)]{Fundamentals}, if $G$ is a finite discrete group acting on a complex space $X$, then $X/G$ has an open cover by objects of $\mathcal G$. Therefore, any object of $\HolStk^\fin$ has a $\tau_{\Nis^\an}$-cover by objects of $\mathcal G$, so by \Cref{lem:hol emb}, we may apply \Cref{thm:global pseudotop gluing} to reduce to showing that if $G$ is a finite discrete group, $Z,S,X$ are Stein spaces with $G$-action, $i : Z \to S$ is a $G$-equivariant embedding, $X \to S$ is a $G$-equivariant submersion, and $t : Z \to X$ is a $G$-equivariant holomorphic map over $S$, then $\hat \theta_{i/G}(X/G,t/G) \in \Psh(\HolStk^\fin_{/(S/G)})$ is $\tau_\hol$-locally contractible.

	If we write $\nu : Z \to X_Z \coloneqq X \times_S Z$ for the map induced by $t$, then \Cref{prp:lifting qsections} says that $Z$ has an open cover $\mathcal R$ by $G$-invariant open subspaces of $Z$ such that each $U \to Z$ in $\mathcal R$ is the base change along $X_Z \to X$ of a $G$-equivariant immersion $\hat V \to X$ such that $\hat V \to X \to S$ is a submersion.

	By \Cref{lem:hol emb are incisions}, there is an open cover $\mathcal R'$ of $S$ consisting of $G$-invariant Stein open subsets, and such that $i^* \mathcal R'$ refines $\mathcal R$. Since open covering sieves are $\tau_\hol$-acyclic, it follows that $\mathcal R'$ is an $H^\hol$-pseudocover of $S$, so by \Cref{lem:gluing locality and excision on base}, it suffices to work after base change from $S$ to the members of this cover. Thus, we may assume that the map $Z \to X_Z$ is the base change along $X_Z \to X$ of some ($G$-equivariant) map $S' \to X$ such that $S' \to X \to S$ is a submersion.

	In particular, $\{S'/G, (S \setminus Z)/G \to S/G\}$ is an analytic Nisnevich cover of $S/G$, so $S'/G \to S/G$ is a $H^\hol$-pseudocover of $i$, and by \Cref{lem:gluing locality and excision on base}, it suffices to show $\hat \theta_{i/G}(X/G,t/G)$ is $\tau_\hol$-locally contractible after pulling back along $(S' \to S)/G$, in which case the map $t$ factors through a section $s : S \to X$ of $X \to S$. By \Cref{prp:hol sec has bihol nbhd in vb}, there is a $G$-invariant open neighborhood of $s$ in $X$ that is also a locally biholomorphic neighborhood of the zero section of a vector bundle in $S$. Thus, by \Cref{thm:gluing is adm invar}, $\hat \theta_i(X,t)$ is $\tau_{\Nis^\an}$-locally equivalent to $\hat \theta_i(V, 0 \circ i)$ where $0 : S \to V$ is the zero section of a vector bundle. This is seen to be $\tau_\hol$-locally contractible by means of an explicit nullhomotopy (\cf{} \cite[Lemma 4.17]{sixopsequiv} and \cite[Lemma II.5.2.11]{TwAmb}).

	We have shown that if $i : Z \to S$ is an embedding in $\HolStk^\fin$, then $H^\hol$ has gluing for $i$. To show that $i$ is $H^\hol$-closed, it remains to show that $i_*$ is conservative. Once again, we can use \Cref{lem:gluing locality on base} reduce to the case that $S = X/G$ for $G$ a finite group, and $X$ a Stein $G$-space, in which case we can conclude by \Cref{lem:lifting submersions}, which shows that $H^\tau(Z)$ is generated under colimits by the essential image of $i^*$.
	
	Finally, we must address the statements about $\SH^\hol$. The fully faithful inclusion $\HolStk^{\fin,\red} \to \HolStk^\fin$ admits a right adjoint $(-)_\red$, and for any $X \in \HolStk^\fin$, the counit $X_\red \to X$ is an embedding with empty complement, so any extension of $\SH^\hol$ to $\HolStk^\fin$ such that all embeddings are $\SH^\hol$-closed must send all of these maps $X_\red \to X$ to equivalences. It follows that if $f$ is a map such that $f_\red$ is an equivalence, then any such extension of $\SH^\hol$ must send $f$ to an equivalence. Since $(-)_\red$ preserves quasi-admissible maps and pullbacks, \cite[Proposition 5.2.7.12]{htt} shows that there is a unique such extension of $\SH^\hol$ to a pullback formalism on $\HolStk^\fin$, and it only remains to show any embedding $i$ is then $\SH^\hol$-closed.

	Recalling the construction of $\SH^\hol$ as the stabilization of $H^\hol$, It follows from \Cref{lem:gluing for ptd}, \Cref{prp:stabilize gluing}, and \cite[Remark 5.3.8]{Fundamentals}, that the embedding $i_\red$ is $\SH^\hol$-closed. We can then deduce that $i$ is $\SH^\hol$-closed using \cite[Proposition 4.17]{6FF}.
\end{proof}

\appendix
\section{Locally Ringed Spaces}

%

The following is a pleasing version of Nakayama's Lemma:
\begin{lem}[Geometric Nakayama] \label{lem:geom Nak}
	Let $i : Z \to X$ be an immersion of locally ringed spaces. If $M$ is an $\mathcal O_X$-module of finite type such that $i^*(M) = 0$, then there is an open neighborhood of $i(Z)$ where $M$ vanishes.

	In fact, this holds as long as $i$ is any map of ringed spaces such that for all $z \in Z$, $\mathcal O_{X,i(z)} \to \mathcal O_{Z,z}$ is a surjective map of (nonzero) local rings.
	\begin{proof}
		Note that since $M$ is of finite type, its support is closed by \stackscite{01B9}, so it suffices to show that $i(Z)$ is disjoint from the support of $M$, \ie{} $M_{i(z)} = 0$ for all $z \in Z$.

		Let $\mathcal I \coloneqq \ker(\mathcal O_X \to i_* \mathcal O_Z)$ be the ideal sheaf of $i$. Note that if $z \in Z$, then since
		\[
			\mathcal O_{Z,z} \cong (\mathcal O_X/\mathcal I)_{i(z)} \cong \mathcal O_{X,i(z)}/\mathcal I_{i(z)}
		\]
		is nonzero, it follows that $\mathcal I_{i(z)}$ is a proper ideal of the local ring $\mathcal O_{X,i(z)}$. Since $M_{i(z)}$ is finitely generated, $\mathcal O_{X, i(z)}$ is a local ring, and
		\[
			M_{i(z)}/\mathcal I_{i(z)} M_{i(Z)} = i^*(M)_z = 0
		,\]
		Nakayama's Lemma says that $M_{i(z)} = 0$, as desired.
	\end{proof}
\end{lem}

\begin{lem} \label{lem:regular local}
	Let $X$ be a ringed space, let $\mathcal F$ be a coherent $\mathcal O_X$-module, $\mathcal I \subseteq \mathcal O_X$ a finitely generated ideal sheaf, and let $x \in X$ be a point such that $\mathcal I_x$ is a $\mathcal F_x$-regular ideal.\footnotemark{} Then there is some open neighbourhood $U \subseteq X$ of $x$ such that $\mathcal I|_U$ is a $\mathcal F|_U$-regular ideal.\footnotemark[\thefootnote]{}
	\footnotetext{If $X$ is a ringed space, $\mathcal F$ is a $\mathcal O_X$-module, and $\mathcal I \subseteq \mathcal O_X$ is an ideal sheaf, we say that $\mathcal I$ is $\mathcal F$-regular if $\mathcal I$ is generated by a $\mathcal F$-regular sequence, that is, a sequence $f_1, \dotsc, f_c \in \mathcal I(X)$ such that for $1 \leq i \leq c$, the map $f_i : \mathcal F/(f_1, \dotsc, f_{i - 1}) \mathcal F \to \mathcal F/(f_1, \dotsc, f_{i - 1})\mathcal F$ is injective.}
	\begin{proof}
		See \stackscite{061L} for an argument in the case of schemes.

		After shrinking $X$ to some open neighbourhood of $x$, we may assume that there are global sections $f_1, \dotsc, f_c$ of $\mathcal I$ that are sent to a $\mathcal F_x$-regular sequence that generates $\mathcal I_x$. We note that since $\mathcal I$ is of finite type, and the stalk of $\mathcal I/(f_1, \dotsc, f_c)$ vanishes at $x$, \stackscite{01B9} shows that after shrinking $X$ again, we may assume that $\mathcal I$ is generated by $f_1, \dotsc, f_c$, so it suffices to show that there is some open neighbourhood $U$ of $x$ where $f_1|_U, \dotsc f_c|_U$ define a $\mathcal F|_U$-regular sequence.

		For $1 \leq i \leq c$, write
		\[
			K_i \coloneqq \ker(\mathcal F/(f_1, \dotsc, f_{i - 1})\mathcal F \to \mathcal F/(f_1, \dotsc, f_{i - 1})\mathcal F)
		.\]
		Since $\mathcal F$ is coherent, we have that $K_i$ is of finite type for all $i$ by \stackscite{01BY}. We have assumed that the stalk of $K_i$ at $x$ is zero for all $i$, so by \stackscite{01B9}, there is an open neighbourhood of $x$ where $K_i$ vanishes for all $i$.
	\end{proof}
\end{lem}

\begin{lem} \label{lem:check flatness and regularity at point of flat closed}
	Let $A \to B$ be a flat local homomorphism of local rings. Let $A \to A'$ be a surjective nonzero ring homomorphism, and
	\[
		f_1, \dotsc, f_c \in B' \coloneqq A' \otimes_A B
	\]
	be a quasi-regular sequence in $B'$ such that
	\[
		A \to B'/(f_1, \dotsc, f_c)
	\]
	is flat. If $A \to B$ is essentially of finite presentation, or $B$ is Noetherian, then for any $g_1, \dotsc, g_c \in B$ such that $g_i \mapsto f_i$, we have that $g_1, \dotsc, g_c$ is a regular sequence in $B$ and $B/(g_1, \dotsc, g_c)$ is flat over $A$.
	\begin{proof}
		Note that since $A \to A'$ is a surjective nonzero ring homomorphism, it factors the residue field projection $A \to \kappa$. By \stackscite{0CEP}, since $B'/(f_1, \dotsc, f_c)$ is flat over $A'$, and $f_1, \dotsc, f_c$ is quasi-regular, we have that the image of this sequence in $B' \otimes_{A'} \kappa$ is quasi-regular.

		Next we will show that $B' \otimes_{A'} \kappa$ is Noetherian:
		\begin{itemize}

			\item If $A \to B$ is essentially of finite presentation, so is
				\[
					\kappa \to B' \otimes_{A'} \kappa = B \otimes_A \kappa
				,\]
				so since $\kappa$ is a field, $B' \otimes_{A'} \kappa$ is Noetherian (\stackscite{00FN}).
		
			\item Otherwise, we have assumed $B$ is Noetherian, so $B' \otimes_{A'} \kappa$ must also be Noetherian, as it is a quotient of $B$.

		\end{itemize}
		Thus, any quasi-regular sequence in $B' \otimes_{A'} \kappa$ is regular, and in particular, the image of $f_1, \dotsc, f_c$ in $B \otimes_A \kappa$ is regular, so by
		\begin{itemize}

			\item \stackscite{0470} in the case that $A \to B$ is essentially of finite presentation,
			
			\item or \stackscite{00MG}\footnotemark{} in the case that $B$ is Noetherian,
				\footnotetext{This result also asks for $A$ to be Noetherian, but this hypothesis is not necessary, as the result is proven by induction using \stackscite{00ME}.}
				
		\end{itemize}
		if the sequence $g_1, \dotsc, g_c \in B$ maps to $f_1, \dotsc, f_c$, then it is a regular sequence, and $B/(g_1, \dotsc, g_c)$ is flat over $A$.
	\end{proof}
\end{lem}

\begin{lem} \label{lem:check flatness of extn of qreg imm}
	Let
	\[
		\begin{tikzcd}
			\mathcal O_{V,v} & \ar[l] \mathcal O_{\hat V, v} \\
			\mathcal O_{X_Z, v} \ar[u] & \ar[l] \mathcal O_{X,v} \ar[u] \\
			\mathcal O_{Z,z} \ar[u] & \ar[l] \mathcal O_{S,s} \ar[u]
		\end{tikzcd}
	\]
	be a commutative diagram of local rings where the horizontal arrows and the top two vertical arrows are surjective, and the bottom square is coCartesian. Write $I, \hat I$ for the ideals of the left and right top vertical arrows respectively.

	Assume
	\begin{itemize}

		\item $\mathcal O_{S,s} \to \mathcal O_{X,v}$ is flat and either that $\mathcal O_{X,v}$ is Noetherian, or that the map is essentially of finite presentation,

		\item the ideal $I$ is an ideal generated by a quasi-regular sequence $f_1, \dotsc, f_c \in \mathcal O_{X_Z, v}$,

		\item the map $\hat I/\hat I^2 \otimes_{\mathcal O_{\hat V, v}} \mathcal O_{V,v} \to I/I^2$ is invertible,

		\item $\hat I$ is finitely generated, and

		\item $\mathcal O_{Z,z} \to \mathcal O_{V,v}$ is flat.

	\end{itemize}
	Then $\mathcal O_{S,s} \to \mathcal O_{\hat V,v}$ is flat.
	\begin{proof}
		Since $f_1, \dotsc, f_c$ is a quasi-regular sequence, these elements form a basis of $I/I^2$ over $\mathcal O_{X_Z, v}/\mathcal I_v \cong \mathcal O_{V,v}$. 

		Since $\hat I$ is finitely generated, by Nakayama's Lemma\footnote{For example, part (6) of \stackscite{00DV}.} we have that $\hat I \to I$ is surjective, so we can pick $\hat f_1, \dotsc, \hat f_c \in \hat{\mathcal I}_v$ which are sent to $f_1, \dotsc, f_c$. The sequence
		\[
			1 \otimes \hat f_1, \dotsc, 1 \otimes \hat f_c \in \mathcal O_{V,v} \otimes_{\mathcal O_{X,v}} \hat I = \mathcal O_{V,v} \otimes_{\mathcal O_{\hat V, v}} \mathcal O_{\hat V,v} \otimes_{\mathcal O_{X,v}} \hat I = \mathcal O_{V,v} \otimes_{\mathcal O_{\hat V,v}} \hat I/\hat I^2 \xrightarrow{\sim} I/I^2
		\]
		is sent to the basis $f_1, \dotsc, f_c$ of $I/I^2$ over $\mathcal O_{V,v}$, so $1 \otimes \hat f_1, \dotsc, 1 \otimes \hat f_c$ is a basis of $\mathcal O_{V,v} \otimes_{\mathcal O_{X,v}} \hat I$. Since $\hat I$ is finitely generated, and $\mathcal O_{X,v} \to \mathcal O_{V,v}$ is a surjection of local rings, Nakayama's Lemma says that $\hat f_1, \dotsc, \hat f_c$ is a set of generators for $\hat I$.

		Finally, if $\mathcal O_{Z,z} \to \mathcal O_{V,v}$ and $\mathcal O_{S,s} \to \mathcal O_{X,v}$ are flat, and $\mathcal O_{X,v}$ is Noetherian or the map is essentially of finite presentation, we can conclude by \Cref{lem:check flatness and regularity at point of flat closed} since $\hat f_1, \dotsc, \hat f_c \in \mathcal O_{X,v}$ are sent to the quasi-regular sequence $f_1, \dotsc, f_c \in \mathcal O_{X_Z, v}$ (by construction).
	\end{proof}
\end{lem}

\begin{lem} \label{lem:describe closed as vanishing locus in a nbhd}
	Let $i : Z \to X$ be a closed immersion of locally ringed spaces. If the ideal sheaf $\mathcal I$ of $i$ is of finite type, and $f : \mathcal N \to \mathcal I$ is a map such that
	\[
		\mathcal N \to \mathcal I \to i_* i^* \mathcal I
	\]
	is surjective, then $Z \to Z(f)$ is a clopen immersion of locally ringed spaces.
	\begin{proof}
		Note that we have a commutative square of $\mathcal O_X$-algebras
		\[
			\begin{tikzcd}
				\Sym_{\mathcal O_X} \mathcal N \ar[d] \ar[r] & \mathcal O_X \ar[d] \\
				\mathcal O_X \ar[r] & i_* \mathcal O_Z
			\end{tikzcd}
		,\]
		where the bottom and right arrows are the unit of $i^* \dashv i_*$, and the top and left arrows are induced by the maps $f,0 : \mathcal N \to \mathcal O_X$.

		We need to show that this square becomes coCartesian after restricting to an open neighborhood of $i$, which amounts to showing that $f : \mathcal N \to \mathcal I$ is surjective in some open neighborhood of $i$. Since $\mathcal I$ is of finite type, Nakayama's Lemma (\Cref{lem:geom Nak}) says that it suffices to show that $i^* f$ is surjective.

		Indeed, note that $i^*(\id \to i_*i^*)$ is the identity of $i^*$, so $i^* f$ is the same as $i^*$ of the surjection $\mathcal N \to \mathcal I \to i_* i^* \mathcal I$. We conclude since $i^*$ is right exact.
	\end{proof}
\end{lem}

The following \lcnamecref{lem:sec has flat nbhd in conormal} can be seen as giving criteria for when a section $s : S \to X$ of a map $p : X \to S$ is locally given as the vanishing locus of a flat map $X \to V$ over $S$, where $V$ is a vector bundle over $S$. We do not state the result in these terms because we would like to be able to apply it in various geometric contexts that can be modeled by locally ringed spaces, in which the total space of a locally free sheaf may be given by different constructions (\eg{} the cases of schemes or complex spaces).
\begin{lem} \label{lem:sec has flat nbhd in conormal}
	Let $s : Z \to X$ be a quasi-regular immersion that is a section of a map $p : X \to Z$. Let $\mathcal I$ be the ideal sheaf of $s$, and suppose $f : p^* s^* \mathcal I \to \mathcal I$ is a map such that the composite with $\mathcal I \to s_* s^* \mathcal I$ is surjective. Then $s$ factors through $Z(f)$ by a clopen immersion, and if $p : X \to Z$ is flat at a point $x$ in the image of $s$, then the composite
	\[
		(\Sym s^* \mathcal I)_{p(x)} \to p^*(\Sym s^* \mathcal I)_x \cong (\Sym p^* s^* \mathcal I)_x \to \mathcal O_{X,x}
	\]
	is flat as long as $\mathcal O_{X,x}$ is Noetherian, or $\mathcal O_{Z,p(x)} \to \mathcal O_{X,x}$ is essentially of finite presentation.
	\begin{proof}
		Note that it suffices to work in any open neighbourhood of $s$, so we may assume that $s$ is a closed immersion. Since $s$ is quasi-regular, $\mathcal I$ is of finite type, so by \Cref{lem:describe closed as vanishing locus in a nbhd}, we have that $s$ factors through $Z(f)$ by a clopen immersion $Z \to Z(f)$.
		Thus, we have a commutative diagram of local rings:
		\[
			\begin{tikzcd}
				\mathcal O_{Z,z} & \ar[l] \mathcal O_{X, x} & \\
				\mathcal O_{X,s(z)} \ar[u] & \ar[l] (\Sym p^* s^* \mathcal I)_x \ar[u] & \ar[l] \mathcal O_{X, x} \\
				\mathcal O_{Z,p(s(z))} \ar[u] & \ar[l] (\Sym s^* \mathcal I)_z \ar[u] & \ar[l] \mathcal O_{Z, z} \ar[u]
			\end{tikzcd}
		.\]
		The bottom right square is coCartesian, and the horizontal composites on the bottom are identities, so the bottom left square is also coCartesian. Since $\mathcal O_{Z,z} \to \mathcal O_{X, x}$ is flat, so is the bottom middle vertical arrow. If $\mathcal O_{Z, z} \to \mathcal O_{X,x}$ is essentially of finite presentation, then so is the bottom middle vertical arrow. Since $s^* \mathcal I$ is finite over $\mathcal O_Z$, $p^* s^* \mathcal I$ is also finite over $\mathcal O_X$, so if $\mathcal O_{X,x}$ is Noetherian, then $(\Sym p^* s^* \mathcal I)_x$ is Noetherian. Thus, we may apply \Cref{lem:check flatness of extn of qreg imm} using that
		\begin{enumerate}

			\item We have already shown that $(\Sym s^* \mathcal I)_{z} \to (\Sym p^* s^* \mathcal I)_{x}$ is flat, and that it is either essentially of finite presentation, or the codomain is Noetherian.

			\item $s$ is quasi-regular, so $\mathcal I_{x}$ is generated by a quasi-regular sequence in $\mathcal O_{X, x}$.

			\item If $\hat I$ is the kernel of the map $(\Sym p^* s^* \mathcal I)_{x} \to \mathcal O_{X, x}$, then the map
				\[
					\lambda : \hat I/\hat I^2 \otimes_{\mathcal O_{X, x}} \mathcal O_{Z,z} \to (\mathcal I/\mathcal I^2)_{z}
				\]
				is the map
				\[
					(p^* s^* \mathcal I \otimes_{\mathcal O_X} \mathcal O_Z \to s^* \mathcal I)_{z}
				,\]
				which is the map
				\[
					(s^* p^* s^* \mathcal I \to s^* \mathcal I)_{z}
				,\]
				where
				\[
					s^* p^* s^* \mathcal I \to s^* \mathcal I
				\]
				is the map 
				\[
					s^*(p^* s^* \mathcal I \xrightarrow{f} \mathcal I)
				,\]
				which is a surjection since we have assumed that the composite $p^* s^* \mathcal I \to \mathcal I \to s_* s^* \mathcal I$ is a surjection. (\cf{} the proof of \Cref{lem:describe closed as vanishing locus in a nbhd}.)
				Since $s$ is quasi-regular, this is a surjective map of finite locally free sheaves of the same rank, so it is invertible by \stackscite{089Q}.

				Thus, $\lambda$ is an equivalence since it is the stalk at $z$ of this equivalence.

			\item We have already seen that $p^* s^* \mathcal I$ is finite over $\mathcal O_X$, so $\hat I$, which is the kernel of the map $(\Sym p^* s^* \mathcal I)_{x} \to \mathcal O_{X, x}$, is a finitely generated ideal.

			\item The map $\mathcal O_{Z, p(s(z))} \to O_{Z,z}$ is the identity, so it is flat.

		\end{enumerate}

		Therefore, the vertical composite in the middle is flat, as desired.
	\end{proof}
\end{lem}

\section{Gluing for Topoi} \label{S:excision}

The ideas of \cite[\S 7.3.2]{htt} are relevant to this section.

We will study a version of the gluing property in the setting of topoi. First we will need some preliminary notions.

\begin{defn}
	Given a category $\mathcal C$, write $\Op(\mathcal C)$ for the full subcategory of $\mathcal C$ of $(-1)$-truncated objects in $\mathcal C$, that is, those $U \in \mathcal C$ such that for any $X \in \mathcal C$, the space $\mathcal C(X,U)$ is either empty or contractible.
\end{defn}

\begin{rmk}
	For any category $\mathcal C$, the category $\Op(\mathcal C)$ is equivalent to a partially ordered class. Indeed, the mapping space between any two objects is either empty or contractible.
\end{rmk}

\begin{rmk} \label{rmk:opens of a pres cat}
	If $\mathcal C$ is a presentable category, then $\Op(\mathcal C)$ is equivalent to a (small) complete lattice.
	\begin{proof}
		By \cite[Proposition 5.5.6.18]{htt}, the category $\Op(\mathcal C)$ is presentable, so it is closed under small limits and colimits. We have already remarked that it is equivalent to a partially ordered class, and furthermore, we have that it is generated under small joins by a small subset.
	\end{proof}
\end{rmk}

\subsection{Trivializations}

cf. \cite[Lemma 7.3.2.4]{htt}
\begin{defn} \label{defn:trivial wrt obj}
	Given a category $\mathcal C$, and $X,U \in \mathcal C$, say $X$ is \emph{$U$-trivial} if the following equivalent conditions are satisfied:
	\begin{itemize}

		\item For every $U'$ admitting a map to $U$, the space $\mathcal C(U',X)$ is contractible.

		\item There is a map $g : U \to X$ so that $\id_U, g$ exhibits $U$ as a product $U \times X$.

	\end{itemize}
	\begin{proof}[Proof that the conditions are equivalent]
		Suppose that for any $U'$ admitting a map to $U$, the space $\mathcal C(U',X)$ is contractible. In particular, the space $\mathcal C(U,X)$ is nonempty, so there is a map $g : U \to X$. Now, for any $X' \in \mathcal C$, we must show that the induced
		\[
			\mathcal C(X',U) \to \mathcal C(X',U) \times \mathcal C(X',X)
		\]
		is an equivalence.

		If $\mathcal C(X',U)$ is empty, then this is automatic, so assume there is a map $X' \to U$, and therefore that $\mathcal C(X',X)$ is contractible, so the projection
		\[
			\mathcal C(X',U) \times \mathcal C(X',X) \to \mathcal C(X',U)
		\]
		is invertible.

		Since the composite
		\[
			\mathcal C(X',U) \to \mathcal C(X',U) \times \mathcal C(X',X) \to \mathcal C(X',U)
		\]
		is the identity, and the second map is invertible, it follows that the first map is also an equivalence, as desired.

		For the converse, the projection
		\[
			\mathcal C(-,U) \times \mathcal C(-,X) \to \mathcal C(-,U)
		\]
		admits a section that is an equivalence, so for any $U' \in \mathcal C$, every fibre of
		\[
			\mathcal C(U',U) \times \mathcal C(U',X) \to \mathcal C(U',U)
		\]
		is an equivalence $\mathcal C(U',X) \to \pt$. Thus, $\mathcal C(U',X)$ is contractible if a fibre exists, that is, if $\mathcal C(U',U)$ is nonempty.
	\end{proof}
\end{defn}

\begin{defn}
	Given $U \in \mathcal C$, say a morphism $f$ in $\mathcal C$ is an \emph{equivalence away from $U$} if for any $U$-trivial $X \in \mathcal C$, $\mathcal C(f,X)$ is an equivalence.

	For any $U,X \in \mathcal C$, say $X \to X \oslash U$ is a \emph{$U$-trivialization of $X$} if $X \oslash U$ is $U$-trivial and the map is an equivalence away from $U$.
\end{defn}

\begin{rmk} \label{rmk:triviality and squares}
	Suppose $Y \in \mathcal C$ is trivial on $U \in \Op(\mathcal C)$. If $X \in \mathcal C$ is such that $X \times U$ exists, then for any map $X \to Y$, there is a unique commutative square
	\[
		\begin{tikzcd}
			X \times U \ar[d] \ar[r] & X \ar[d] \\
			U \ar[r] & Y
		\end{tikzcd}
	\]
	where $X \to Y$ is the prescribed map and the top and left arrows are the projections. Indeed, since $Y$ is $U$-trivial, the space of maps $U \to Y$ is contractible, so the space of such (not necessarily commutative squares) is contractible, and since the space of maps $X \times U \to Y$ is also contractible, every such square is commutative in a unique way.
\end{rmk}

\begin{defn} \label{defn:excision square}
	Given a category $\mathcal C$, a commutative square
	\[
		\begin{tikzcd}
			X_U \ar[d] \ar[r] & X \ar[d] \\
			U \ar[r] & X'
		\end{tikzcd}
	\]
	is an \emph{excision square} if it is Cartesian and $X \to X'$ is a $U$-trivialization.
\end{defn}

\begin{lem} \label{lem:compatibility of trivializations}
	Let $F : \mathcal C \to \mathcal D$ be any functor that preserves binary products. Then for any $X,U \in \mathcal C$, if $X$ is $U$-trivial, then $F(X)$ is $F(U)$-trivial. If $F$ is also fully faithful, then the converse also holds.

	In particular, $X$ is $U$-trivial in $\mathcal C$ if and only if $\mathcal C(-,X)$ is $\mathcal C(-,U)$-trivial in $\Psh(\mathcal C)$.
	\begin{proof}
		Suppose $X \times U$ exists. Then $F(X) \times F(U) \to F(U)$ is equivalent over $F(U)$ to $F(X \times U \to U)$. Assume $X$ is $U$-trivial, so $X \times U \to U$ admits a section that is an equivalence. Then $F(X) \times F(U) \to F(U)$ admits a section $F(U) \to F(X \times U) \simeq F(X) \times F(U)$ that is an equivalence since it is $F$ of the section $U \to X \times U$, which is an equivalence.

		Now assume $F$ is fully faithful. If $F(X)$ is $F(U)$-trivial, then since we have a map $F(U) \to F(X)$ making $F(U) \gets F(U) \to F(X)$ a product, since $F$ is fully faithful, we have a map $U \to X$ that is sent to this $F(U) \to F(X)$, so since $F(U \gets U \to X)$ is sent to the product $F(U) \gets F(U) \to F(X)$, and $F$ is fully faithful, it follows that $U \gets U \to X$ is a product.
	\end{proof}
\end{lem}

\begin{prp} \label{prp:topos trivialization and excision}
	Let $\mathcal E$ be a topos. For any $U \in \Op(\mathcal E)$, we have the following:
	\begin{enumerate}

		\item If $L_U$ is the localization $\mathcal E \to \mathcal E/U$, then for any $X \in \mathcal E$, the map $X \to L_U(X)$ is a $U$-trivialization.

		\item For any $X \in \mathcal E$, there is an excision square of the form
			\[
				\begin{tikzcd}
					X \times U \ar[d] \ar[r] & X \ar[d] \\
					U \ar[r] & L_U(X)
				\end{tikzcd}
			,\]
			where the top and left arrows are product projections.

	\end{enumerate}
	In particular, every excision square in $\mathcal E$ where the bottom map is a monomorphism from a $(-1)$-truncated object is an excision square in the sense of \cite[2.5.2.2]{SAG}, so it is also biCartesian, and for $U \in \Op(\mathcal E)$,
	\[
		L_U \simeq - \coprod_{- \times U} U
	.\]
	\begin{proof}
		Note that by \cite[Lemma 7.3.2.4]{htt}, the $U$-trivial objects of $\mathcal E$ are the objects in $\mathcal E/U$, so this guarantees that for any $X \in \mathcal E$, the map $X \to L_U(X)$ is a $U$-trivialization.

		Next, for any $X \in \mathcal E$, since $L_U(X)$ is $U$-trivial, and $X \times U$ admits a map to $U$, the space of maps $X \times U \to L_U(X)$ is contractible, so every square
		\[
			\begin{tikzcd}
				X \times U \ar[d] \ar[r] & X \ar[d] \\
				U \ar[r] & L_U(X)
			\end{tikzcd}
		\]
		commutes. Consider such a square where all the maps are given by the canonical choices. We will show that this square is biCartesian.

		Recall \cite[Lemma A.5.11]{ha} says that any map that is an equivalence away from $U$ (sent to an equivalence in $\mathcal E/U$), and an equivalence over $U$ (the base change along $U \to \pt$ is an equivalence) is an equivalence. Thus, the map
		\[
			\mathcal E \to \mathcal E_{/U} \times \mathcal E/U
		\]
		is conservative, and preserves biCartesian squares, so it also reflects them. Thus, it suffices to show that the images of our square in $\mathcal E_{/U}$ and $\mathcal E/U$ are biCartesian.

		Indeed, the images are
		\[
			\begin{tikzcd}
				X \times U \times U \ar[d] \ar[r] & X \times U \ar[d] \\
				U \ar[r] & L_U(X) \times U
			\end{tikzcd}
			\text{ and }
			\begin{tikzcd}
				L_U(X \times U) \ar[d] \ar[r] & L_U(X) \ar[d] \\
				L_U(U) \ar[r] & L_U L_U(X)
			\end{tikzcd}
		,\]
		which are
		\[
			\begin{tikzcd}
				X \times U \ar[d] \ar[r, equals] & X \times U \ar[d] \\
				U \ar[r, equals] & U
			\end{tikzcd}
			\text{ and }
			\begin{tikzcd}
				\emptyset \ar[d, equals] \ar[r] & L_U(X) \ar[d, equals] \\
				\emptyset \ar[r] & L_U(X)
			\end{tikzcd}
		,\]
		using that $U$ is $(-1)$-truncated, $L_U$ is a localization,
		and $L_U(U) = \emptyset$ ($U$ is the initial object of $\mathcal E/U$).

		Both of these are biCartesian.

		To see that this is an excision square in the sense of \cite[2.5.2.2]{SAG}, it suffices to note that $X \times U \to X$ is a monomorphism.
	\end{proof}
\end{prp}

\begin{cor} \label{cor:closed complement preserves finite limits}
	If $\mathcal E$ is a topos, and $U \in \Op(\mathcal E)$, then $- \oslash U$ defines a functor that preserves finite limits.
\end{cor}

\subsection{Gluing}

Fix a geometric morphism $i : \mathcal E' \to \mathcal E$. Note that many of the results and definitions actually also work for LC adjunctions between presentable categories with universal colimits.

\begin{defn} \label{defn:open complement}
	Define the \emph{open complement} of $i$ to be the object $i^\complement \in \Op(\mathcal E)$ characterized by the following property: the essential image of the fully faithful functor $\mathcal E_{/i^\complement} \to \mathcal E$ is given by all objects $X \in \mathcal E$ such that $i^*(X)$ is initial. 
	\begin{proof}[Proof of existence]
		%
		Throughout the argument, we will make use of the fact, given by \cite[Lemma 6.1.3.6]{htt}, that any map in $\mathcal E'$ to an initial object is invertible.

		We can define a $(-1)$-truncated presheaf $i^\complement$ on $\mathcal E$ given by
		\[
			i^\complement(X) = \begin{cases}
				\pt & i^*(X) \text{ is initial} \\
				\emptyset & \text{otherwise}
			\end{cases}
		.\]
		It suffices to show that $i^\complement$ is representable, and by \cite[Proposition 5.5.2.2]{htt} this is equivalent to showing that it is limit-preserving.

		Suppose we have some small diagram $\{X_a\}_a$ in $\mathcal C$. If $i^*(X_a)$ is initial for all $a$, then so is $i^*(\varinjlim_a X_a) \simeq \varinjlim_a i^*(X_a)$. The converse also holds, since $i^*(X_a)$ admits a map to $i^*(\varinjlim_a X_a)$. Thus, we find that
		\[
			i^\complement(\varinjlim_a X_a) = \pt \iff i^\complement(X_a) = \pt \text{ for all $a$}
		.\]
		Therefore, it suffices to show that if $\varprojlim_a i^\complement(X_a) = \pt$, then $i^\complement(X_a) = \pt$ for all $a$, but since the limit admits a map to $i^\complement(X_a)$ for all $a$, if the limit is nonempty, none of the $i^\complement(X_a)$ can be empty.
	\end{proof}
\end{defn}

\begin{exa} \label{exa:compl of map}
	Suppose there is some map $Z \to X$ in a topos $\bar{\mathcal E}$ such that $i$ is the induced morphism $\bar{\mathcal E}_{/Z} \to \bar{\mathcal E}_{/X}$. Then $i^\complement \in \bar{\mathcal E}_{/X}$ is the usual complement of $Z \to X$: it is given by the monomorphism $X \setminus Z \to X$ such that a map $X' \to X$ factors through $X \setminus Z$ if and only if $X' \times_X Z$ is initial.
\end{exa}

\begin{lem} \label{lem:open compl in site}
	Let $\mathcal C$ be a site, suppose that $\mathcal E = \Shv(\mathcal C)$, and write $\jmath : \mathcal C \to \mathcal E$ for the sheafified Yoneda embedding. Then $i^\complement \in \Shv(\mathcal C) \subseteq \Psh(\mathcal C)$ is given as follows
	\[
		i^\complement(X) \simeq \begin{cases}
			\pt & i^* \jmath(X) \text{ is initial} \\
			\emptyset & \text{otherwise}
		\end{cases}
	.\]
	\begin{proof}
		Write $\mathcal U$ for the presheaf of $(-1)$-truncated spaces given above (note that this is a presheaf by \cite[Lemma 6.1.3.6]{htt}). We will show that $\mathcal U \in \Shv(\mathcal C)$, and that $\mathcal U$ defines an open complement of $i$.

		Let $\mathcal R$ be a sieve on an object $X \in \mathcal C$. Then
		\[
			\mathcal U(X) \to \varprojlim_{\substack{X' \to X \\ \text{in $\mathcal R$}}} \mathcal U(X')
		\]
		is a map between $(-1)$-truncated spaces, so the only way it could not be invertible is if it is of the form $\emptyset \to \pt$. Thus, it suffices to show that if $\mathcal R$ is a covering sieve, and the right-hand side is $\pt$, then $\mathcal U(X) \simeq \pt$.

		In this case, since the right-hand side admits maps to $\mathcal U(X')$ for every $X' \to X$ in $\mathcal R$, we cannot have $\mathcal U(X') = \emptyset$ for any such $X'$, and therefore $i^* \jmath(X')$ is initial for all $X' \to X$ in $\mathcal R$.

		Since $\mathcal R$ is a covering sieve, and $i^*$ preserves colimits we have that $i^* \jmath(X)$ is a colimit of the $i^* \jmath(X')$, all of which are initial, so $i^* \jmath(X)$ is initial, and $\mathcal U(X) = \pt$, as desired.

		Since $\mathcal U$ is $(-1)$-truncated, to see that it is the open complement of $i$, it only remains to show that if $F \in \Shv(\mathcal C)$ is such that $i^* F$ is initial, then there is a map $F \to \mathcal U$. Indeed, $F$ admits a map to $\mathcal U$ if and only if $F(X) = \emptyset$ for every $X \in \mathcal C$ such that $\mathcal U(X) = \emptyset$, so we must show that if $F(X)$ is nonempty, then $i^* \jmath(X)$ is initial.

		If $F(X)$ is nonempty, then there is a map $\jmath(X) \to F$, and applying $i^*$ to this map gives $i^* \jmath(X) \to i^* F$, so $i^* \jmath(X)$ admits a map to an initial object, and is therefore initial (using \cite[Lemma 6.1.3.6]{htt} again).
	\end{proof}
\end{lem}

\begin{defn} \label{defn:topos gluing}
	There is a unique commutative \emph{gluing square} $\square_i$
	\[
		\begin{tikzcd}
			i^\complement \times - \ar[d] \ar[r] & \id \ar[d] \\
			i^\complement \ar[r] & i_* i^*
		\end{tikzcd}
	\]
	where the top arrow and left arrows are the projections, the right arrow is the unit, and the bottom arrow is adjoint to $i^*(i^\complement) = \emptyset \to i^*$.

	Thus, we obtain a ``gluing morphism''
	\[
		\Theta_i : - \oslash i^\complement \to i_* i^*
	\]
	given by the coCartesian gap.

	Given a map $t : A \to i_* i^* X$, write
	\[
		\Theta_i(X,t) : \theta_i(X,t) \to A
	\]
	for the base change of $\Theta_i(X)$ along $t$.
\end{defn}

\begin{rmk} \label{rmk:describe relative gluing in topos}
	In \Cref{defn:topos gluing}, we have that $\Theta_i(X,t)$ is the unique map
	\[
		(X \times_{i_* i^* X} A) \oslash (i^\complement \times A) \to A
	\]
	extending $X \times_{i_* i^* X} A \to A$.
\end{rmk}

\begin{lem} \label{lem:truncated gluing}
	In the setting of \Cref{defn:topos gluing}, if $A,X \in \mathcal E$ are $n$-truncated, then $\theta_i(X,t)$ is $n$-truncated in $\mathcal E$.
	\begin{proof}
		Note that
		\[
			\theta_i(X, t) = (X \oslash i^\complement) \times_{i_* i^* X} A
		.\]
		By \Cref{cor:closed complement preserves finite limits}, We know that $- \oslash i^\complement$ is a left exact functor, and so is $i_* i^*$, so they preserve $n$-truncated objects by \cite[Proposition 5.5.6.16]{htt}. Furthermore, $A$ is $n$-truncated, so the fibred product is $n$-truncated by \cite[Proposition 5.5.6.5]{htt}.
	\end{proof}
\end{lem}

\begin{lem} \label{lem:gluing lex}
	Given finite diagrams $A,X : K \to \mathcal E$, and a morphism $t : A \to i_*i^* X$, we have that
	\[
		\Theta_i(\varprojlim X, \varprojlim t) \to \varprojlim_{p \in K} \Theta_i(X(p), t(p))
	\]
	is invertible.
	\begin{proof}
		First note that
		\[
			\Theta_i(\varprojlim X) \to \varprojlim_{p \in K} \Theta_i(X(p))
		\]
		is
		\[
			\begin{tikzcd}[column sep=large]
				(\varprojlim X) \oslash i^\complement \ar[d] \ar[r, "\Theta_i(\varprojlim X)"] & i_* i^* \varprojlim X \ar[d] \\
				\varprojlim (X \oslash i^\complement) \ar[r, "\varprojlim \Theta_i X"'] & \varprojlim i_* i^* X
			\end{tikzcd}
		,\]
		but both $- \oslash i^\complement$ and $i_* i^*$ preserve finite limits (recall \Cref{cor:closed complement preserves finite limits}), so the vertical arrows are invertible, whence the morphism of arrows
		\[
			\Theta_i(\varprojlim X) \to \varprojlim_{p \in K} \Theta_i(X(p))
		\]
		is invertible.

		Now we can conclude using the fact that limits commute with limits.
	\end{proof}
\end{lem}

\begin{lem} \label{lem:map on topos theta is eff epi}
	Let $f : X \to Y$ be a map $\mathcal E$, and let
	\[
		\begin{tikzcd}
			A \ar[d, "t"'] \ar[r, "\nu"] & B \ar[d, "u"] \\
			i_*i^* X \ar[r, "i_*i^* f"'] & i_* i^* Y
		\end{tikzcd}
	\]
	be a Cartesian square in $\mathcal E$ such that $i^\complement \times \nu$ is invertible and
	\[
		i^\complement \times Y \times_{i_* i^* Y} B \coprod X \times_{i_* i^* Y} B \to Y \times_{i_* i^* Y} B
	\]
	is an effective epimorphism. Then
	\[
		\theta_i(X,t) \to \theta_i(Y, u)
	\]
	is an effective epimorphism.
	\begin{proof}
		By Remark \ref{rmk:describe relative gluing in topos}, we have a map of coCartesian squares
		\[ \tag{$\square$} \label{eqn:map of squares}
			\begin{tikzcd}
				i^\complement \times X \times_{i_* i^* X} A \ar[dd] \ar[dr] \ar[rr] & & X \times_{i_* i^* X} A \ar[dd] \ar[dr] \\
																			 & i^\complement \times Y \times_{i_* i^* Y} B \ar[rr, crossing over] & & Y \times_{i_* i^* Y} B \ar[dd] \\
				i^\complement \times A \ar[dr, "\sim"] \ar[rr] & & \theta_i(A,t) \ar[dr] \\
														 & i^\complement \times B \ar[from=uu, crossing over] \ar[rr] & & \theta_i(B, u)
			\end{tikzcd}
		.\]

		By Lemma \ref{lem:eff epi from colim vs eff epi from diagram}, we have that
		\[
			i^\complement \times B \coprod Y \times_{i_* i^* Y} B \to \theta_i(B, u)
		\]
		is an effective epimorphism, so by hypothesis (and \cite[Corollary 6.2.3.12 (1)]{htt}), we have that
		\[
			i^\complement \times B \coprod i^\complement \times Y \times_{i_* i^* Y} B \coprod X \times_{i_* i^* Y} B \to \theta_i(B, u)
		\]
		is an effective epimorphism. Thus, it suffices to show that the map from each of these disjoint summands to $\theta_i(B, i_*i^* f \circ t)$ factors through a member of the back square in \eqref{eqn:map of squares}. This will give us that the whole coproduct factors through the map from $\theta_i(A,t)$, whence we can conclude by \cite[Corollary 6.2.3.12 (2)]{htt}.

		This is easy for $i^\complement \times B$, since the map on the bottom left corners is invertible.

		For $i^\complement \times Y \times_{i_* i^* Y} B$, by the commutativity of the front square in \eqref{eqn:map of squares}, the map actually factors through $i^\complement \times B$, so we conclude in the same way.

		Finally, for $X \times_{i_* i^* Y} B$,
		\[
			X \times_{i_* i^* X} A \to X \times_{i_* i^* Y} B
		\]
		is invertible since $i_* i^* X \times_{i_* i^* Y} B = A$, so the map $X \times_{i_* i^* Y} B \to Y \times_{i_* i^* Y} B$ factors through the map between top right corners in \eqref{eqn:map of squares}, as desired.
	\end{proof}
\end{lem}

\begin{lem} \label{lem:map on geom theta is eff epi}
	Let $f : X \to Y$ be a morphism in $\mathcal E$, and let $u : B \to i^*Y$ be a map in $\mathcal E'$.

	Suppose that there is some $V \to Y$ such that $i^*(V \to Y)$ is disjoint from $u$ (\ie{} $i^* V \times_{i^* Y} B$ is initial), and $X \coprod V \to Y$ is an effective epimorphism.
	Then
	\[
		\theta_i(X, i_* t) \to \theta_i(Y, i_* u)
	\]
	is an effective epimorphism, where $t : A \to i^* X$ is the base change of $u$ along $i^* f$.
	\begin{proof}
		Since $i_*$ preserves Cartesian squares,
		\[
			\begin{tikzcd}
				i_* A \ar[d] \ar[r] & i_* B \ar[d] \\
				i_* i^* X \ar[r, "i_* i^* f"'] & i_* i^* Y
			\end{tikzcd}
		\]
		is Cartesian, so by \Cref{lem:map on topos theta is eff epi}, it only remains to show that
		\[
			i^\complement \times Y \times_{i_* i^* Y} i_* B \coprod X \times_{i_* i^* Y} i_* B \to Y \times_{i_* i^* Y} i_* B
		\]
		is an effective epimorphism.

		By our choice of $V \to Y$,
		\[
			X \times_{i_* i^* Y} i_* B \coprod V \times_{i_* i^* Y} i_* B \to Y \times_{i_* i^* Y} i_* B
		\]
		is an effective epimorphism, so it suffices to show that we have a factorization
		\[
			V \times_{i_* i^* Y} i_* B \to i^\complement \times Y \times_{i_* i^* Y} i_* B \to Y \times_{i_* i^* Y} i_* B
		.\]
		Since $i^\complement$ is $(-1)$-truncated, such a factorization exists if and only if $V \times_{i_* i^* Y} i_* B$ admits a map to $i^\complement$, which is equivalent to the condition that $i^*(V \times_{i_* i^* Y} i_* B)$ is initial.

		We have a commutative diagram
		\[
			\begin{tikzcd}
				i^*(V \times_{i_* i^* Y} i_* B) \ar[d] \ar[rr] & & i^* i_* B \ar[d] \ar[r] & B \ar[d, "u"] \\
				i^* V \ar[r] & i^* Y \ar[r] & i^* i_* i^* Y \ar[r] & i^* Y
			\end{tikzcd}
		.\]

		Note that the composite $i^* Y \to i^* Y$ is the identity, so the map $i^* V \to i^* Y$ is $i^*$ of the map $V \to Y$, which is disjoint from $u$. Since the outermost rectangle commutes, we have that the top left corner must be initial, as desired.
	\end{proof}
\end{lem}

\begin{cor} \label{cor:map on geom theta is eff epi}
	Suppose there is some object $Z \in \mathcal E$ such that $i$ is the induced \'{e}tale geometric morphism $\mathcal E_{/Z} \to \mathcal E$. Let
	\[
		\begin{tikzcd}
			X' \ar[d] \ar[r] & X \ar[d] \\
			Y' \ar[r] & Y
		\end{tikzcd}
	\]
	be a Cartesian square in $\mathcal E$. Any map $Y' \to Z$ induces maps $u : Y' \to Z \times Y = i^* Y$ and $t : X' \to Z \times X = i^* X$ in $\mathcal E_{/Z}$.

	Recall the notation $Y \setminus Y'$ from \Cref{exa:compl of map}, and suppose that $X \coprod (Y \setminus Y') \to Y$ is an effective epimorphism. Then
	\[
		\theta_i(X, i_* t) \to \theta_i(Y, i_* u)
	\]
	is an effective epimorphism.
	\begin{proof}
		Note that we have Cartesian squares
		\[
			\begin{tikzcd}
				X' \ar[d] \ar[r] & Z \times X \ar[d] \ar[r] & X \ar[d] \\
				Y' \ar[r] & Z \times Y \ar[r] & Y
			\end{tikzcd}
		\]
		in $\mathcal E$, where the top square is the image under the slice projection $\mathcal E_{/Z} \to \mathcal E$ of a square 
		\[
			\begin{tikzcd}
				A \ar[d] \ar[r, "t"] & i^* X \ar[d] \\
				B \ar[r, "u"'] & i^* Y
			\end{tikzcd}
		.\]
		Since the slice projection creates pullbacks, this is a Cartesian square in $\mathcal E_{/Z}$.

		It follows that $i^*(Y \setminus Y')$ is disjoint from $u$, so we conclude by \Cref{lem:map on geom theta is eff epi}.
	\end{proof}
\end{cor}

Now let us fix an object $Z \in \mathcal E$ such that $i$ is the induced \'{e}tale geometric morphism $\mathcal E_{/Z} \to \mathcal E$, and we fix collections of maps $\mathcal K$ and $\Sigma$ in $\mathcal E$ such that if
\[
	\begin{tikzcd}
		Z \ar[d, equals] \ar[r, "t'"] & X' \ar[d] \\
		Z \ar[r, "t"'] & X
	\end{tikzcd}
\]
is a Cartesian square where $t \in \mathcal K$ and $X' \to X$ is in $\Sigma$, then $t' \in \mathcal K$, and the map
\[
	X' \coprod (X \setminus Z) \to X
\]
is an effective epimorphism.

For each $n \geq -1$, inductively define $\Sigma_n$ as follows: $\Sigma_{-1}$ is the collection of all maps, and $\Sigma_{n + 1}$ is the collection of maps in $\Sigma$ whose diagonal is in $\Sigma_n$. Define
\[
	\Sigma_\infty = \bigcap_n \Sigma_n
.\]

\begin{exa}
	Suppose that $\mathcal E = \Shv^\Nis(\Sch)$ is the topos of sheaves on the Nisnevich site, $Z$ is representable by a scheme, and let $\Sigma$ be the collection of all \'{e}tale morphisms between schemes. Then $\Sigma_\infty = \Sigma$, and $\mathcal K$ can be the collection of all closed immersions from $Z$.
\end{exa}

\begin{lem} \label{lem:geom gluing connective}
	Let
	\[
		\begin{tikzcd}
			Z \ar[d, equals] \ar[r] & X' \ar[d] \\
			Z \ar[r] & X
		\end{tikzcd}
	\]
	be a Cartesian square, where $Z \to X$ is in $\mathcal K$, and $X' \to X$ is in $\Sigma_n$. Write $t : \pt \to i^* X$ and $t' : \pt \to i^* X'$ for the maps in $\mathcal E_{/Z}$ corresponding to $Z \to X$ and $Z \to X'$ above.

	Then the induced map
	\[
		Q : \theta_i(X', i_* t') \to \theta_i(X, i_* t)
	\]
	is $n$-connective.

	In particular, if
	\begin{itemize}

		\item $\mathcal E$ is hypercomplete and $X' \to X$ is in $\Sigma_\infty$, or

		\item $X$ and $X'$ are $(n - 1)$-truncated,

	\end{itemize}
	then $Q$ is invertible.
	\begin{proof}
		We proceed by induction on $n$. The case $n = -1$ is vacuous, so assume $n \geq 0$, and that the result holds for $n - 1$.

		Since $n \geq 0$, we have that $X' \to X$ is in $\Sigma$, so $X' \coprod (X \setminus Z) \to X$ is an effective epimorphism. \Cref{cor:map on geom theta is eff epi} then shows that $Q$ is an effective epimorphism, so by \cite[Proposition 6.5.1.18]{htt}, it suffices to show that the diagonal of $Q$ is $(n - 1)$-connective.

		Consider the following commutative diagram in $\mathcal E$ where all squares are Cartesian:
		\[
			\begin{tikzcd}
				Z \ar[d, equals] \ar[r] & X' \ar[d] & \\
				Z \ar[d, equals] \ar[r] & X' \times_X X' \ar[d] \ar[r] & X' \ar[d] \\
				Z \ar[r] & X' \ar[r] & X
			\end{tikzcd}
		.\]
		By our assumptions on $\mathcal K$, all of the horizontal arrows from $Z$ are in $\mathcal K$.

		\Cref{lem:gluing lex} shows that the diagonal of $Q$ is the map
		\[
			 \theta_i(X', i_* t') \to \theta_i(X' \times_X X'), i_* (t',t'))
		\]
		induced by the top Cartesian square of the above diagram. Since $X' \to X' \times_X X'$ is in $\Sigma_{n - 1}$, we can apply the inductive hypothesis to the top square to show that the diagonal of $Q$ is $(n - 1)$-connective.

		If $X' \to X$ is in $\Sigma_\infty$-admissible, then $Q$ is $\infty$-connective, so if $\mathcal E$ is hypercomplete, then $Q$ is invertible. On the other hand, if $X,X'$ are $(n - 1)$-truncated, then by \Cref{lem:truncated gluing}, we have that $Q$ is an $n$-connective map between $(n - 1)$-truncated objects, so it must be invertible.
	\end{proof}
\end{lem}

\section{Effective Epimorphisms}

\begin{lem} \label{lem:eff epi from colim vs eff epi from diagram}
	Let $\mathcal E$ be a semitopos, and let $\alpha : K \to \mathcal E$ be a small diagram. Suppose $S$ is a set of zero-simplices in the simplicial set $K$ such that for any $p \in K$, there is a $q \in S$ and a map $p \to q$ in $K$. Then for any
	\[
		f : \varinjlim \alpha \to X
	,\]
	we have that $f$ is an effective epimorphism if and only if
	\[
		\coprod_{q \in S} \alpha(q) \to X
	\]
	is an effective epimorphism.
	\begin{proof}
		We will show that
		\[
			\coprod_{q \in S} \alpha(q) \to \varinjlim \alpha
		\]
		is an effective epimorphism, whence the result follows from \cite[Corollary 6.2.3.12]{htt}.

		Consider a colimiting extension $\bar \alpha : K^\triangleright \to \mathcal E$ of $\alpha$, so $\varinjlim \alpha = \bar \alpha(\infty)$, where $\infty \in K^\triangleright$ is the cone point.

		By \cite[Lemma 6.2.3.13]{htt}, the map
		\[
			\coprod_{p \in K_0} \alpha(p) \to \alpha(\infty)
		\]
		is an effective epimorphism, so by \cite[Corollary 6.2.3.12 (2)]{htt}, it suffices to show that it lifts through the map
		\[
			\coprod_{q \in S} \alpha(q) \to \alpha(\infty)
		.\]

		Indeed, define
		\[
			\coprod_{p \in K_0} \alpha(p) \to \coprod_{q \in S} \alpha(q)
		\]
		by picking for each $p \in K_0$ a $q \in S$ and a map $p \to q$, using the map
		\[
			\alpha(p) \xrightarrow{\alpha(p \to q)} \alpha(q) \to \coprod_{q \in S} \alpha(q)
		.\]

		Since the map $p \to \infty$ lifts through $q \to \infty$ by any map $p \to q$, this gives the desired lift. 
	\end{proof}
\end{lem}

\begin{lem} \label{lem:characterize covering families by eff epi}
	Let $\mathcal C$ be a site, and let $\{X_i\}_i$ be a diagram in $\mathcal C_{/X}$ for some $X \in \mathcal C$. Then $\{X_i \to X\}_i$ is a covering family if and only if it induces an effective epimorphism
	\[
		\varinjlim_i j(X_i) \to j(X)
	\]
	where $j$ is the sheafified Yoneda embedding.
	\begin{proof}
		By Lemma \ref{lem:eff epi from colim vs eff epi from diagram} it suffices to show that the family is covering if and only if it induces an effective epimorphism
		\[
			\coprod_i j(X_i) \to j(X)
		.\]

		Recall that the family is covering if and only if it generates a covering sieve. By \cite[Lemma 6.2.3.18]{htt}, the sieve generated by the family can be identified with the $(-1)$-truncation of the map
		\[
			\coprod_i \yo(X_i) \to \yo(X)
		.\]
		By \cite[Lemma 6.2.2.16]{htt}, this truncation is a covering sieve if and only if its sheafification is an equivalence.

		Since sheafification is a left exact presentable functor which sends $\yo(X)$ to $j(X)$, the functor $\Psh(\mathcal C)_{/\yo(X)} \to \Shv(\mathcal C)_{/j(X)}$ is also a left exact presentable functor, so by \cite[Proposition 5.5.6.28]{htt}, the sheafification of the $(-1)$-truncation of above map is the $(-1)$-truncation of the map
		\[
			\coprod_i j(X_i) \to j(X)
		.\]
		Thus, we have reduced to showing that the $(-1)$-truncation of this map is an equivalence if and only if it is an effective epimorphism, but this is by definition (see \cite[Corollary 6.2.3.5]{htt}).
	\end{proof}
\end{lem}

\begin{lem} \label{lem:crit for pt to be colim of family}
	Let $\mathcal E$ be a semitopos, and let $\mathcal O$ be a small collection of objects in $\mathcal E$. If the terminal object of $\mathcal E$ is a small colimit of objects in $\mathcal O$, then
	\[
		\coprod_{x \in \mathcal O} x \to \pt
	\]
	is an effective epimorphism, and the converse holds if $\mathcal O$ is closed under binary products.
	\begin{proof}
		Note that for any diagram $F : K \to \mathcal E$ such that for each $p \in K$, $F(p)$ is equivalent to an object in $\mathcal O$, we have a map
		\[
			\coprod_{p \in K} F(p) \to \coprod_{x \in \mathcal O} x
		.\]

		By \cite[Lemma 6.2.3.13]{htt}, there is an effective epimorphism from the domain to the terminal object, so by \cite[Corollary 6.2.3.13 (2)]{htt}, the map from the codomain to the terminal object is also an effective epimorphism.

		For the converse, we assume that the $(-1)$-truncation of $\coprod_{x \in \mathcal O} x$ is terminal, and note that by \cite[Proposition 6.2.3.4]{htt}, the $(-1)$-truncation of this map is a colimit of objects which are of the form $x_0 \times \dotsb \times x_n$ for $n \geq 0$, where $x_0, \dotsc, x_n \in \mathcal O$. Thus, if $\mathcal O$ is closed under binary products, all of these objects are in $\mathcal O$, and we are done.
	\end{proof}
\end{lem}

\section{Local Cartesianness} \label{S:LC}

cf. \cite[6.3.5.11]{htt}
\begin{defn} \label{defn:LC adj}
	Say an adjunction $j_! : \mathcal U \rightleftharpoons \mathcal X : j^*$ is locally Cartesian (LC) if for any $f : X \to Y$ in $\mathcal X$, and $g : Z \to j^* Y$ in $\mathcal U$, the pullback $j^* X \times_{j^* Y} Z$ exists, and the square
	\[
		\begin{tikzcd}
			j_!(j^* X \times_{j^* Y} Z) \ar[d] \ar[r] & X \ar[d, "f"]
			\\
			 j_! Z \ar[r] & Y
		\end{tikzcd}
	\]
	is Cartesian.

	We say that $j_!$ is a left Cartesian functor and $j^*$ is a right Cartesian functor.
\end{defn}

\begin{rmk} \label{rmk:LC adj has Cart counit}
	If $j_! \dashv j^*$ is LC, then the counit is Cartesian.
	\begin{proof}
		Take $g : Z \to f^*Y$ to be the identity in the definition of LC adjunctions.
	\end{proof}
\end{rmk}

\begin{lem} \label{lem:pullbacks+Cart-counit→proj-form}
	Consider an adjunction $j_! : \mathcal U \rightleftharpoons \mathcal X : j^*$ where $\mathcal U$ admits pullbacks, and $j_!$ preserves pullbacks. Then the counit is Cartesian if and only if the adjunction is LC.
	\begin{proof}
		Given $f : X \to Y$ in $\mathcal X$, and $g : Z \to j^* Y$ in $\mathcal
		U$, we have
		\[
			\begin{tikzcd}
				j_!(j^* X \times_{j^* Y} Z) \ar[d] \ar[r] & j_! j^* X \ar[d,
				"j_! j^* f"] \ar[r] &  X \ar[d, "f"] \\
				 j_! Z \ar[r, "j_! g"'] & j_! j^* Y \ar[r] &  Y
			\end{tikzcd}
		,\]
		where the left square is Cartesian since $j_!$ preserves pullbacks. If the counit is Cartesian, then the right square is also Cartesian, so the outer square is Cartesian, whence the adjunction is LC. Conversely, if the adjunction is LC, then taking $g = \id_{j^* Y}$, we find that the counit is Cartesian.
	\end{proof}
\end{lem}

\begin{cor} \label{cor:slice proj is LC}
	Slice projections for categories with pullbacks are left Cartesian.
\end{cor}

\begin{prp} \label{prp:LC iff slice localizations}
	If $\mathcal U$ is a category with pullbacks, then a left adjoint functor $j_! : \mathcal U \to \mathcal X$ is left Cartesian if and only if for all $U \in \mathcal U$, the functor $\mathcal U_{/U} \to \mathcal X_{/j_! U}$ is a reflective localization.
	\begin{proof}
		Let $j^*$ be a right adjoint of $j_!$. \cite[Proposition 5.2.5.1]{htt} tells us that in general, the functor $F_U : \mathcal U_{/U} \to \mathcal X_{/j_! U}$ admits a right adjoint $G_U$ given by the composite $\mathcal X_{/j_! U} \to \mathcal U_{/j^* j_! U} \to \mathcal U_{/U}$, which is applying $j^*$ and then base changing along the unit $U \to j^* j_! U$.

		If $j_! \dashv j^*$ is LC, then for any $X \in \mathcal X_{/j_! U}$, the counit of $F_U \dashv G_U$ at $X \to j_! U$ in $\mathcal X_{/j_! U}$ gives a square
		\[
			\begin{tikzcd}
				j_!(j^* X \times_{j^* j_! U} U) \ar[d] \ar[r] & X \ar[d]
				\\
				 j_! U \ar[r, equals] & j_! U
			\end{tikzcd}
		,\]
		which is Cartesian since the adjunction $j_! \dashv j^*$ is LC (so the top arrow, which is given by the counit, is invertible).

		Conversely, for any maps $X \to Y \gets j_! U$ in $\mathcal X$, the counit of $F_U \dashv G_U$ at $X \times_Y j_! U \in \mathcal X_{/j_! U}$ is
		\[
			j_!(j^*(X \times_Y j_! U) \times_{j^* j_! U} U) \to X \times_Y j_! U
		,\]
		so if the counit is invertible, then the composite
		\[
			j_!(j^* X \times_{j^* Y} U) \simeq j_!(j^* X \times_{j^* Y} j^* j_! U \times_{j^* j_! U} U) \simeq j_!(j^*(X \times_Y j_! U) \times_{j^* j_! U} U) \to X \times_Y j_! U
		\]
		is invertible, whence the square
		\[
			\begin{tikzcd}
				j_!(j^* X \times_{j^* Y} Z) \ar[d] \ar[r] & X \ar[d, "f"]
				\\
				j_! Z \ar[r] & Y
			\end{tikzcd}
		\]
		is Cartesian.
	\end{proof}
\end{prp}

\begin{rmk} \label{rmk:conservative iff conservative on slices}
	A functor $F : \mathcal C \to \mathcal D$ is conservative if and only if for all $X \in \mathcal C$, the functor $F_X : \mathcal C_{/X} \to \mathcal D_{/FX}$ is conservative.
	\begin{proof}
		For all $X \in \mathcal C$ we have a commutative square
		\[
			\begin{tikzcd}
				\mathcal C_{/X} \ar[d] \ar[r, "F_X"] & \mathcal D_{/FX} \ar[d] \\
				\mathcal C \ar[r, "F"'] & \mathcal D
			\end{tikzcd}
		\]
		where the columns are conservative. Thus, if $F$ is conservative, then $F_X$ is conservative, and the converse holds since any map in $\mathcal C$ is the image of a map in $\mathcal C_{/X}$ for some $X$ (indeed, a map $\phi : X' \to X$ is the image of $\phi \to \id_X$),
	\end{proof}
\end{rmk}

\begin{prp} \label{prp:characterize slice proj}
	Given an adjunction $j_! : \mathcal U \rightleftharpoons \mathcal X : j^*$ where $\mathcal U$ admits pullbacks, the following are equivalent
	\begin{enumerate}

		\item $j_!$ is conservative, preserves pullbacks, and the counit of $j_! \dashv j^*$ is Cartesian

		\item $j_!$ is conservative and left Cartesian.

		\item For all $U \in \mathcal U$, the functor $\mathcal U_{/U} \to \mathcal X_{/j_! U}$ is an equivalence.

	\end{enumerate}
	Furthermore, $j_!$ is equivalent to a slice projection if and only if these conditions hold and $\mathcal U$ has a terminal object (still assuming $\mathcal U$ admits pullbacks).
	\begin{proof}
		We know that if $j_!$ preserves pullbacks and the counit of $j_! \dashv j^*$ is Cartesian, then $j_!$ is left Cartesian by Lemma \ref{lem:pullbacks+Cart-counit→proj-form}. By Proposition \ref{prp:LC iff slice localizations} and Remark \ref{rmk:conservative iff conservative on slices}, $j_!$ is left Cartesian and conservative if and only if for all $U \in \mathcal U$, the functor $\mathcal U_{/U} \to \mathcal X_{/j_! U}$ is a reflective localization.

		In this case, by Remark \ref{rmk:LC adj has Cart counit}, the counit of $j_! \dashv j^*$ is Cartesian, and since $\mathcal U_{/U} \to \mathcal X_{/j_! U}$ preserves binary products for all $U \in \mathcal U$, we have that $j_!$ also preserves pullbacks.

		Now, note that any slice category has a terminal object, so if $j_!$ is equivalent to a slice projection, then $\mathcal U$ has a terminal object. Furthermore, slice projections are conservative, preserve pullbacks, and have Cartesian counits. For the converse, if $\pt$ is a terminal object of $\mathcal U$, $j_!$ factors as
		\[
			\mathcal U \simeq \mathcal U_{/\pt} \simeq \mathcal X_{/j_! \pt} \to \mathcal X
		,\]
		so it is equivalent to the slice projection $\mathcal X_{/j_! \pt} \to \mathcal X$.
	\end{proof}
\end{prp}

\section{Linear Colimits}

\begin{defn} \label{defn:linear diag}
	Given a monoidal category $\mathcal C$, a linear diagram in $\mathcal C$ is a diagram $\alpha : K \to \mathcal C$ such that for all $X \in \mathcal C$, $\alpha \otimes X$ admits a colimit, and
	\[
		\varinjlim (\alpha \otimes X) \to (\varinjlim \alpha) \otimes X
	\]
	is invertible for all $X$.
\end{defn}

\begin{rmk}
	In the setting of \Cref{defn:linear diag}, if $\mathcal C$ admits $K$-indexed colimits, then for a diagram $\alpha : K \to \mathcal C$ to be linear is equivalent to the condition that for all $X \in \mathcal C$, the square
	\[
		\begin{tikzcd}
			\mathcal C \ar[d, "\otimes X"'] \ar[r] & \Fun(K, \mathcal C) \ar[d, "\otimes \underline{X}"] \\
			\mathcal C \ar[r] & \Fun(K, \mathcal C)
		\end{tikzcd}
	\]
	is horizontally left adjointable \emph{at $\alpha$}, that is, that the left base change transformation is invertible when evaluated at $\alpha$.
\end{rmk}

\begin{rmk} \label{rmk:mon closed linear}
	Any diagram in a monoidal-closed category is linear if and only if it admits a colimit.
\end{rmk}

\begin{rmk} \label{rmk:char of presentable CCC}
	For a category $\mathcal C$ with finite products, consider the following conditions.
	\begin{enumerate}

		\item $\mathcal C$ is Cartesian closed

		\item For any object $X$, the functor $X \times - : \mathcal C \to \mathcal C_{/X}$ has a right adjoint.

		\item For any object $X$, the endofunctor $X \times -$ preserves colimits

		\item For any map $X \to Y$, and (small) diagram $\alpha : K \to \mathcal C$
			\[
				\begin{tikzcd}
					\varinjlim_K (\alpha \times X) \ar[d] \ar[r] & X \ar[d] \\
					\varinjlim_K (\alpha \times Y) \ar[r] & Y
				\end{tikzcd}
			\]
			is Cartesian.

	\end{enumerate}
	The first two are always equivalent (almost immediately by definition). The third is also equivalent to the first two if $\mathcal C$ is presentable by \cite[Corollary 5.5.2.9]{htt}. The last condition is also equivalent to the third: consider the diagram
	\[
		\begin{tikzcd}
			\varinjlim_K (\alpha \times X) \ar[d] \ar[r] & X \ar[d] \\
			\varinjlim_K (\alpha \times Y) \ar[d] \ar[r] & Y \ar[d] \\
			\varinjlim_K \alpha \ar[r] & \pt
		\end{tikzcd}
	.\]
	If the third condition holds, then the bottom and outer squares are Cartesian, so the top square is Cartesian. Conversely, if the fourth condition holds, then the outer square is Cartesian for all $X$, which gives the third condition since the Cartesian gap of the outer square is
	\[
		\varinjlim_K (\alpha \times X) \to (\varinjlim_K \alpha) \times X
	.\]
\end{rmk}

\begin{rmk} \label{rmk:Cart-lin}
	If $\mathcal C$ is a Cartesian monoidal category with small colimits, then a small diagram $\alpha : K \to \mathcal C$ is linear if and only if for every $X \to Y$ in $\mathcal C$, the square
	\[
		\begin{tikzcd}
			\varinjlim (\alpha \times X) \ar[d] \ar[r] & X \ar[d] \\
			\varinjlim (\alpha \times Y) \ar[r] & Y
		\end{tikzcd}
	\]
	is Cartesian.
\end{rmk}

\begin{prp} \label{prp:monoidal crit for cons family}
	Let $\mathcal C$ be a monoidal category, and let $\{f_i^* : \mathcal C \to \mathcal C_i\}_i$ be a collection of functors such that for each $i$, there is a functor $f_i : \mathcal C_i \to \mathcal C$ and object $I_i \in \mathcal C_i$ such that $f_i(I_i) \otimes -$ factors through $f_i^*$. If the monoidal unit of $\mathcal C$ is a linear colimit of a diagram all of whose vertices are of the form $f_i(I_i)$ for some $i$, then
	\[
		\mathcal C \to \prod_i \mathcal C_i
	\]
	is conservative.
	\begin{proof}
		Let $\alpha : K \to \mathcal C$ be a linear diagram whose colimit is the monoidal unit $1$. For any $\phi : X \to Y$, we have a commutative diagram
		\[
			\begin{tikzcd}
				\varinjlim (\alpha \otimes X) \ar[d] \ar[r] & X \ar[d]  \\
				\varinjlim (\alpha \otimes Y) \ar[r] & Y
			\end{tikzcd}
		\]
		with invertible rows.

		For each $p \in K$, there is some $i$ such that $\alpha(p) \simeq f_i(I_i)$,
		\[
			\alpha(p) \otimes \phi \simeq f_i(I_i) \otimes \phi
		,\]
		but since $f_i(I_i) \otimes -$ factors through $f_i^*$, and $f_i^*(\phi)$ is an equivalence, this map is an equivalence.

		Thus, the left arrow in the square is a colimit of equivalences, whence it is an equivalence. Therefore, $\phi$ -- the right arrow in the square -- is an equivalence.
	\end{proof}
\end{prp}

\section{Miscellaneous Categorical Results}
In this section we record some unsorted abstract results about categories.

\begin{lem} \label{lem:compat loc ind}
	Let $\Ind : \mathcal C' \to \mathcal C$ be a fully faithful functor that has a right adjoint $\Res$. Let $W,W'$ be collections of maps in $\mathcal C, \mathcal C'$ such that
	\begin{enumerate}

		\item $W$ is stable under base change along maps sent to equivalences by $\Res$.

		\item If $Q' \in \mathcal C'$, and $f \in W$ is a map to $\Ind Q'$, then $f \simeq \Ind f'$ for some $W'$-equivalence\footnotemark{} $f'$.

	\end{enumerate}
	Then $\Res : \mathcal C' \to \mathcal C$ sends maps in $W$ to $W'$-equivalences.\footnotemark[\thefootnote]{}
	\footnotetext{See \cite[Definition 5.5.4.1]{htt}: a $W'$-equivalence is a map $X \to Y$ such that for any $W'$-local object $Z$, the map of spaces $\mathcal C'(Y,Z) \to \mathcal C'(X,Z)$ is an equivalence. This contains the strongly saturated class generated by $W'$.}
	\begin{proof}
		We need to show that if $f : P \to Q$ is in $W$, and $F \in \mathcal C'$ is $W'$-local, then 
		\[
			\mathcal C'(\Res f, F) : \mathcal C'(\Res Q, F) \to \mathcal C'(\Res P, F)
		\]
		is an equivalence.

		Write $Q' \to Q$ for the counit $\Ind \Res Q \to Q$ of $\Ind \dashv \Res$ at $Q$. Since $\Ind$ is fully faithful, it follows that $\Res(Q' \to Q)$ is an equivalence.

		We may form a Cartesian square
		\[
			\begin{tikzcd}
				P' \ar[d] \ar[r, "f'"] & Q' \ar[d] \\
				P \ar[r, "f"'] & Q
			\end{tikzcd}
		.\]

		Since $\Res$ has a left adjoint, it preserves limits, so $\Res(P' \to P)$ is an equivalence since it is a base change of the equivalence $\Res(Q' \to Q)$. Thus, $\Res(f) \simeq \Res(f')$, so it suffices to show that
		\[
			\mathcal C'(\Res f', F) : \mathcal C'(\Res Q', F) \to \mathcal C'(\Res P', F)
		\]
		is an equivalence.

		Since $W$ is stable by base change along maps sent to equivalences by $\Res$, we have that $f' \in W$. Since $Q'$ is in the image of $\Ind$, it follows that $f' \simeq \Ind(f'')$ for some $W'$-equivalence $f''$. Thus, since $\Ind$ is fully faithful, we have that
		\[
			\Res f' \simeq \Res \Ind f'' \simeq f''
		,\]
		so
		\[
			\mathcal C'(\Res f', F) \simeq \mathcal C'(f'', F)
		,\]
		but $F$ is $W'$-local, and $f''$ is a $W'$-equivalence, so this is an equivalence.
	\end{proof}
\end{lem}

\begin{lem} \label{lem:crit for equiv of psh cats}
	Let $f : \mathcal C \to \mathcal D$ be a functor. Then the induced colimit-preserving functor $f_! : \Psh(\mathcal C) \to \Psh(\mathcal D)$ is an equivalence if and only if $f$ is fully faithful, and every object of $\mathcal D$ is a retract of an object in the essential image of $f$.
	\begin{proof}
		Note that we have a commutative diagram
		\[
			\begin{tikzcd}
				\mathcal C \ar[d, "\yo"'] \ar[r, "f"] & \mathcal D \ar[d, "\yo"] \\
				\Psh(\mathcal C) \ar[r, "f_!"'] & \Psh(\mathcal D)
			\end{tikzcd}
		,\]
		where the vertical arrows are given by Yoneda embeddings. In particular, they are fully faithful by \cite[Propsition 5.1.3.1]{htt}. It follows that
		\begin{enumerate}

			\item $f_! \circ \yo$ is fully faithful if and only if $f$ is fully faithful, and

			\item for every $x \in \mathcal C$, the object $f_!(\yo(x)) \simeq \yo(f(x))$ is completely compact (by \cite[Proposition 5.1.6.8]{htt}).

		\end{enumerate}

		Thus, \cite[Corollary 5.1.6.11]{htt} says that $f_!$ is an equivalence if and only if $f$ is fully faithful, and $\Psh(\mathcal D)$ is generated under colimits by objects of the form $\yo(f(x))$ for $x \in \mathcal C$.

		Therefore, if $f_!$ is an equivalence, for any $y \in \mathcal D$, since $\yo(y)$ is completely compact, $f_!^{-1} \yo(y)$ is a completely compact object of $\Psh(\mathcal C)$, so \cite[Proposition 5.1.6.8]{htt} says there is an object $x \in \mathcal C$ such that $f_!^{-1} \yo(y)$ is a retract of $\yo(x)$. Applying $f_!$, we find that $\yo(y)$ is a retract of $f_! \yo(x) \simeq \yo(f(x))$, so since $\yo : \mathcal D \to \Psh(\mathcal D)$ is fully faithful, we have that $y$ is a retract of $f(x)$.

		Conversely, suppose that every object of $\mathcal D$ is a retract of an object in the essential image of $f$, and that $f$ is fully faithful. For any $y \in \mathcal D$, let $x \in \mathcal C$ such that $y$ is a retract of $f(x)$. So we have maps
		\[
			y \xrightarrow{s} f(x) \xrightarrow{r} y
		\]
		such that $rs \simeq \id$. Now, $f_!$ has a right adjoint $f^*$ given by precomposition by $f$, so we have a commutative diagram
		\[
			\begin{tikzcd}
				f_! f^* \yo(y) \ar[d] \ar[r] & f_! f^* \yo(f(x)) \ar[d] \ar[r] & f_! f^* \yo(y) \ar[d] \\
				\yo(y) \ar[r] & \mathcal \yo(f(x)) \ar[r] & \yo(y)
			\end{tikzcd}
		,\]
		where the vertical arrows are induced by the counit of the adjunction $f_! \dashv f^*$, the left horizontal arrows are induced by $s$, and the right horizontal arrows are induced by $r$. In particular, both horizontal composites are equivalent to identities, and since $f$ is fully faithful, the middle vertical arrow is an equivalence. Thus, we find that $f_! f^* \yo(y) \to \yo(y)$ is a retract of an equivalence, so it is an equivalence, so $\yo(y)$ is in the essential image of $f_!$. It follows that $f_!$ is essentially surjective, and by \cite[Proposition 5.1.6.10]{htt}, it is fully faithful.
	\end{proof}
\end{lem}

\phantomsection
\bibliography{refs}
\bibliographystyle{amsalpha}

\end{document}